\documentclass{article}
\usepackage{graphicx}

\usepackage{setspace}

\usepackage{amsmath}
\usepackage{amssymb}
\usepackage{amsthm}
\usepackage{mathtools}
\usepackage{empheq}
\usepackage{enumitem}
\usepackage{titlesec}
\usepackage{graphicx}
\usepackage{caption}
\usepackage{subcaption}
\usepackage{diagbox}
\usepackage{hyperref}

\usepackage[T1]{fontenc}
\usepackage[utf8]{inputenc}
\usepackage{comment}
\usepackage{indentfirst}
\usepackage[top=1.25in,left=1.25in,right=1.25in]{geometry}

\numberwithin{equation}{section}

\title{\Large{\uppercase{\bf The Facial Common Neighbourhood Graph}}}
\author{\Large{Riccardo W. Maffucci}}
\date{}
\newcommand{\Addresses}{  
		R.W.~Maffucci, \textsc{Dipartimento di Matematica, Universit\`a di Torino\\\indent Via Carlo Alberto 10, Turin 10123, Italy}\par\nopagebreak\vspace{-0.35cm}
		\textit{E-mail address}, R.W.~Maffucci: \href{mailto:riccardowm@hotmail.com}{\texttt{riccardowm@hotmail.com}}
  }

\def\co{\mathcal{O}}
\def\calr{\mathcal{R}}
\def\ct{\mathcal{T}}

\def\cg{\text{con}}
\def\od{\text{odd}}

\def\inte{\text{int}}

\def\fcg{\text{facecon}}

\def\dist{\text{dist}}

\def\om{\omega}

\newtheorem{thm}{Theorem}[section]
\newtheorem{lemma}[thm]{Lemma}
\newtheorem{prop}[thm]{Proposition}
\newtheorem{cor}[thm]{Corollary}
\newtheorem{defin}[thm]{Definition}
\newtheorem{rem}[thm]{Remark}

\begin{document}
\titleformat{\section}
  {\Large\scshape}{\thesection}{1em}{}
\titleformat{\subsection}
  {\large\scshape}{\thesubsection}{1em}{}
\maketitle
\Addresses

\begin{abstract}
Given a polyhedron (planar, $3$-connected graph) $G$, we investigate its common neighbourhood graph $\cg(G)$. For cubic ($3$-regular) polyhedra, we show that the planarity of $\cg(G)$ depends on the number of odd faces of $G$, and on their adjacency. We then prove that for all other polyhedra, $\cg(G)$ is non-planar.

We introduce a novel concept for polyhedra (and more generally, for plane graphs) $G$, namely the `facial common neighbourhood graph' $\fcg(G)$. Its definition takes into account pairs of vertices with a common neighbour on the same face of $G$. It is a spanning subgraph of $\cg(G)$, that coincides with $\cg(G)$ for cubic polyhedra. It also generalises the reverse construction of the radial graph.

As part of our investigation, we also prove a technical result of independent interest: if a maximal planar graph (triangulation of the sphere) has exactly two vertices of odd degree, then they are not adjacent.

We also answer several questions in extremal graph theory. Fixing the number of vertices, we characterise the polyhedra $G$ such that $\cg(G)$ is planar and the number of edges in $\cg(G)$ is minimal/maximal. We address the same problem for $\fcg(G)$, and prove that if it is maximal planar, then $G$ has no face of length greater than $6$. We notably characterise and explicitly construct all polyhedra $G$ of maximal face length $4$ such that $\fcg(G)$ is maximal planar.
\end{abstract}

\tableofcontents

\section{Introduction}
\subsection{The congraph of a polyhedron}
\label{sec:mainres}

In this paper, we deal with finite simple graphs $G=(V=V(G),E=E(G))$ with no multiple edges or loops. The {\bf common neighbourhood graph} of $G$, also referred to as the {\bf congraph}, is defined as
$V(\cg(G)):=V(G)$ and
\begin{equation}
	\label{eq:con}
E(\cg(G)):=\{uv : u,v \text{ have at least one common neighbour in } G\}.
\end{equation}

A planar graph is a graph that may be immersed in the plane, or equivalently on the surface of a sphere, so that edges cross only at vertices. By Kuratowski's Theorem, a graph is planar if and only if it does not contain as subgraph a subdivision of $K_5$ or $K(3,3)$ (equivalently, it does not contain these graphs as minors). A plane graph is a planar graph considered together with a plane immersion. A graph on more than $k$ vertices is said to be $k$-connected if and only if however one removes fewer than $k$ vertices, the resulting graph is connected.

Our focus is the class of {\bf polyhedral graphs} or polyhedra for short, $1$-skeleta of polyhedral solids, corresponding to the class of planar, $3$-connected graphs via the Rademacher-Steinitz Theorem \cite{radste}. In the literature they are also called $3$-polytopal graphs or $3$-polytopes. Many concepts defined for plane graphs in general e.g., the regions, the planar dual, the radial and medial graphs are unambiguous and work especially well in the case of polyhedra, encompassing their geometric properties. This stems from the uniqueness of planar immersion for the class of polyhedra (up to choosing the external region), an observation due to Whitney \cite{whit32}. For this reason, a polyhedron may be considered as a plane graph unambiguously. The dual $G^*$ of a polyhedron $G$ is also a polyhedron, corresponding to the dual solid. In this sense for a polyhedron, the plane immersion, the regions (also referred to as faces), the dual, the radial, the medial constitute intrinsic properties.

Moreover, due to their $3$-connectivity, the class of polyhedra may be recursively generated starting from the pyramids (wheel graphs), via two transformations, as proven by Tutte \cite{tutt61}. Another natural concept for plane graphs, that works particularly well for polyhedra, will be introduced in this paper -- see Section \ref{sec:facecon} to follow.

The terminology `cubic' refers to $3$-regular graphs. Note that cubic polyhedra constitute the dual class of the maximal planar graphs (on at least four vertices), since these are the triangulations of the sphere. Henceforth when we write `maximal planar', we will tacitly rule out $K_1,K_2,K_3$.

A planar graph is $2$-connected if and only if in any planar immersion, all the regions are delimited by cycles (polygons) \cite[Proposition 4.2.5]{dieste}. The lengths of (the cycles delimiting) the regions may differ between two immersions.

A plane, $2$-connected graph is a polyhedron if and only if every pair of regions is either disjoint, or has one common vertex, or has one common edge. In the latter case we will say that the faces are adjacent. This is consistent with the corresponding vertices being adjacent in the dual polyhedron. Equivalently, a plane, $2$-connected graph other than $K_3$ is a polyhedron if and only if whenever two regions have two common vertices, these vertices are adjacent.

A graph $G$ is called bipartite if there is a partition $\{V_1,V_2\}$ of $V(G)$ such that if $uw\in E(G)$, then $u\in V_1$ and $w\in V_2$ or vice versa. Equivalently, we may assign one of two colours to each vertex of $G$ in such a way that every edge has endpoints of opposite colours. A planar, $2$-connected graph is bipartite if and only if in any plane immersion, the regions are delimited by polygons with an even number of edges. We will use the shorthand even/odd regions (and faces) to refer to their lengths being even/odd.

For a non-bipartite graph $G$, we will denote by $\od(G)$ the subgraph of $G$ generated by its odd vertices (vertices of odd degree). We will write $F=F(G)$ for the set of faces of a polyhedron $G$. If $G$ is a non-bipartite polyhedron, we call
\begin{equation}
\label{eq:odfa}
\co=\co(G)\subseteq F(G)
\end{equation}
the set of odd faces of $G$. The vertices of
\begin{equation}
	\label{eq:odd}
\od(G^*)
\end{equation}
correspond naturally to the elements of $\co$.

In this investigation, we have completely classified the polyhedra $G$ such that $\cg(G)$ is planar. We write $\overline{G}$ for the complement graph of $G$.
\begin{thm}
	\label{thm:0}
Let $G$ be a polyhedron such that $\cg(G)$ is planar. Then $G$ is cubic. Moreover, either $G$ is bipartite, in which case $\cg(G)$ has exactly two connected components, both of which are $2$-connected, or $\od(G^*)\in\{\overline{K_2},K_4\}$, in which case $\cg(G)$ is a polyhedron.
\end{thm}
The proof of Theorem \ref{thm:0} will be completed at the end of Section \ref{sec:4}. We will inspect separately the cases of $G$ cubic and $G$ non-cubic. Specifically, for cubic polyhedra we will prove the following.

\begin{thm}
	\label{thm:1}
Let $G$ be a cubic polyhedron. Then the following are equivalent:
\begin{enumerate}[label=(\Alph*)]
\item $\cg(G)$ is planar and connected;\label{(a)}
\item $\cg(G)$ is a polyhedron;\label{(b)}
\item $\od(G^*)\in\{\overline{K_2},K_4\}$.\label{(c)}
\end{enumerate}
\end{thm}

Theorem \ref{thm:1} will be proven in Section \ref{sec:3} (the implication \ref{(b)}$\Rightarrow$\ref{(a)} is clear).

If $G$ is a graph containing a vertex of degree at least $5$, then all of its neighbours are mutually adjacent in $\cg(G)$, so that $\cg(G)$ contains a subgraph isomorphic to $K_5$. Hence for every graph $G$, if $\cg(G)$ is planar, then for the maximal degree $\Delta(G)$ one must have the bound $\Delta(G)\leq 4$. Here for polyhedra we show a stronger result, that represents a milestone in the proof of Theorem \ref{thm:0}: if $\cg(G)$ is planar, then $\Delta(G)=3$.

\begin{thm}
	\label{thm:2}
	Let $G$ be a polyhedron satisfying $\Delta(G)\geq 4$. Then $\cg(G)$ is not planar.
\end{thm}

Theorem \ref{thm:2} will be proven in Section \ref{sec:4}. We remark that, perhaps surprisingly, $\cg(G)$ may be planar also in the case of non-planar $G$, as illustrated in Figure \ref{f:np}. Other examples are given by $\cg(K_{m,n})\simeq K_m\dot\cup K_n$ for $m,n\in\{3,4\}$ (the notation $\dot\cup$ stands for the disjoint union of graphs).
\begin{figure}[ht]
	\centering
	\begin{subfigure}{0.48\textwidth}
		\centering
		\includegraphics[width=5.25cm]{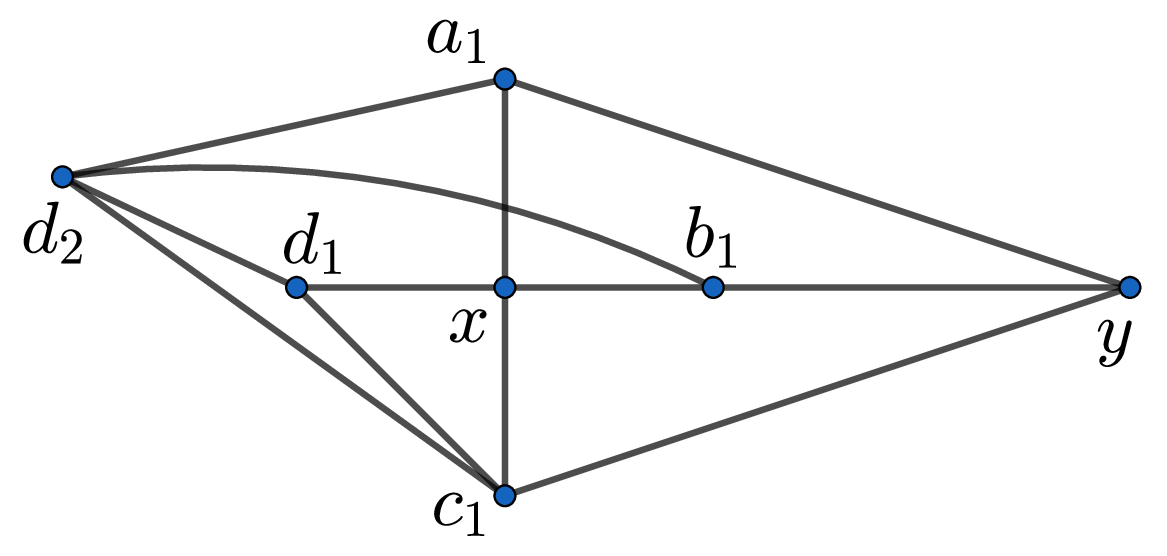}
		\caption{$G$.}
		\label{f:npg}
	\end{subfigure}
	\begin{subfigure}{0.48\textwidth}
		\centering
		\includegraphics[width=3.5cm]{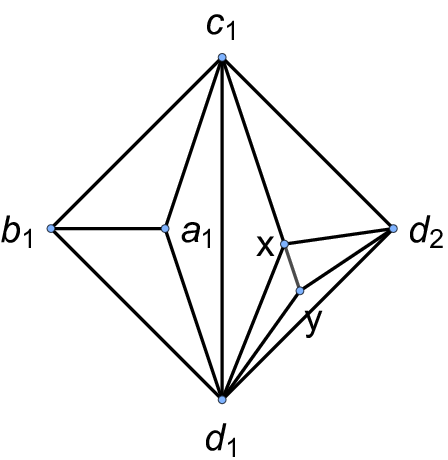}
		\caption{$\cg(G)$.}
		\label{f:npcong}
	\end{subfigure}
	\caption{A non-planar graph $G$ such that $\cg(G)$ is planar (and $2$-connected).}
	\label{f:np}
\end{figure}

\begin{rem}
	Our results on classifying polyhedral graphs $G$ where $\cg(G)$ is planar and their proofs rely on studying the {\bf odd faces} of $G$, including how many there are and {\bf which ones are adjacent} i.e., the vertices and edges of the graph $\od(G^*)$. This is reminiscent of \cite{maffucci2024classification}, where the polyhedral graphs that are Kronecker products of graphs (i.e., direct/tensor products) are characterised and classified, and the results depend on the number of odd faces of the polyhedron, and {\bf how they intersect} (rather than which ones are adjacent).
\end{rem}

\subsection{Definition and remarkable properties of the facial common neighbourhood graph}
\label{sec:facecon}
According to Theorem \ref{thm:0}, $\cg(G)$ is planar only for three specific subclasses of cubic polyhedra $G$. We naturally look to introduce a related concept, a graph to attach specifically to plane graphs $G$, that may be planar for a wider class of polyhedra.
\begin{defin}
Let $G$ be a plane graph. We define its {\bf facial common neighbourhood graph}, or for short {\bf facecongraph} by $V(\fcg(G)):=V(G)$ and
\begin{equation}
	\label{eq:con}
	E(\fcg(G)):=\{uv : u,v \text{ lie along a region of }G\text{ with exactly one vertex between them}\}.
\end{equation}
\end{defin}
An illustration may be found in Figure \ref{f:facecon}. Clearly $\fcg(G)$ is a (spanning) subgraph of $\cg(G)$.
\begin{figure}[ht]
	\centering
	\includegraphics[width=5.cm]{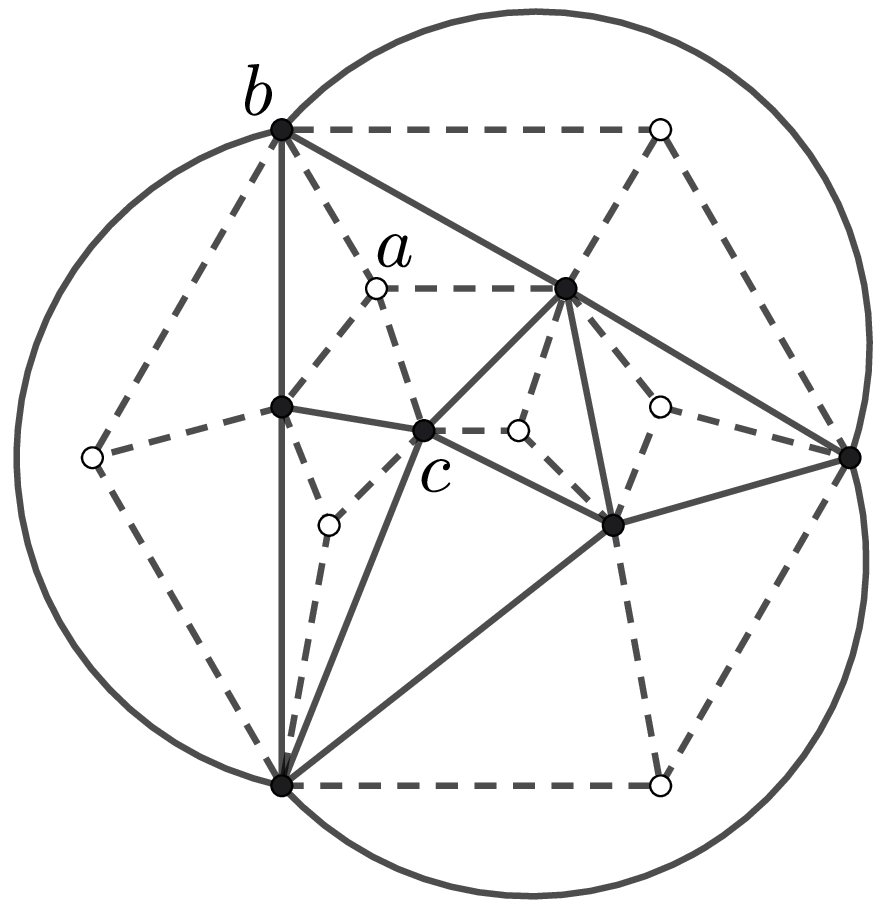}
	\caption{A polyhedron $G$ (dashed) and the subgraph $G_1$ of $\fcg(G)$ spanned by the dark vertices. Note that $b,c$ have the common neighbour $a$ in $G$, but there is no face containing $a,b,c$ in $G$, thus by definition $bc\not\in E(\fcg(G))$, whereas $bc\in E(\cg(G))$.}
	\label{f:facecon}
\end{figure}

In general for a planar graph $G$, the graph $\fcg(G)$ depends on the plane immersion of $G$. Moreover, it may contain multiple edges and/or loops, and one may consider the simplification of $\fcg(G)$ as appropriate. However for a {\bf polyhedron} $G$, the graph $\fcg(G)$ is uniquely determined, since the plane immersion of $G$ is unique, and $\fcg(G)$ is always a simple graph, since in $G$ no two faces may contain two non-adjacent vertices. In what follows, we will focus on the study of $\fcg(G)$ where $G$ is a polyhedron.

We note that if $G$ is a cubic polyhedron, then $\fcg(G)=\cg(G)$. In the case of $G$ being a maximal planar graph, we have instead $\fcg(G)=G$. More generally, if three vertices form a triangular face in $G$, then they form a triangular face in $\fcg(G)$ as well.

As observed in Section \ref{sec:mainres}, if $\cg(G)$ is planar, then $\Delta(G)\leq 4$. As opposed to this, there exist polyhedral graphs $G$ with arbitrarily high $\Delta(G)$ such that $\fcg(G)$ is planar e.g., all maximal planar graphs, and as we shall check in a moment, all bipartite polyhedra.

It is worth underscoring that the common neighbourhood graph $\cg(G)$ is defined for any graph, while its spanning subgraph $\fcg(G)$ is defined for plane graphs, and this definition fits especially well for polyhedra. As a special case, the facecongraph of a cubic polyhedron coincides with its congraph. This is reminiscent of the {\bf line graph} being defined for any graph $G$, and its spanning subgraph the {\bf medial graph} being defined for plane graphs, and fitting especially well for polyhedra \cite{hollowbread2025generation,maffucci2025deza}. Indeed, the line graph connects pairs of vertices corresponding to pairs of incident edges in the original graph, whereas the medial graph takes into account only pairs of incident edges {\em on the same region} of the original plane graph. The cubic polyhedra are a special case, seeing as their medial graph coincides with their line graph, and is a quartic (i.e., $4$-regular) polyhedron. The only polyhedra where the line graph is also polyhedral are the cubic ones \cite{hollowbread2025generation}. More generally, there is a characterisation of which graphs admit a polyhedral line graph \cite{hollowbread2025generation}.

Let $G$ be a polyhedron. The dual of its medial graph is called radial graph $\calr(G)$, or vertex-face graph. It satisfies $V(\calr(G))=V(G)\cup F(G)$ and
\[E(\calr(G))=\{vf : v\in V(G),\ f\in F(G),\text{ and }v\text{ lies on }f\text{ in }G\}.\]
We say that a graph is $k$-edge-connected, $k\geq 1$, if the removal of fewer than $k$ edges results in a connected graph. The class of radial graphs of plane, $2$-connected, $3$-edge-connected graphs are precisely the class of $3$-connected quandrangulations of the sphere \cite[Theorem 5]{brin05} (i.e., duals of $4$-regular polyhedra). Radial graphs of polyhedra are precisely the class of $3$-connected quandrangulations of the sphere with no separating $4$-cycles \cite[Theorem 5]{brin05}.

Let $G$ be a $3$-connected quadrangulation of the sphere. Then the connected components of $\fcg(G)$ are the plane, $2$-connected, $3$-edge-connected graphs $H$ and $H^*$ such that
\begin{equation}
\label{eq:rad}
G\simeq\calr(H)\simeq\calr(H^*).
\end{equation}
Indeed, by definition of radial graph, to recover $H$ and $H^*$ from $\calr(H)$, it suffices to join the opposite vertices of every face of $\calr(H)$. By definition of facecongraph, if $G$ is a $3$-connected quadrangulation, then to obtain $\fcg(G)$ it suffices to join the opposite vertices of every face of $G$.

\begin{rem}
	\label{rem:bip}
	Let $G$ be a plane, $2$-connected, bipartite graph, different from $K_{2,\ell}$, $\ell\geq 2$. Then $\fcg(G)$ is a planar graph with two connected components $G_1,G_2$, both of which are $2$-connected. For $i=1,2$, the plane immersion of $G$ determines a plane immersion of $G_i$ as follows. A region of $G_i$ is delimited by a polygon of length $\ell$, formed either by connecting alternating vertices on a region of $G$ of length $2\ell$, or by connecting in cyclic order the neighbours of a vertex of degree $\ell$ in $G$.
\end{rem}

An illustration for $G,G_1$ is given in Figure \ref{f:facecon}. We have thus checked that the operator
\[G\to\fcg(G)\]
{\bf extends} the operator
\[\calr(H)\to H\ \dot\cup\ H^*\]
from the class of $3$-connected quadrangulations to the class of bipartite polyhedra. Furthermore, we deduce that if $\fcg(G)\simeq\fcg(G')$ where $G,G'$ are $3$-connected quadrangulations, then $G\simeq G'$.

On the other hand, there exist polyhedra $G\not\simeq G'$ such that $\fcg(G)\simeq\fcg(G')$, as illustrated in Figure \ref{f:fcisoa}. Another example is given by the triangular prism and the octahedron, both of which have the octahedron itself as facial common neighbourhood graph.
\begin{figure}[ht]
	\centering
	\begin{subfigure}{0.32\textwidth}
		\centering
		\includegraphics[width=2.5cm]{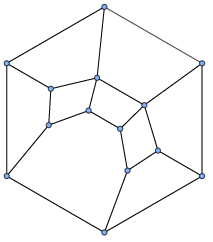}
		\caption{$G$.}
		\label{f:fciso1}
	\end{subfigure}
	\begin{subfigure}{0.32\textwidth}
		\centering
		\includegraphics[width=2.5cm]{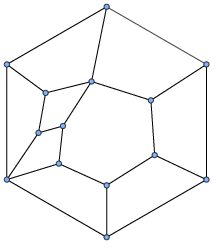}
		\caption{$G'$.}
		\label{f:fciso2}
	\end{subfigure}
	\begin{subfigure}{0.32\textwidth}
		\centering
		\includegraphics[width=3.cm]{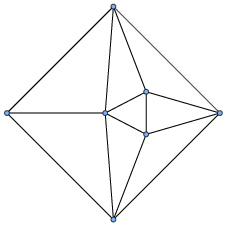}
		\caption{$H$.}
		\label{f:fciso}
	\end{subfigure}
	\caption{$\fcg(G)\simeq\fcg(G')\simeq H\ \dot\cup\ H$, and $G\not\simeq G'$.}
	\label{f:fcisoa}
\end{figure}

The facecongraph thus constitutes an intrinsic property of a polyhedron $G$, similarly to the dual or the radial graph. It may be used to study $G$ and its relevant properties.

\subsection{A result of independent interest on maximal planar graphs}
\label{sec:mpl}

While proving Theorems \ref{thm:0} and \ref{thm:1}, in Section \ref{sec:3} we will establish the following interesting and perhaps surprising result.
\begin{lemma}
	\label{le:cuodd}
	If a maximal planar graph has exactly two odd vertices, then they are not adjacent.
\end{lemma}
Equivalently, if a cubic polyhedron has exactly two odd faces, then they are not adjacent. Still equivalently, using the notation that we have introduced, every cubic polyhedron $G$ satisfies $\od(G^*)\not\simeq K_2$.

\subsection{Results in extremal graph theory}
Thanks to Theorems \ref{thm:0} and \ref{thm:1}, it is straightforward to characterise all $(p,q)$-polyhedral graphs $G$ with $p$ fixed, such that $\cg(G)$ is planar, and $q$ is minimised/maximised. This will be done in Section \ref{sec:ext}.

On the other hand, the characterisation of polyhedra $G$ such that $\fcg(G)$ is planar seems to be more intricate. Nevertheless, in Section \ref{sec:faceconpf}, we will investigate the $(p,q)$-polyhedra $G$ with $p$ fixed, such that $\fcg(G)$ is planar, and $q$ is minimised/maximised. For the minimisation, we have the following.

\begin{prop}
	\label{p:min}
	Let $G$ be a polyhedron of order $p$, such that $\fcg(G)$ is planar. Then
	\[|E(\fcg(G))|\geq 2p-4\]
	with equality if and only if $G$ is a quadrangulation.
\end{prop}
Proposition \ref{p:min} will be proven in Section \ref{sec:faceconpf}.

As for the maximal value of $\fcg(G)$ in the planar case, as mentioned if $G$ is maximal planar, then so is $\fcg(G)$ since in this case $\fcg(G)=G$. However, there are polyhedra $G$ other than the maximal planar graphs such that $\fcg(G)$ is maximal planar. A few examples are sketched in Figure \ref{f:mpl}.
\begin{figure}[ht]
	\centering
	\includegraphics[width=0.2\textwidth]{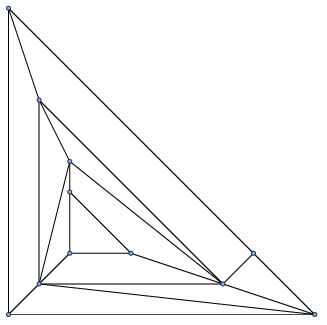}
	\hspace{0.5cm}
	\includegraphics[width=0.125\textwidth]{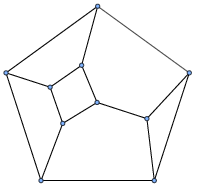}
	\hspace{0.5cm}
	\includegraphics[width=0.11\textwidth]{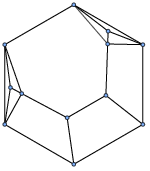}
	\hspace{0.5cm}
	\includegraphics[width=0.11\textwidth]{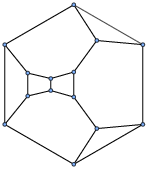}
	\hspace{0.5cm}
	\includegraphics[width=0.11\textwidth]{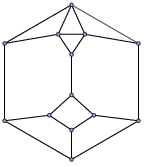}
	\caption{Examples of polyhedra $G$ such that $\fcg(G)$ is maximal planar.}
	\label{f:mpl}
\end{figure}

If $\fcg(G)$ is maximal planar, then we have the following.
\begin{prop}
	\label{p:3456}
	Let $G$ be a polyhedron such that $\fcg(G)$ is maximal planar. Then all faces of $G$ have length at most $6$.
\end{prop}
Proposition \ref{p:3456} will be proven in Section \ref{sec:faceconpf}.

The complete classification of the polyhedra $G$ such that $\fcg(G)$ is maximal planar is probably difficult. Due to Proposition \ref{p:3456}, we know that every face of $G$ is of length at most $6$. Focussing on the class of face length at most $4$, we will establish its full characterisation and construction. This involves the following graph transformations.

\begin{defin}
	\label{def:tr}
	Let $G$ be a polyhedron, and $[a,b,c]$ a triangular face of $G$. We define the transformations $\ct_1$ where we replace $[a,b,c]$ with the graph in Figure \ref{f:tr1}, $\ct_2$ where we replace $[a,b,c]$ with the graph in Figure \ref{f:tr2}, also colouring a face in red as shown, and $\ct_3$ where we replace $[a,b,c]$ with any triangulation.
	\begin{figure}[h!]
		\centering
		\begin{subfigure}{0.48\textwidth}
			\centering		\includegraphics[width=2.75cm]{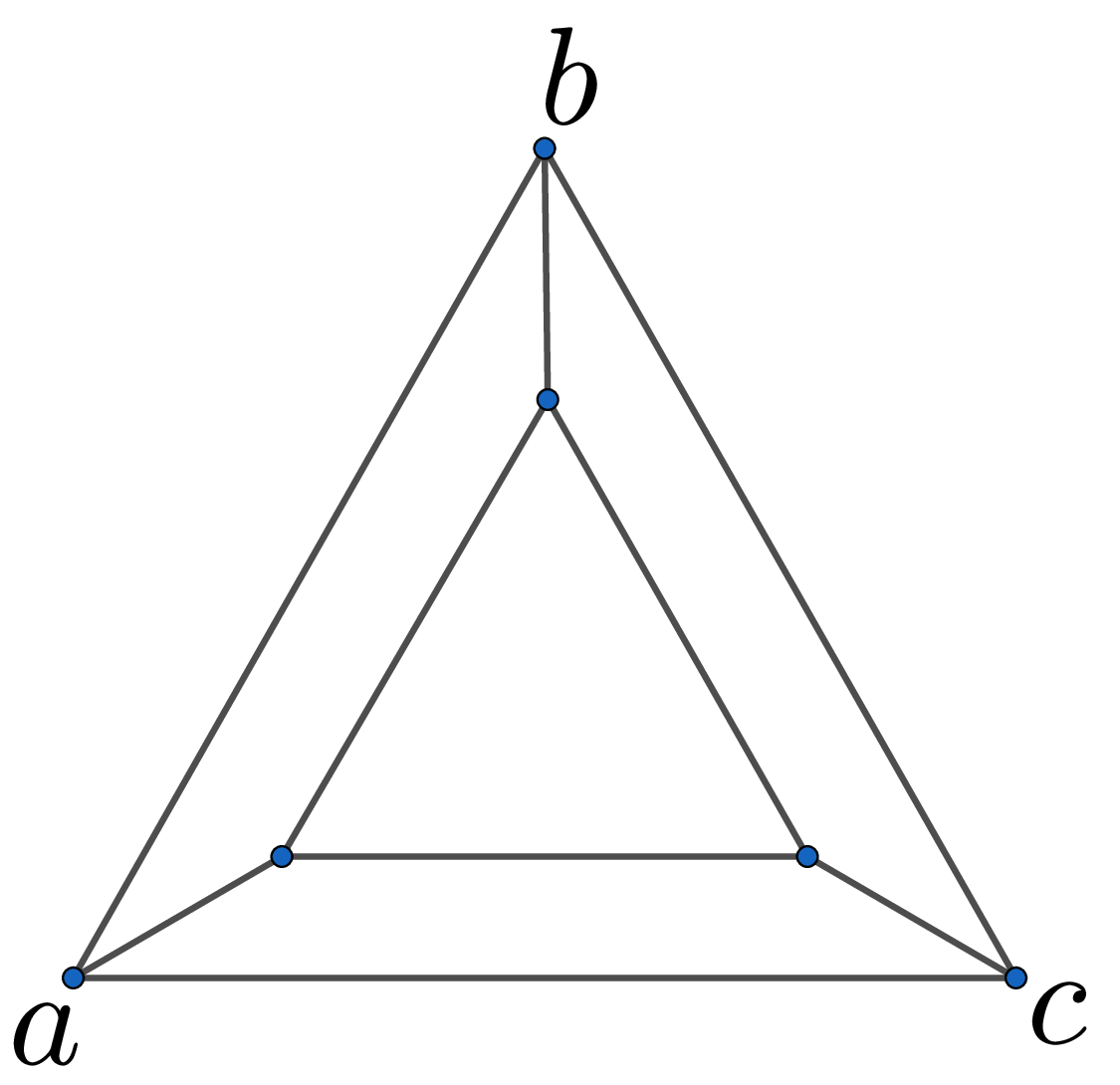}
			\caption{Transformation $\ct_1$.}
			\label{f:tr1}
		\end{subfigure}
		\begin{subfigure}{0.48\textwidth}
			\centering
			\includegraphics[width=2.75cm]{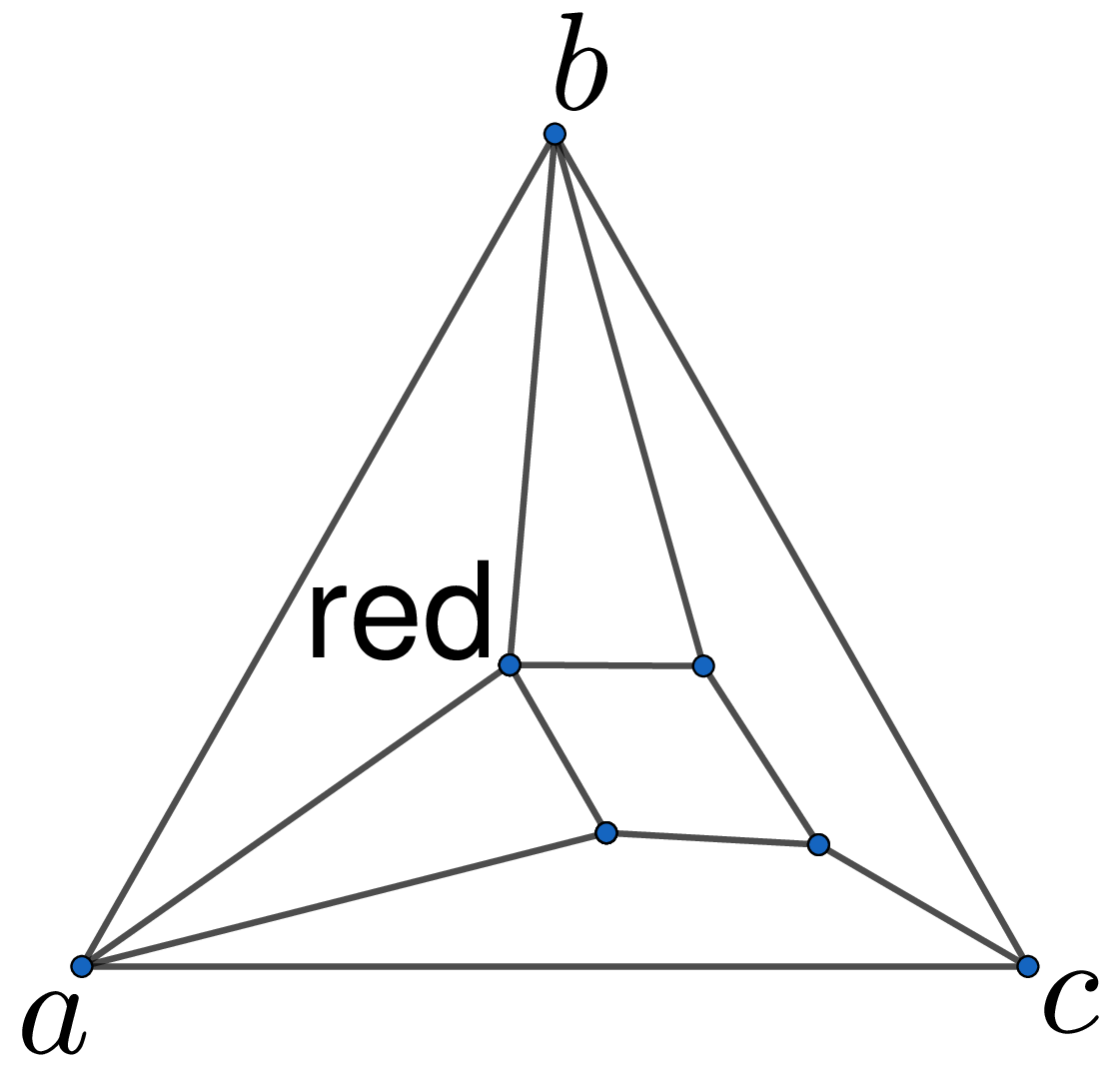}
			\caption{Transformation $\ct_2$.}
			\label{f:tr2}
		\end{subfigure}
		\caption{Definition \ref{def:tr}.}
		\label{f:tr}
	\end{figure}
	
\end{defin}

Remarkably, the class of polyhedra $G$ of maximum face length $4$ and such that $\fcg(G)$ is maximal planar admits the following neat construction.
\begin{thm}
	\label{thm:maxpl}
	Let $G$ be a polyhedron of maximal face length $4$. Then $\fcg(G)$ is maximal planar if and only if $G$ may be constructed from an initial triangle by iterated applications of the transformations $\ct_1$, $\ct_2$, $\ct_3$ to a non-red triangular face $[a,b,c]$ of $G$, where for $\ct_1$ each of the edges $ab,bc,ca$ belongs to a triangular face, and for $\ct_2$ each of $bc,ca$ belongs to a triangular face.
	\\
	Moreover, the number of edges of $G$ is a multiple of $3$.
\end{thm}

Theorem \ref{thm:maxpl} will be proven in Section \ref{sec:faceconpf}.

\paragraph{Related literature.} Common neighbourhood graphs, sometimes called open neighbourhood graphs or two-step graphs, have received recent attention e.g., \cite{alwardi2012common,knor2014wiener,schweitzer2013iterated,sonntag2012iterated,tian2020exploiting}.
\\
In \cite{maffucci2025classification,maffucci2025common,maffucci2025deza}, we investigated the set $A_n(G)$ of quantities of common neighbours for any $n$-tuple of distinct vertices of a graph $G$. We fully classified the planar graphs according to $A_n(G)$ \cite{maffucci2025classification,maffucci2025common}. In the case of $G$ being a regular polyhedron, we focused on the finer structure of $G$ according to $A_n(G)$, generalising the concept of Deza graph in the polyhedral scenario \cite{maffucci2025deza}.
\\
Other problems concerning adjacency in polyhedra have been addressed in \cite{gaspoz2024independence,maffucci2023smallest}.

\paragraph{Notation and terminology.} The set $N(u)=N_G(u)$ collects the neighbours of $u$ in a graph $G$. The length of a path/cycle/circuit is its number of edges.
\\
We denote by $\langle U\rangle$ the subgraph of $G$ induced by the set $U\subset V(G)$. A minor of a graph $G$ is a graph obtained from $G$ by deleting vertices and/or edges, and/or contracting edges.
\\
If $G$ is a polyhedron, we will write $[a_1,a_2,\dots,a_n]$ for an $n$-gonal face. We will use Greek letters for a vertex $\alpha\in V(G^*)$, and the same label for the corresponding face $\alpha\in F(G)$. For $u,v\in V(G)$, we call $\dist(u,v)$ the distance in $G$ between these vertices. For $\alpha,\beta\in F(G)$, we call $\dist(\alpha,\beta)$ the distance in $G^*$ between the vertices corresponding to $\alpha,\beta$.

\paragraph{Plan of the paper.}
In Section \ref{sec:3}, we will prove Theorem \ref{thm:1} and Lemma \ref{le:cuodd}. In Section \ref{sec:4}, we will prove Theorems \ref{thm:2} and \ref{thm:0}. Section \ref{sec:ext} is dedicated to extremal results for polyhedra where the congraph is planar. In Section \ref{sec:faceconpf}, we establish several extremal results for polyhedra where the facecongraph is planar, including Theorem \ref{thm:maxpl} and Propositions \ref{p:min} and \ref{p:3456}.

\section{Cubic polyhedra: proof of Theorem \ref{thm:1}}
\label{sec:3}
\subsection{Proof of Theorem \ref{thm:1}, \ref{(c)}$\Rightarrow$\ref{(b)} and of Lemma \ref{le:cuodd}}

We begin by proving the result on maximal planar graphs stated in Section \ref{sec:mpl}.
\begin{proof}[Proof of Lemma \ref{le:cuodd}]
Recall that the maximal planar graphs are the triangulations of the sphere. By contradiction, let $G$ be a triangulation on $p$ vertices such that $\od(G)\simeq K_2$ and $p$ is minimal. We denote by $u,u_1$ the odd vertices, and by
\[u_1,u_2,\dots,u_n, \quad n\geq 3\]
the neighbours of $u$, in cyclic order around this vertex in the plane immersion of $G$. Note that $n$ is odd. The subgraph of $G$ spanned by $u,u_1,u_2,\dots,u_n$ is an $n$-gonal pyramid of apex $u$.

We claim that
\begin{equation}
	\label{eq:claim}
u_iu_j\not\in E(G), \quad\forall\ 2\leq i,j\leq n,\ |j-i|\geq 2.
\end{equation}
By contradiction, we may find $2\leq i,j\leq n$, $|j-i|\geq 2$ such that $u_iu_j\in E(G)$, as in Figure \ref{f:inout}. The triangle
\[C: u,u_i,u_j\]
is a separating triangle in $G$. It delimits two maximal planar subgraphs of $G$, that we will call $G_\text{in}$ and $G_\text{out}$, spanned respectively by the vertices on and inside of $C$, and by those on and outside of $C$. W.l.o.g., $u_1\in V(G_\text{in})$. We note that since $|j-i|\geq 2$, then $G_\text{in}$ and $G_\text{out}$ have strictly fewer vertices than $G$. As $\deg_G(u)$ is odd, and moreover by construction we have $\deg_{G_\text{in}}(u)+\deg_{G_\text{out}}(u)=\deg_G(u)+2$, then
\[\deg_{G_\text{in}}(u) \qquad\text{ and }\qquad \deg_{G_\text{out}}(u)\]
have different parity. Suppose for the moment that $\deg_{G_\text{out}}(u)$ is odd. Then $V(\od(G_{out}))$ contains $u$, possibly $u_i$ and $u_j$, and no other vertices. By the handshaking lemma for $G_\text{out}$, one of
\[\deg_{G_\text{out}}(u_i)\qquad \text{ and }\qquad \deg_{G_\text{out}}(u_j)\]
is even and the other odd. Thus $\od(G_{out})\simeq K_2$, contradicting the minimality of $p$.
\begin{figure}[ht]
\centering
\begin{subfigure}{0.48\textwidth}
\centering
\includegraphics[width=3.5cm]{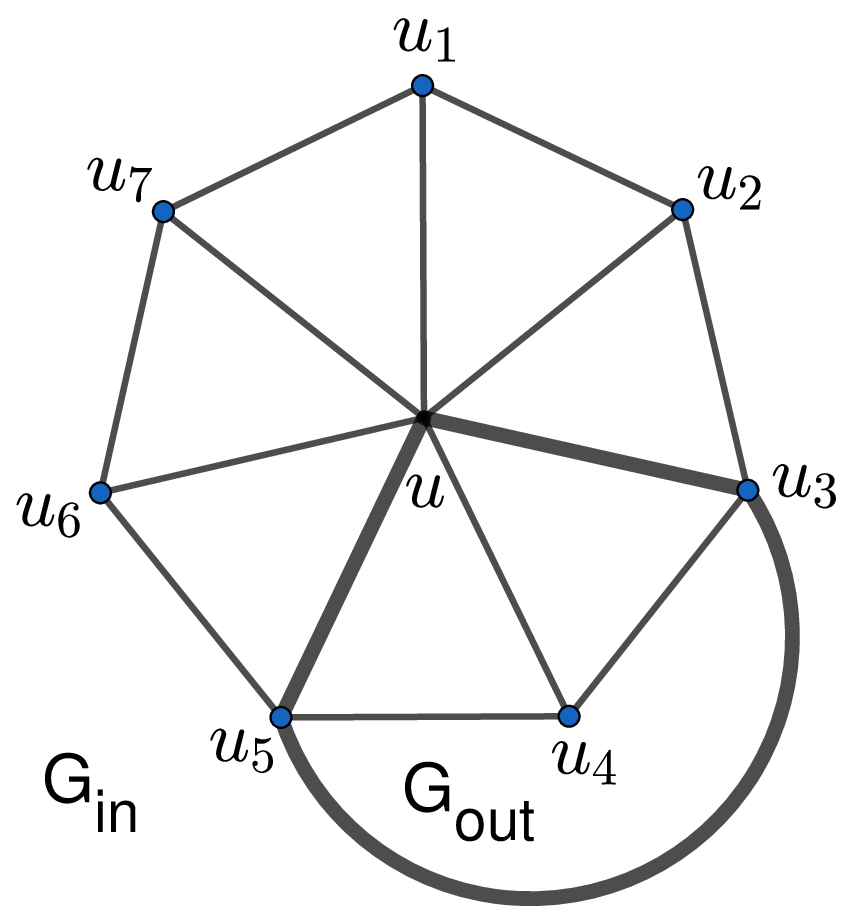}
\caption{$u_iu_j\in E(G)$, $i=3$ and $j=5$.}
\label{f:inout}
\end{subfigure}
\begin{subfigure}{0.48\textwidth}
\centering
\includegraphics[width=3.5cm]{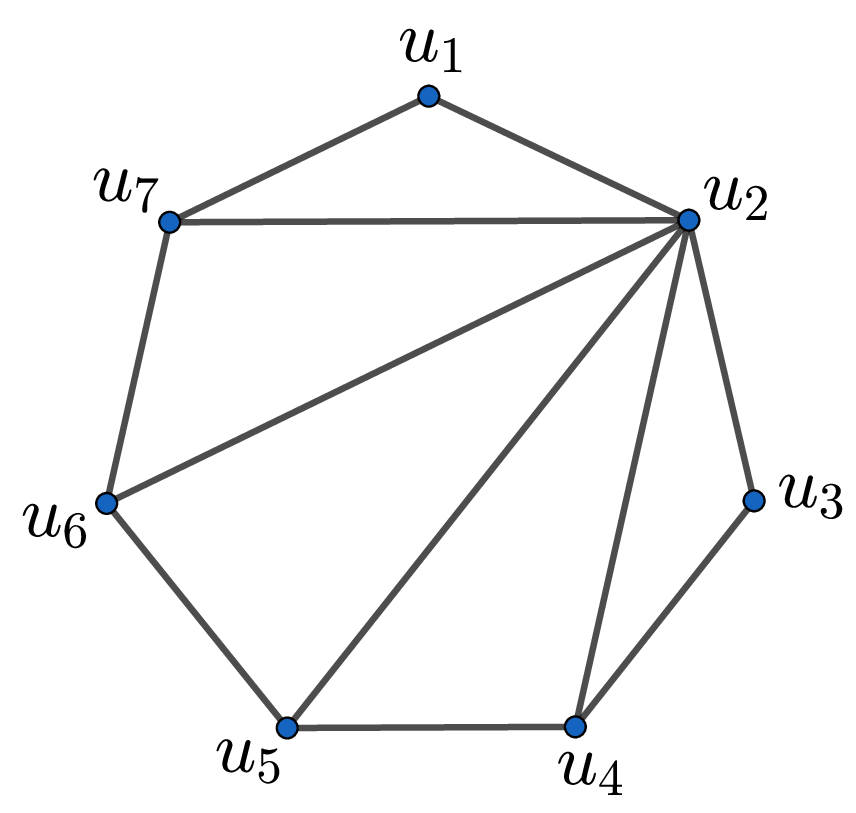}
\caption{Construction of $G'$.}
\label{f:delu}
\end{subfigure}
\caption{Lemma \ref{le:cuodd}. Only a subgraph of $G$ is shown.}
\label{f:odd}
\end{figure}

The remaining possibility is that $\deg_{G_\text{out}}(u)$ is even, thus $\deg_{G_\text{in}}(u)$ is odd. Since $\deg_{G_\text{out}}(u)$ is even, by the handshaking lemma for $G_\text{out}$, $\deg_{G_\text{out}}(u_i)$ and $\deg_{G_\text{out}}(u_j)$ have the same parity. By minimality of $p$, $\od(G_{out})\not\simeq K_2$, thus $\deg_{G_\text{out}}(u_i)$ and $\deg_{G_\text{out}}(u_j)$ are both even. Since $u_i,u_j$ have even degree in $G$ and $G_{out}$, then they have even degree also in $G_{in}$. It follows that $V(\od(G_{in}))=\{u,u_1\}$, again contradicting the minimality of $p$. Hence indeed \eqref{eq:claim} holds.

We now construct the graph of order $p-1$
\[G':=G-u+u_2u_4+u_2u_5+\dots+u_2u_n,\]
as in Figure \ref{f:delu} (if $n=3$, then simply $G'=G-u$). By \eqref{eq:claim}, $G'$ is a simple graph, and in fact it is a triangulation of the sphere. The degrees of $u_4,u_5,\dots,u_n$ are the same in $G$ and $G'$. The degrees of $u_1$ and $u_3$ have each been reduced by $1$, so that $u_1$ is even and $u_3$ is odd in $G'$. As for $u_2$, we have
\[\deg_{G'}(u_2)=\deg_{G}(u_2)-1+(n-3).\]
As $n$ is odd, the parity of $u_2$ has changed. Hence $V(\od(G'))=\{u_2,u_3\}$, and since $u_2u_3\in E(G')$, we deduce that $\od(G')\simeq K_2$, contradicting the minimality of $p$. The proof of this lemma is complete.
\end{proof}

The next step towards proving the theorems stated in Section \ref{sec:mainres} is the following (cf.\ Remark \ref{rem:bip}).
\begin{prop}
	\label{p:p2cb}
Let $G\not\simeq K(2,2), K(2,3)$ be a planar, $2$-connected, bipartite graph satisfying $\Delta(G)\leq 3$. Then $\cg(G)$ has exactly two connected components, and each of them is a planar, $2$-connected graph.
\end{prop}
\begin{proof}
We consider a plane immersion of $G$. By $2$-connectivity, each region is bounded by a cycle. Since $G$ is bipartite, these are all even cycles. Let $V_1$ be either of the two sets in the vertex bipartition of $G$. Since $\Delta(G)\leq 3$, the edges incident to elements of $V_1$ in $\cg(G)$ are all obtained by drawing diagonals between vertices of $V_1$ at distance $2$ along the cycles bounding the regions of $G$. Thus indeed $\cg(G)$ has two connected components, and both are planar graphs.

It remains to show the $2$-connectivity of the components of $\cg(G)$. Let $u,v\in V_1$. By $2$-connectivity of $G$, there exist in $G$ two internally disjoint $uv$-paths
\[u,a_1,a_2,\dots,a_m,v,\qquad m\geq 1\]
and
\[u,b_1,b_2,\dots,b_n,v,\qquad n\geq 1.\]
Since $G$ is bipartite, $m,n$ are odd. Hence in $\cg(G)$ we may find two internally disjoint $uv$-paths
\[u,a_2,a_4,\dots,a_{m-1},v\]
and
\[u,b_2,b_4,\dots,b_{n-1},v.\]
We have proved $2$-connectivity of both components of $\cg(G)$, unless $m=n=1$ i.e., unless $u,a_1,v,b_1$ is a square in $G$. In this case, if $a_1$ (or $b_1$) has degree $3$ in $G$ and $c$ is its third neighbour, then $u,c,v$ is a path in $\cg(G)$. Else if both $a_1,b_1$ are only adjacent to $u,v$ in $G$, then we look at other $uv$-paths in $G$. One of them has to be of length greater than $2$ in $G$, since $G\not\simeq K(2,2),K(2,3)$.
\end{proof}

We are ready to prove one of the facts stated in Theorem \ref{thm:1}.
\begin{prop}\label{p:K2}
Let $G$ be a cubic polyhedron satisfying $\od(G^*)=\overline{K_2}$. Then $\cg(G)$ is a polyhedron.
\end{prop}
\begin{proof}
Let $\alpha,\beta$ be the odd faces of $G$. They correspond to vertices of $G^*$ that we will also label $\alpha,\beta$. In $G^*$, we take an $\alpha\beta$-path of minimal length
\begin{equation}
	\label{eq:path}
\alpha=:\varphi_1,\varphi_2,\varphi_3,\dots,\varphi_{n},\beta=:\varphi_{n+1},\qquad n\geq 2
\end{equation}
($n\geq 2$ since $\alpha,\beta$ cannot be adjacent, as $\od(G^*)=\overline{K_2}$). In the plane immersion of $G$, \eqref{eq:path} correspond to faces (that we will also label in this way) such that $\varphi_i,\varphi_{i+1}$ are adjacent for every $1\leq i\leq n$. We denote by $u_iv_i$ the common edge of $\varphi_i,\varphi_{i+1}$, assigning these vertex labels so that $u_{i-1},u_{i},v_{i},v_{i-1}$ appear in this order around the face $\varphi_i$, for $2\leq i\leq n$. An illustration may be found in Figure \ref{f:k2a}. Note that by minimality of the length of \eqref{eq:path}, for every $1\leq j,k\leq n+1$ such that $|k-j|\geq 2$, the faces $\varphi_j$ and $\varphi_{k}$ cannot be adjacent, thus as $G$ is cubic, they are disjoint, so that   $u_{1},v_{1},u_2,v_2,\dots,u_{n},v_{n}$ are pairwise distinct.
\begin{figure}[ht]
	\centering
	\begin{subfigure}{0.48\textwidth}
		\centering
		\includegraphics[width=4.cm]{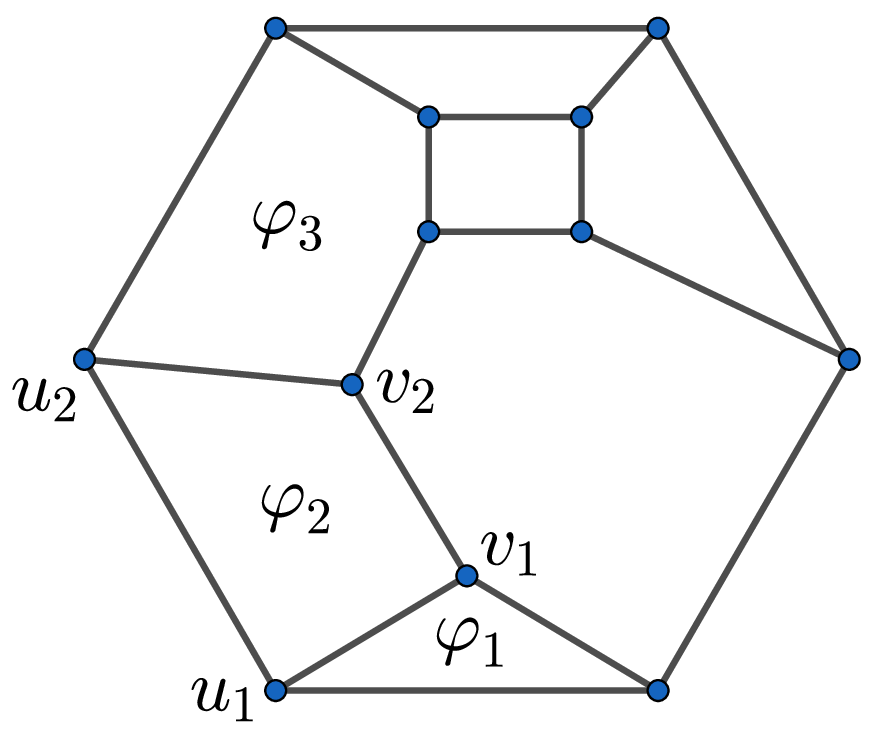}
		\caption{$G$.}
		\label{f:k2a}
	\end{subfigure}
	\begin{subfigure}{0.48\textwidth}
		\centering
		\includegraphics[width=5.cm]{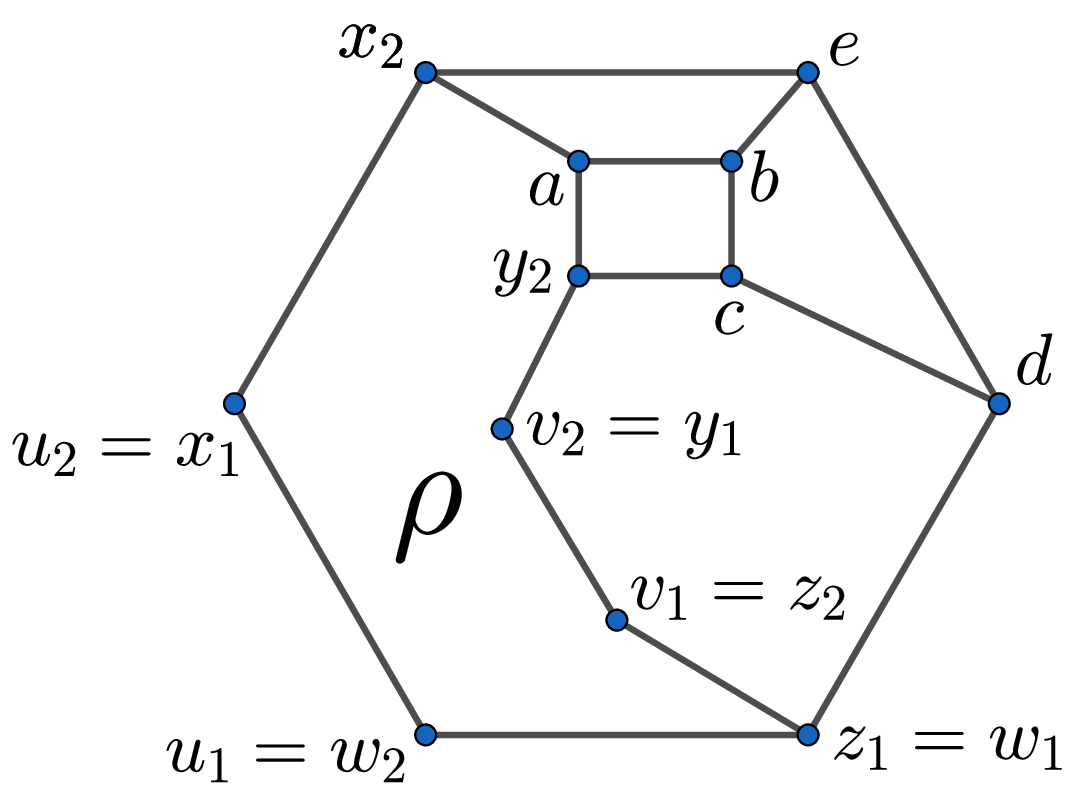}
		\caption{$G'$.}
		\label{f:k2b}
	\end{subfigure}
	\begin{subfigure}{0.48\textwidth}
		\centering
		\includegraphics[width=7.5cm]{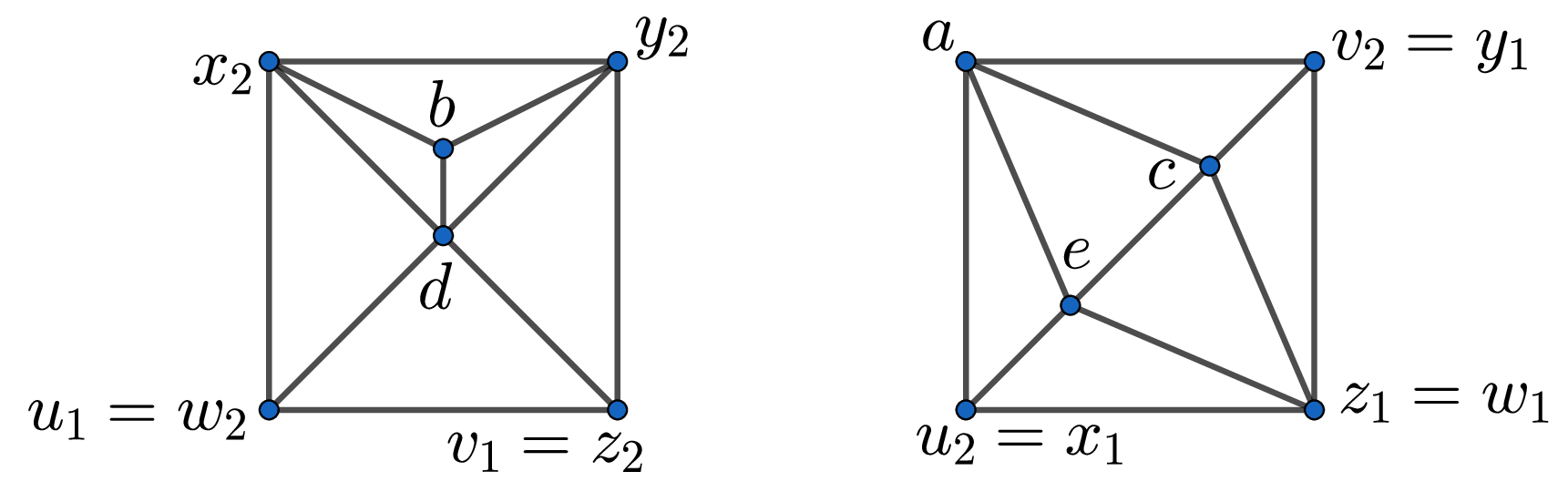}
		\caption{$\cg(G')=G_1\ \dot\cup\ G_2$.}
		\label{f:k2c}
	\end{subfigure}
	\begin{subfigure}{0.48\textwidth}
		\centering
		\includegraphics[width=7.5cm]{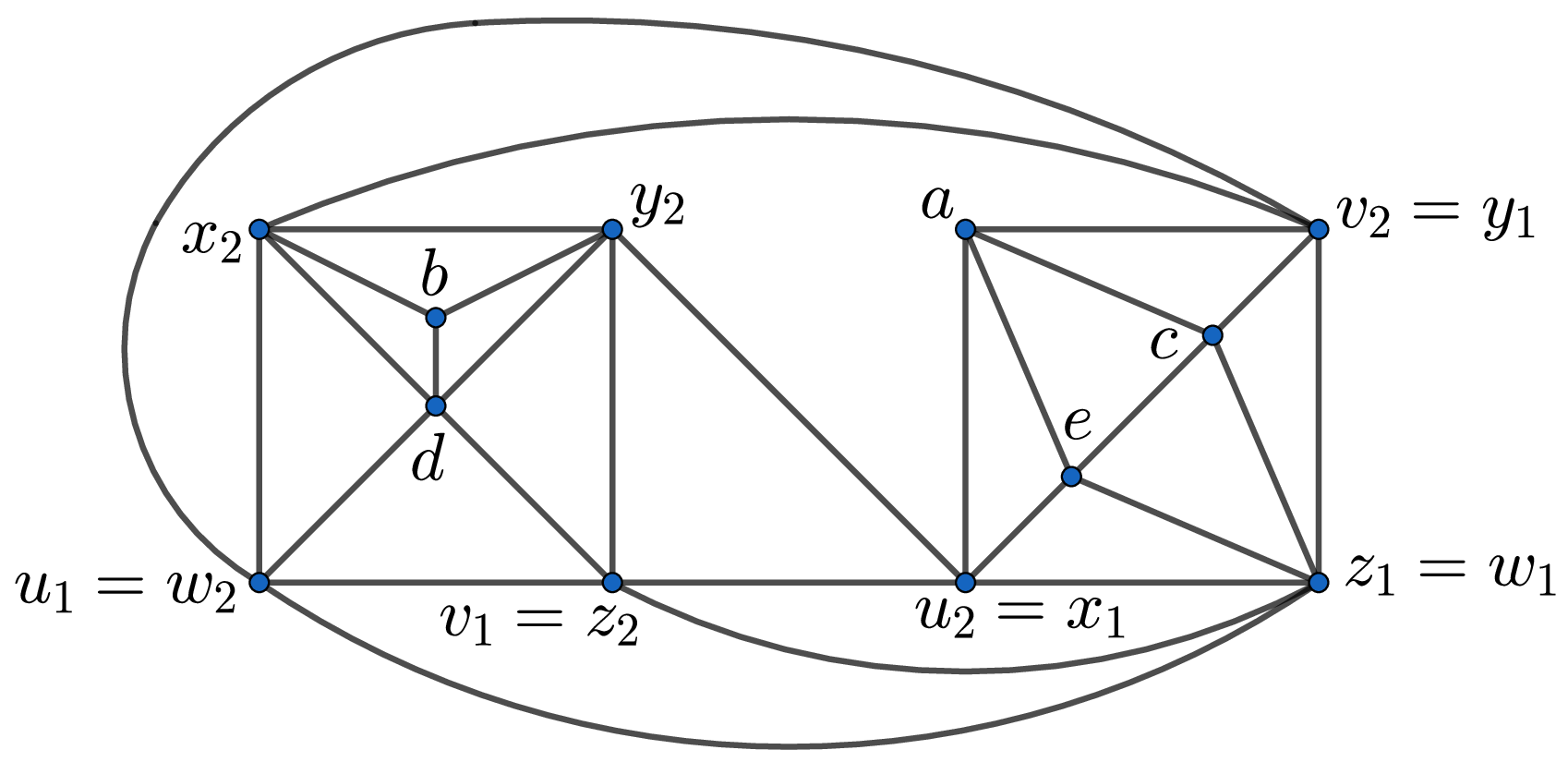}
		\caption{$\cg(G)$.}
		\label{f:k2d}
	\end{subfigure}
	\caption{Illustration of Proposition \ref{p:K2}.}
	\label{f:k2abcd}
\end{figure}

Let
\[G':=G-u_1v_1-u_2v_2-\dots-u_nv_n\]
(Figure \ref{f:k2b}). Clearly $G'$ is a plane graph. By construction, its regions correspond to faces of $G$, save for a region that we will denote by $\rho$, that contains the vertices
\begin{equation}
	\label{eq:uv}
u_1,u_2,\dots,u_n,v_n,v_{n-1},\dots,v_1
\end{equation}
in this order (clockwise, say). If we show  that $\rho$ is bounded by a cycle, it will follow that every region of $G'$ is bounded by a cycle, thus $G'$ is $2$-connected. To construct $G'$ starting from $G$, we delete $u_1v_1,u_2v_2,\dots,u_nv_n$ in turn. At each step, $u_iv_i$ is the common edge of two regions
\[\varphi_i=[u_i,v_i,a_1,a_2,\dots,a_A], \qquad A\geq 1\]
and
\[\varphi_{i+1}=[u_i,v_i,b_1,b_2,\dots,b_B], \qquad B\geq 1,\]
with $A,B$ of opposite parity. Once we delete $u_iv_i$, the resulting new region of the obtained graph is bounded by the cycle
\[a_1,a_2,\dots,a_A,b_B,b_{B-1},\dots,b_1.\]
To see this, we are using the fact that $\varphi_{i+1}$ is disjoint from each of $\varphi_1,\dots,\varphi_{i-1}$, for every $2\leq i\leq n$. Therefore, $\rho$ is indeed bounded by a cycle, thus $G'$ is $2$-connected.

Further, as $\alpha,\beta$ are the only odd faces of $G$, then in $G'$ each region is bounded by a cycle of even length, thus $G'$ is bipartite. Since $\alpha,\beta$ are not adjacent in $G$, then $\rho$ has length at least $6$. In particular, $G'\not\simeq K(2,2),K(2,3)$. We also note that $\Delta(G')\leq\Delta(G)=3$.

We are in a position to apply Proposition \ref{p:p2cb}, hence $\cg(G')$ has exactly two connected components $G_1,G_2$, each of which is a planar, $2$-connected graph. For $j=1,2$, the graph $G_j$ contains a region $\rho_j$ consisting of alternating elements of $\rho$ taken in order.

We sketch $G_1,G_2$ in the plane with $\rho_1,\rho_2$ external, as in Figure \ref{f:k2c}. Since $\varphi_2,\varphi_3,\dots,\varphi_{n}$ are even faces in $G$, then in $\cg(G')$ we have $u_i\in V(G_1)$ if and only if $v_i\in V(G_1)$, for every $1\leq i\leq n$. To obtain $\cg(G)$, it now suffices to add the edges
\begin{equation}
	\label{eq:wxyz}
u_iy_i,\ u_iz_i,\ v_iw_i,\ v_ix_i, \qquad 1\leq i\leq n,
\end{equation}
where $w_i,x_i$ are the neighbours of $u_i$ along $\rho$, and $y_i,z_i$ are the neighbours of $v_i$ along $\rho$ ($x_1$ the clockwise neighbour of $u_i$ and $z_1$ the clockwise neighbour of $v_i$, say). The reader may refer to Figure \ref{f:k2d}.
\begin{figure}
	\centering
	
\end{figure}

Let us show that $\cg(G)$ is a planar graph. W.l.o.g., $u_1,v_1\in V(G_1)$. Then $w_1,x_1,y_1,z_1\in V(G_2)$. We draw the edges \eqref{eq:wxyz} for $i=1$. This can be done in a planar way, since $w_1,x_1,y_1,z_1$ appear in this order around the boundary of $\rho_2$ (possibly $w_1=z_1$, if $\alpha$ is a triangular face of $G$). Next, we draw in turn the edges \eqref{eq:wxyz} for $2\leq i\leq n$. Either \[u_{i-1},u_i,v_i,v_{i-1}\]
appear in this order around $\rho_1$ and
\[w_{i-1},x_{i-1},w_{i},x_{i},y_{i},z_{i},y_{i-1},z_{i-1}\]
appear in this order around $\rho_2$ (possibly $x_{i-1}=w_i$, and possibly $z_{i}=y_{i-1}$), or 
\[u_{i-1},w_{i},x_{i},y_{i},z_{i},v_{i-1}\]
appear in this order around $\rho_1$ and
\[w_{i-1},x_{i-1},u_i,v_i,y_{i-1},z_{i-1}\]
appear in this order around $\rho_2$ (possibly $u_{i-1}=w_{i}$ and $x_{i-1}=u_i$ if $u_{i-1},u_i$ are consecutive along $\rho$, and possibly $y_{i-1}=v_{i}$ and $v_{i-1}=z_{i}$ if $v_{i-1},v_i$ are consecutive along $\rho$), as in Figure \ref{f:k2d}. In either case, it is possible to add all of the edges \eqref{eq:wxyz} to $\cg(G')$ successively in a planar way, thus $\cg(G)$ is indeed a planar graph.

Furthermore, since the two connected components $G_1,G_2$ of $\cg(G')$ are $2$-connected and $\cg(G)$ is obtained from $\cg(G')$ by adding the distinct edges
\[u_1y_1,\ u_1z_1,\ v_1w_1,\ v_1x_1,\ u_2y_2,\ v_2x_2,\]
it follows that $\cg(G)$ is $3$-connected.
\end{proof}

We now prove another one of the facts stated in Theorem \ref{thm:1}.
\begin{prop}\label{p:K4}
	Let $G$ be a cubic polyhedron satisfying $\od(G^*)=K_4$. Then $\cg(G)$ is a polyhedron.
\end{prop}
\begin{proof}
Let $\alpha,\beta,\gamma,\delta$ be the four odd faces of $G$, $ab$ the common edge of $\alpha,\beta$, and $cd$ the common edge of $\gamma,\delta$, as in Figures \ref{f:k4ga} (general construction) and \ref{f:k4a} (specific example). Then
\[G':=G-ab-cd\]
is a planar, $2$-connected, bipartite graph. Now $\cg(K_4)\simeq K_4$, and henceforth we take $G\not\simeq K_4$, thus $G'$ is different from $K(2,2),K(2,3)$. By Proposition \ref{p:p2cb}, $\cg(G')$ has exactly two connected components $G_1,G_2$, each of which is a planar, $2$-connected graph.
\begin{figure}[h!]
	\centering
	\begin{subfigure}{0.48\textwidth}
		\centering
		\includegraphics[width=4.cm]{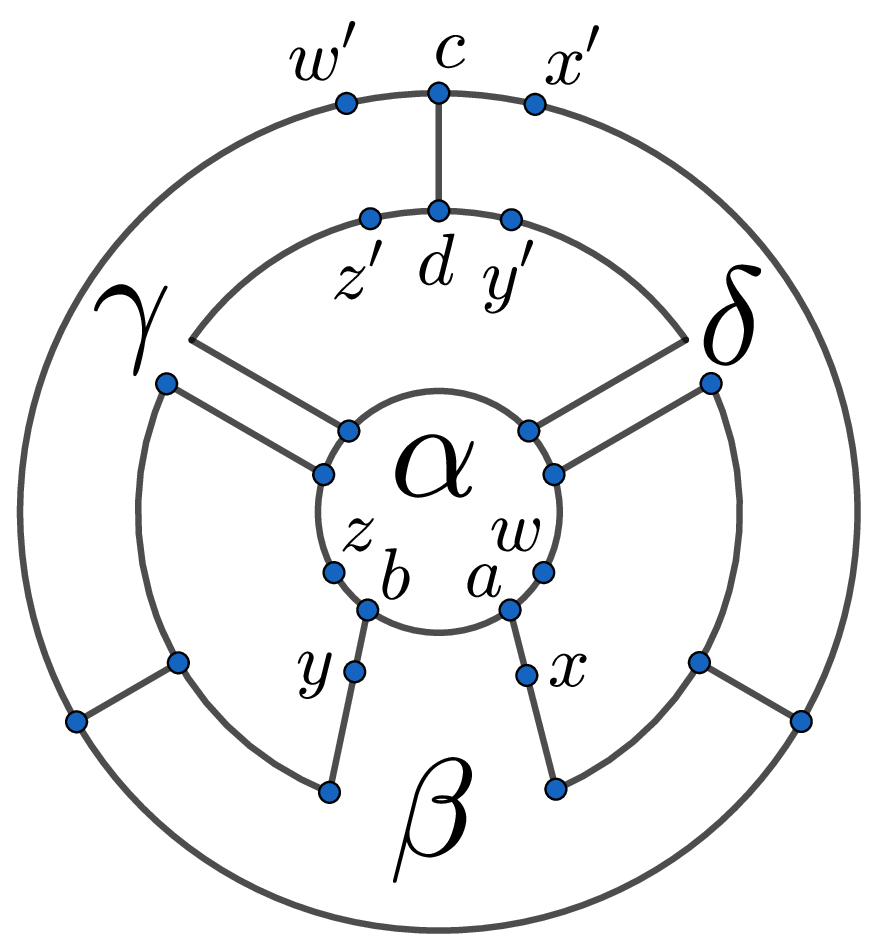}
		\caption{$G$.}
		\label{f:k4ga}
	\end{subfigure}
	\begin{subfigure}{0.48\textwidth}
		\centering
		\includegraphics[width=5.cm]{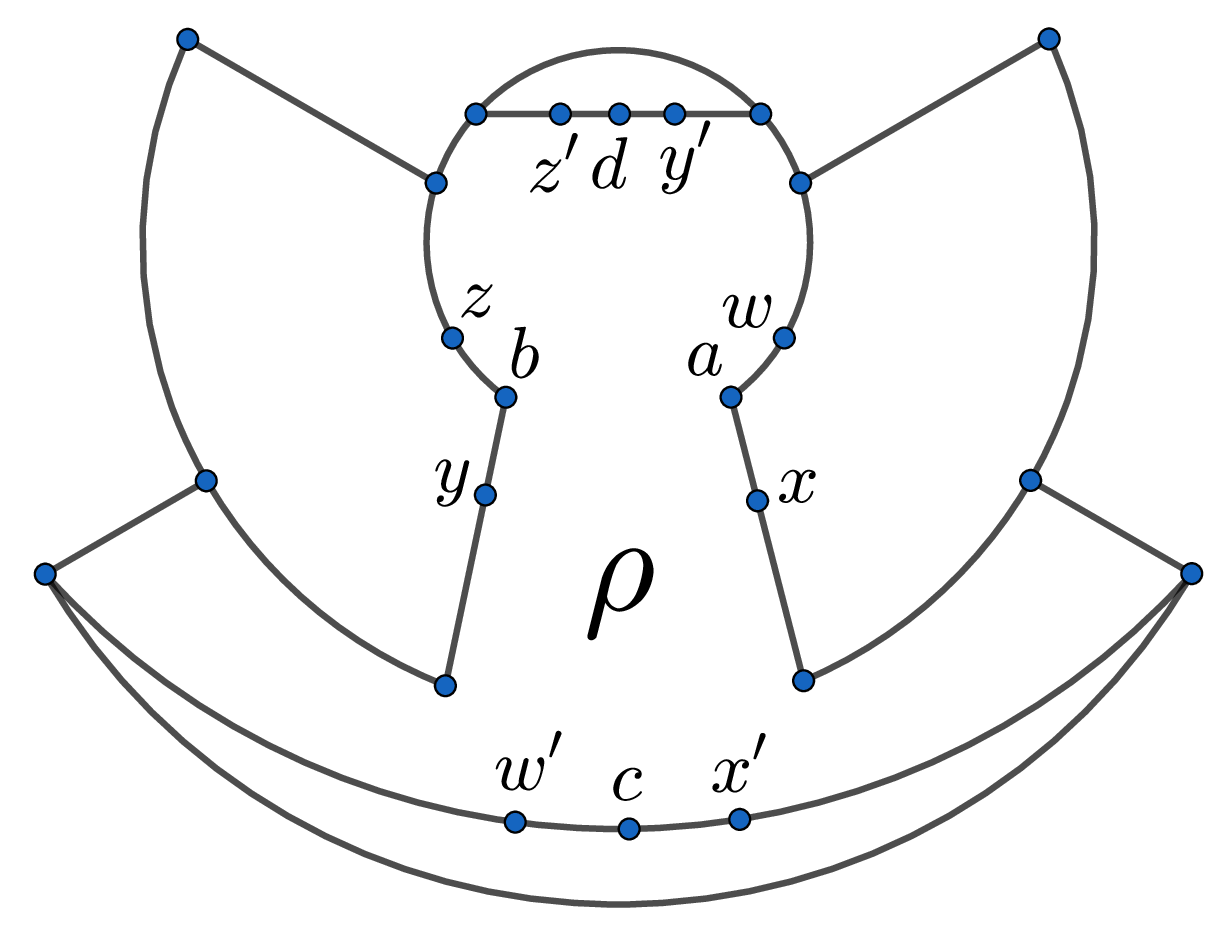}
		\caption{Immersion of $G'$ with $a,c,b,d$ consecutive along the boundary of a region $\rho$.}
		\label{f:k4gb}
	\end{subfigure}
	\caption{Construction for Proposition \ref{p:K4}.}
	\label{f:k4g}
\end{figure}
\begin{figure}[h!]
	\centering
	\begin{subfigure}{0.48\textwidth}
		\centering
		\includegraphics[width=4.cm]{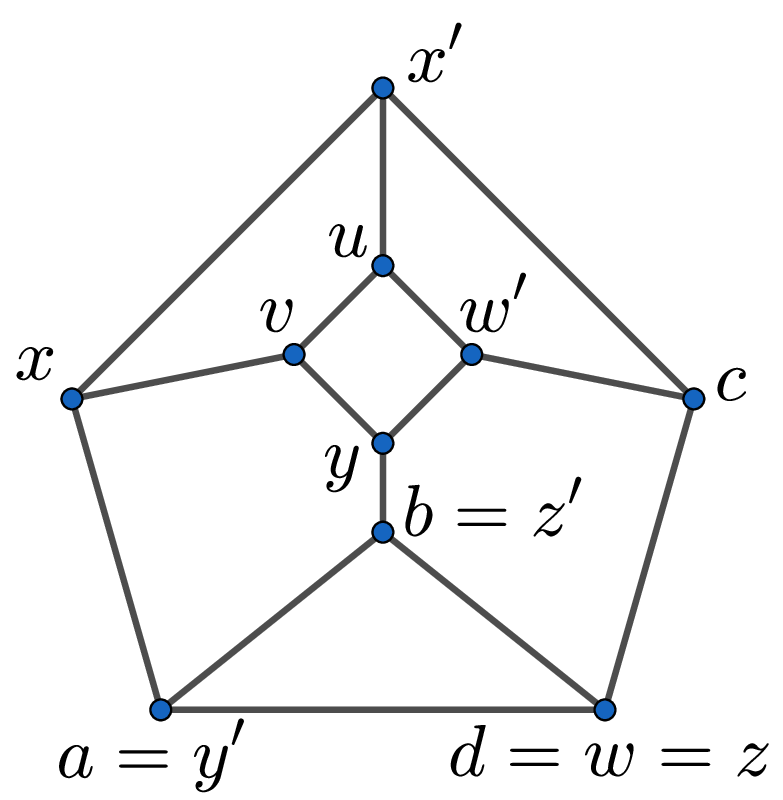}
		\caption{$G$.}
		\label{f:k4a}
	\end{subfigure}
	\begin{subfigure}{0.48\textwidth}
		\centering
		\includegraphics[width=4.cm]{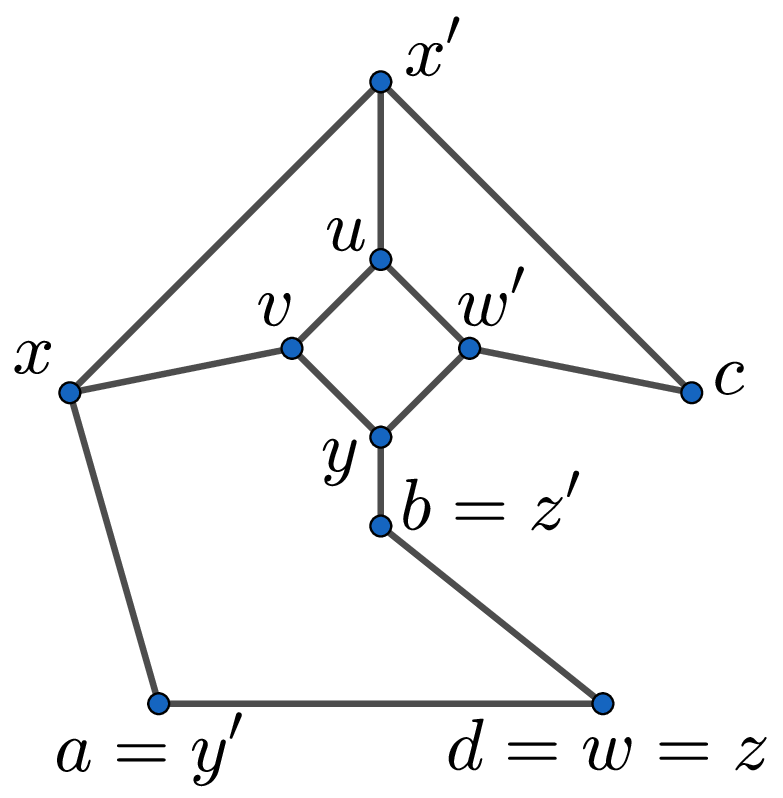}
		\caption{$G'$.}
		\label{f:k4b}
	\end{subfigure}
	\begin{subfigure}{0.48\textwidth}
		\centering
		\includegraphics[width=5.5cm]{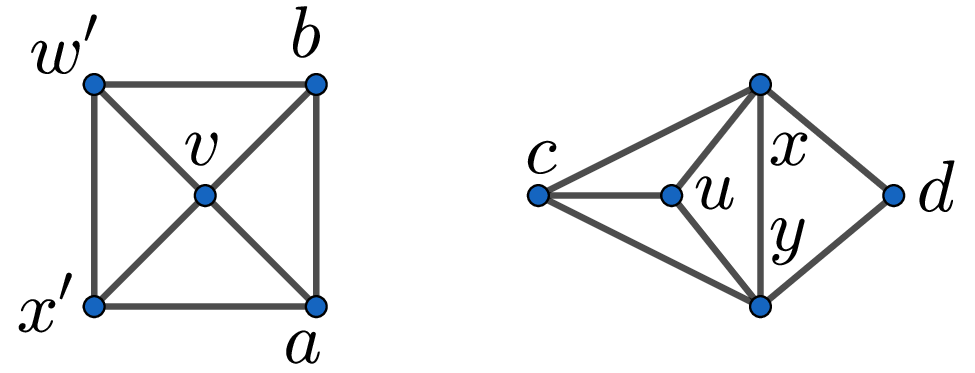}
		\caption{$\cg(G')=G_1\ \dot\cup\ G_2$.}
		\label{f:k4c}
	\end{subfigure}
	\begin{subfigure}{0.48\textwidth}
		\centering
		\includegraphics[width=5.5cm]{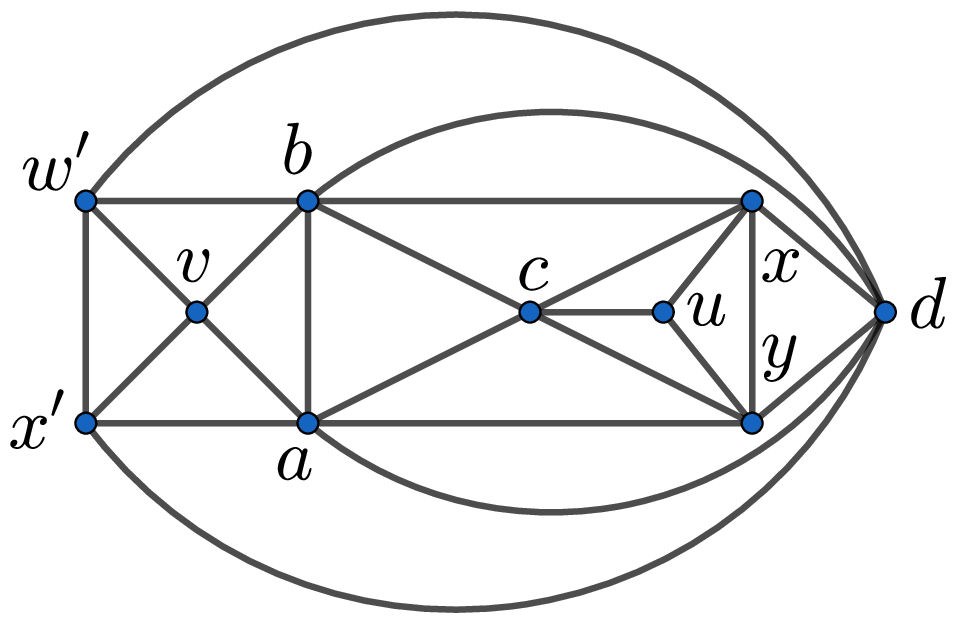}
		\caption{$\cg(G)$.}
		\label{f:k4d}
	\end{subfigure}
	\caption{Illustration of Proposition \ref{p:K4}.}
	\label{f:k4abcd}
\end{figure}

We take a planar immersion of $G'$ such that
\begin{equation*}
a,c,b,d
\end{equation*}
appear in this order along a region $\rho$, as in Figures \ref{f:k4gb} and \ref{f:k4b}. We write $N_G(a)=\{b,w,x\}$, $N_G(b)=\{a,y,z\}$, $N_G(c)=\{d,w',x'\}$, $N_G(d)=\{c,y',z'\}$ where in each case the neighbours appear in clockwise order around the vertex. We are also assuming w.l.o.g.\ that $\beta,\gamma,\delta$ appear in this clockwise order around $\alpha$, that $b$ is the clockwise neighbour of $a$ on $\alpha$, and $d$ the clockwise neighbour of $c$ on $\gamma$ in the polyhedron $G$.

Hence we may sketch $G_1,G_2$ in the plane so that either the external region $\rho_1$ of $G_1$ contains in this order
\[a,c,b,d\]
and the external region $\rho_2$ of $G_2$ contains in this order
\[x,w,y',z',z,y,w',x',\]
or $\rho_1$ contains in this order
\[a,x',w',b,z',y'\]
and $\rho_2$ contains in this order
\[x,w,d,z,y,c\]
(this is the case in Figure \ref{f:k4c}). It now remains to draw the edges
\begin{equation}
	\label{eq:12v}
ay,\ az,\ cy',\ cz',\ bw,\ bx,\ dw',\ dx'
\end{equation}
(Figure \ref{f:k4d}). In both of the above scenarios, these edges may be drawn in a planar way. It follows that $\cg(G)$ is a planar graph.

It remains to prove that $\cg(G)$ is $3$-connected. We claim that one may always find $S\subseteq\{w,x,y,z\}$ and $S'\subseteq\{w',x',y',z'\}$ such that $|S|=|S'|=2$ and
\[S\cup S'\cup\{a,b,c,d\}\]
are eight pairwise distinct vertices. Then the $3$-connectivity of $\cg(G)$ will follow from the $2$-connectivity of $G_1,G_2$, combined with the fact that we add \eqref{eq:12v} to $\cg(G')=G_1\ \dot\cup\ G_2$ to obtain $\cg(G)$.

To prove the claim, we start by remarking that $a,b,c,d$ are pairwise distinct. Indeed, if $a=c$ say, then $a$ would belong to $\alpha,\beta,\gamma,\delta$, impossible since $G$ is cubic. We now distinguish between two cases. If $\alpha$ is a triangular face, then $\alpha=[a,b,d]$, $a=y'$, $b=z'$, and $d=w=z$ (as in Figure \ref{f:k4a}). Now a cubic polyhedron (other than $K_4$) never has two adjacent triangular faces, due to $3$-connectivity. Thus $\beta,\gamma,\delta$ all have length at least five, so that $x,a,b,y$ are distinct consecutive vertices on the boundary of $\beta$, $y,b,d,c,w'$ are distinct and consecutive on $\gamma$, and $x,a,d,c,x'$ are distinct and consecutive on $\delta$. Hence
\[
x,\ y,\ w',\ x',\ a,\ b,\ c,\ d
\]
are pairwise distinct, as claimed.

In the other case, all of $\alpha,\beta,\gamma,\delta$ have length at least five. Then $w,x,y,z,a,b$ are pairwise distinct. If $c,d$ are distinct from $w,x,y,z$, the claim is proven. Assume that $d=w$, say, as in Figure \ref{f:k4e}. As $G$ is cubic, two of the neighbours of $d$ (including $a$) lie on $\alpha$. Since $cd$ is the common edge of $\gamma$ and $\delta$, then $c$ cannot lie on $\alpha$. It follows that $a=y'$, and along the boundary of $\alpha$ we find the distinct consecutive vertices $z,b,a,d,z'$. Similarly, since $G$ is cubic, along the boundary of $\delta$ we find the distinct consecutive vertices $x,a,d,c,x'$. The common edge of $\alpha$ and $\delta$ is $ad$, and the common edge of $\beta$ and $\delta$ is $ax$, so that in particular $c\neq x,y,z,z'$ and $x'\neq b,x,y,z,z'$. It follows that 
\[
x,\ y,\ x',\ z',\ a,\ b,\ c,\ d
\]
are pairwise distinct, as claimed.
\begin{figure}[ht]
	\centering
\includegraphics[width=6cm]{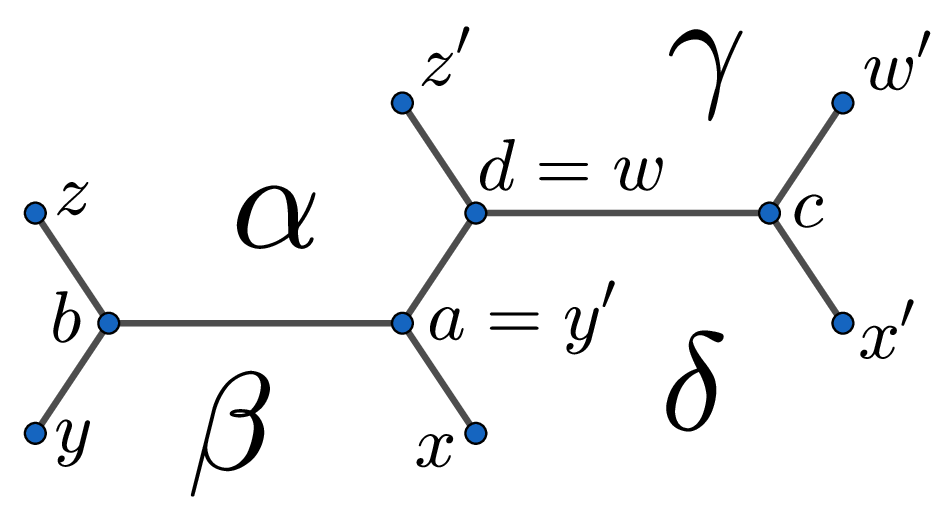}
\caption{Situation when $\alpha,\beta,\gamma,\delta$ are of length at least five, and $d=w$.}
\label{f:k4e}
\end{figure}

\end{proof}

\subsection{Proof of Theorem \ref{thm:1}, \ref{(a)}$\Rightarrow$\ref{(c)}}
Recall that $\co$ is the set of odd faces of a polyhedron $G$.

\begin{defin}
	\label{def:even}
	Given a cubic polyhedron $G$, we may associate to it a planar, bipartite, spanning subgraph $G'$, that we will call the {\em evenisation} of $G$, in the following way. We begin by considering a partition of $\co$ into pairs
	\begin{equation}
		\label{eq:parti}
		\{\om_1,\om_2\},\ \{\om_3,\om_4\}, \dots, \{\om_{|\co|-1},\om_{|\co|}\}
	\end{equation}
	such that the total distance between the corresponding vertices in the dual $G^*$ (labelled in the same fashion)
	\begin{equation}
		\label{eq:sum}
		\sum_{\substack{i=1\\i\text{ odd}}}^{|\co|}\dist(\om_i,\om_{i+1})
	\end{equation}
	is minimised. We then take $|\co|/2$ paths
	\[\om_i=\varphi_{i,1},\varphi_{i,2},\dots,\varphi_{i,n_i+1}=\om_{i+1}, \qquad n_i\geq 1,\quad 1\leq i\leq |\co|,\quad i\text{ odd}\]
	of minimal length. For each of the  paths, we delete from $G$ the common edge between every pair of faces corresponding to consecutive vertices along the path, obtaining $G'$.
\end{defin}

\begin{lemma}
	\label{le:even}
For every cubic polyhedron $G$, the evenisation $G'$ is a $2$-connected graph.
\end{lemma}
\begin{proof}
Taking the $|\co|/2$ paths in $G^*$ as in Definition \ref{def:even}, we denote by $G''$ the subgraph of $G^*$ generated by the edges of these paths. By minimality of the total length of the paths \eqref{eq:sum}, $G''$ contains no cycles. We delete from $G$ the common edge of
\begin{equation}
	\label{eq:com}
\varphi_{i,j},\varphi_{i,j+1},\qquad 1\leq j\leq n_i,\quad 1\leq i\leq |\co|,\quad i\text{ odd}.
\end{equation}
One cannot have $u\in V(G)$ lying on $\varphi_{i,j},\varphi_{i',j'},\varphi_{i,j+1}=\varphi_{i',j'+1}$, which would result in $\deg_{G'}(u)=1$. Indeed, in this scenario one could permute $\omega_{i+1},\omega_{i'}$ in \eqref{eq:parti}, resulting in a smaller sum \eqref{eq:sum}.

We claim that, as each edge \eqref{eq:com} is eliminated in turn, in the resulting graph each region is always bounded by a cycle. Indeed, when eliminating an edge $uv$, the regions that contain it
\[[u,v,a_1,a_2,\dots,a_A],\qquad A\geq 1\]
and
\[[u,v,b_1,b_2,\dots,b_B],\qquad B\geq 1,\]
with $A,B$ of opposite parity, merge to form the new region
\begin{equation}
	\label{eq:AB}
	[a_1,a_2,\dots,a_A,b_B,b_{B-1},\dots,b_1]
\end{equation}
(cf.\ the proof of Proposition \ref{p:K2} for the case $|\co|=2$). As $G''$ contains no cycles, the vertices in \eqref{eq:AB} are pairwise distinct. Hence as we remove the edges from $G$, at each step the graph stays $2$-connected.
\end{proof}

We are ready to complete the proof of Theorem \ref{thm:1}.

\begin{proof}[Proof of Theorem \ref{thm:1}, \ref{(a)}$\Rightarrow$\ref{(c)}]
Suppose that $\cg(G)$ is planar and connected. As established in the first paragraph of Proposition \ref{p:p2cb}, since $\cg(G)$ is connected, then $G$ is not bipartite. Therefore, $|\od(G^*)|>0$ and due to the handshaking lemma, in fact $|\od(G^*)|\geq 2$. By Lemma \ref{le:cuodd}, $\od(G^*)\not\simeq K_2$. In the rest of this proof, we will show that either $\od(G^*)\simeq\overline{K_2}$ or $\od(G^*)\simeq K_4$.

We construct the evenisation $G'$ of $G$ as in Definition \ref{def:even}. We shall denote the edges deleted to obtain $G'$ as
\begin{equation}
\label{eq:uivi}
u_iv_i,\qquad 1\leq i\leq n,\quad n\geq 2.
\end{equation}
By Lemma \ref{le:even}, $G'$ is a plane, $2$-connected, bipartite, spanning subgraph of $G$ (in particular, $u_1,v_1,u_2,v_2,\dots,u_n,v_n$ are pairwise distinct vertices). By Proposition \ref{p:p2cb}, $\cg(G')$ is a planar graph with exactly two connected components, both of which are $2$-connected.

To construct $\cg(G)$, we add to $\cg(G')$ the edges
\begin{equation}
	\label{eq:wxyz2}
	u_iy_i,\ u_iz_i,\ v_iw_i,\ v_ix_i, \qquad 1\leq i\leq n,
\end{equation}
where $w_i,x_i$ are the neighbours of $u_i$ and $y_i,z_i$ the neighbours of $v_i$ in $G'$ (to be precise, in the region $\rho_i$ of $G'$ containing $u_i,v_i$, the vertex $x_i$ is the clockwise neighbour of $u_i$ and $z_i$ the clockwise neighbour of $v_i$). Since $\cg(G)$ is planar, there exists a plane immersion of $\cg(G')$ such that
\begin{equation}
	\label{eq:uvwxyz}
u_i,\ v_i,\ w_i,\ x_i,\ y_i,\ z_i,\qquad 1\leq i\leq n
\end{equation}
all belong to the same region. Now since $G'$ is bipartite, there exists a plane immersion of $G'$ such that \eqref{eq:uvwxyz} all belong to the same region.

Let \[\om_1,\ \om_2,\ \om_3,\ \om_4,\ \om_5,\ \om_6\] be odd faces of $G$ as in \eqref{eq:parti}. We may assume w.l.o.g.\ that the endpoints of the edges \eqref{eq:uivi} are labelled so that, when we construct $G'$ from $G$,
\[u_1v_1,u_2v_2,\dots,u_mv_m\]
are the deleted edges pertaining to the $\om_1\om_2$-path in $G^*$, and $m$ is maximised. Let $u_hv_h$, $u_kv_k$ be any deleted edges pertaining to the $\om_3\om_4$-path and $\om_5\om_6$-path respectively (thus $m<h\neq k\leq n$).

We claim that in the plane immersion of $G'$, the vertices
\begin{equation}
\label{eq:order}
u_1,\ u_h,\ u_k,\ z_1,\ v_1,\ z_2,\ v_2,\ z_h,\ z_k,\ x_2,\ u_2,\ x_1
\end{equation}
appear w.l.o.g.\ in this order around the boundary of a region $\rho$. To prove the claim, it suffices to show that $u_1,u_h,v_1,v_h$ appear in this order around $\rho$ (Figure \ref{f:1hab}, left). Indeed, if the arrangement were $u_1,v_1,u_h,v_h$ as in Figure \ref{f:1hcd}, left, then bearing in mind that $G$ is cubic, $G$ would have a plane immersion with $u_1,v_1,u_h,v_h$ on the same face (Figure \ref{f:1hcd}, right), contradicting the uniqueness of plane immersion for a polyhedron. We also observe that $u_1,u_h,v_1,v_h$ cannot be arranged as in Figure \ref{f:1he}, since $G'$ has a plane immersion such that these four vertices all lie on one region $\rho$.

\begin{figure}[ht]
	\centering
	\begin{subfigure}{0.34\textwidth}
		\centering
		\includegraphics[width=2.5cm]{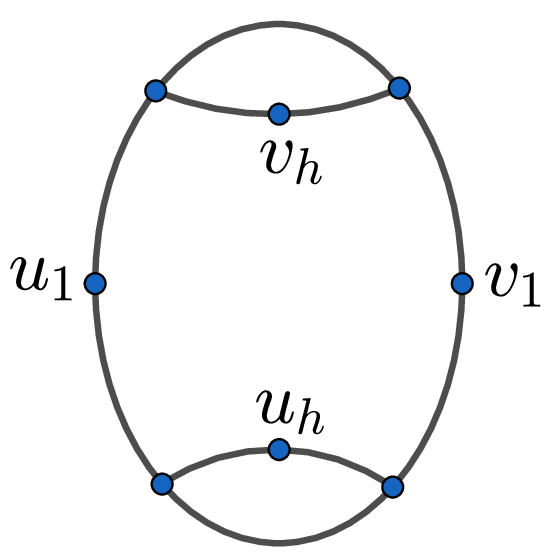}
		\includegraphics[width=2.5cm]{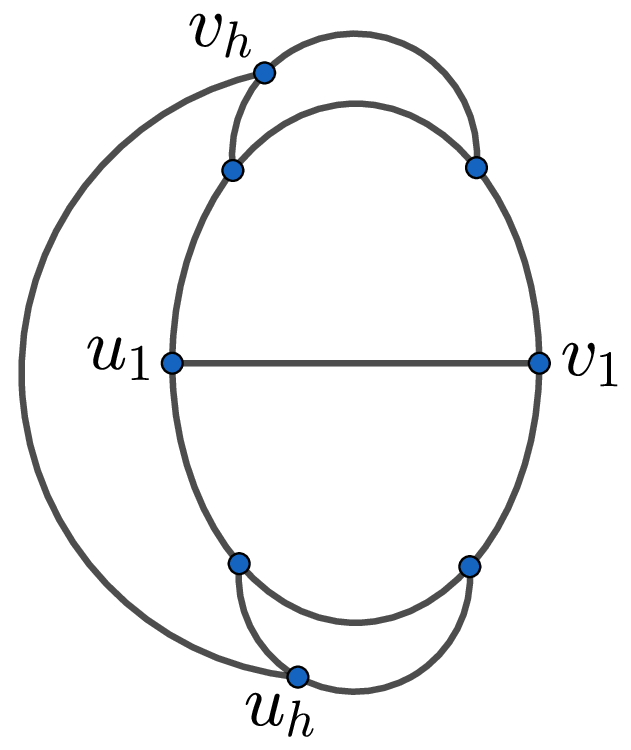}
		\caption{The only possible arrangement in $G'$ (left) and $G$ (right).}
		\label{f:1hab}
	\end{subfigure}
	\hspace{0.75cm}
	\begin{subfigure}{0.34\textwidth}
		\centering
		\includegraphics[width=2.5cm]{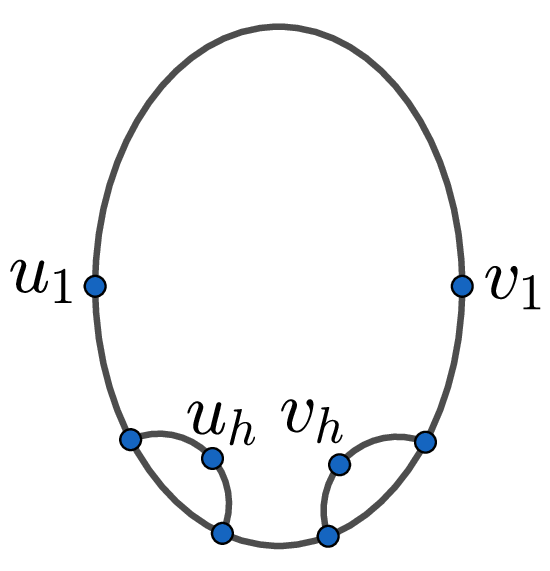}
		\includegraphics[width=2.5cm]{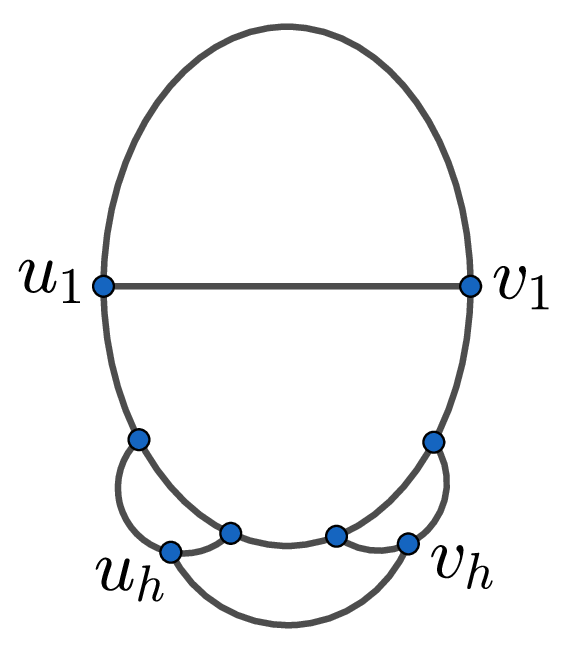}
		\caption{An impossible arrangement in $G'$ (left) and $G$ (right).}
		\label{f:1hcd}
	\end{subfigure}
	\hspace{0.75cm}
	\begin{subfigure}{0.19\textwidth}
		\centering
		\includegraphics[width=2.5cm]{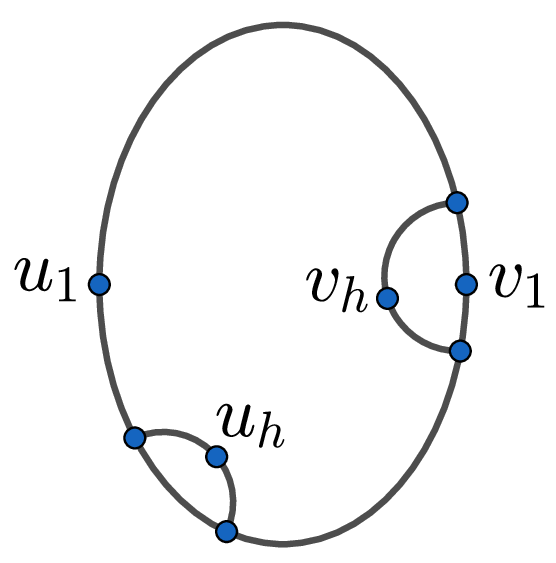}
		\caption{An impossible arrangement in $G'$.}
		\label{f:1he}
	\end{subfigure}
	\caption{$u_1,u_h,v_1,v_h$ must be ordered in this way around $\rho$.}
	\label{f:1h}
\end{figure}

Hence indeed $\rho$ contains the vertices \eqref{eq:order} in this order. Accordingly, in one connected component of $\cg(G')$, there is a region containing
\[u_1',\ u_h',\ u_k',\ v_1',\ v_2',\ v_h',\ v_k',\ u_2'\]
in this order, and in the other component of $\cg(G')$ there is a region containing
\[x_1',\ x_k',\ x_h',\ z_1',\ z_2',\ z_k',\ z_h',\ x_2'\]
in this order, where for $i=1,2,h,k$, either $u_i'=u_i$, $v_i'=v_i$, $x_i'=x_i$ and $z'_i=z_i$, or $u_i'=x_i$, $v_i'=z_i$, $x_i'=u_i$, and $z_i'=v_i$. Therefore, the edges
\[u_1z_1,\ u_hz_h,\ u_kz_k\]
cannot all be added to $\cg(G')$ in a planar way. It follows that the vertices with index $k$ do not appear, thus $|\co|\leq 4$.

Similarly, the edges
\[u_1z_1,\ v_1x_1,\ u_2z_2,\ v_2x_2,\ u_hz_h\]
cannot all be added to $\cg(G')$ in a planar way. It follows that either the vertices with index $h$ do not appear, thus $|\co|\leq 2$, and we have already seen that this implies $\od(G^*)\simeq\overline{K_2}$, or $m=1$ so that the vertices with index $2$ do not appear, thus $|\co|=4$. 

In the latter case, we then know that $\om_1,\om_2$ are adjacent, and that $\om_3,\om_4$ are adjacent. We sketch $G',G$ as in Figure \ref{f:1hfg} (cf.\ Figure \ref{f:1hab}), where $v_h=s_{S'}$ and $u_h=t_{T'}$ for some $1\leq S'\leq S$ and $1\leq T'\leq T$. Aside from the edge $u_hv_h$, for $1\leq i\leq S$ and $1\leq j\leq T$ there cannot be any $s_it_j$-paths outside of the faces $\om_1,\om_2$ in Figure \ref{f:1hg}, otherwise $u_1,u_h,v_1,v_h$ would not lie on the same region $\rho$ of $G'$ (Figure \ref{f:1hf}). Thus \[s_1,s_2,\dots,s_{S'}=v_h,t_{T'}=u_h,t_{T'+1},\dots,t_T\in V(\om_3)\]
and
\[t_1,t_2,\dots,t_{T'}=u_h,s_{S'}=v_h,s_{S'+1},\dots,s_S\in V(\om_4)\]
(or vice versa). Furthermore, since $G$ is cubic, $s_0,t_{T+1}\in V(\om_3)$ and $s_{S+1},t_0\in V(\om_4)$, thus $\om_1$ is adjacent to $\om_3$ and to $\om_4$, and also $\om_2$ is adjacent to $\om_3$ and to $\om_4$. Thereby, $\od(G^*)\simeq K_4$, and the proof of Theorem \ref{thm:1} is complete.
\begin{figure}[ht]
\centering
\begin{subfigure}{0.48\textwidth}
	\centering
	\includegraphics[width=3.5cm]{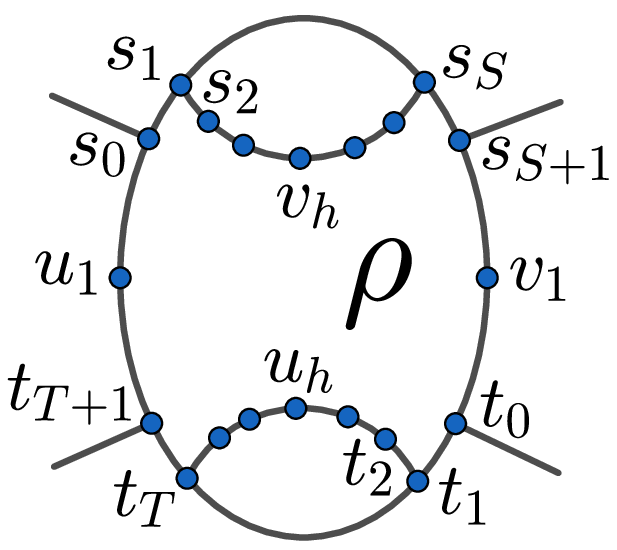}
	\caption{$G'$.}
	\label{f:1hf}
\end{subfigure}
\begin{subfigure}{0.48\textwidth}
	\centering
	\includegraphics[width=3.5cm]{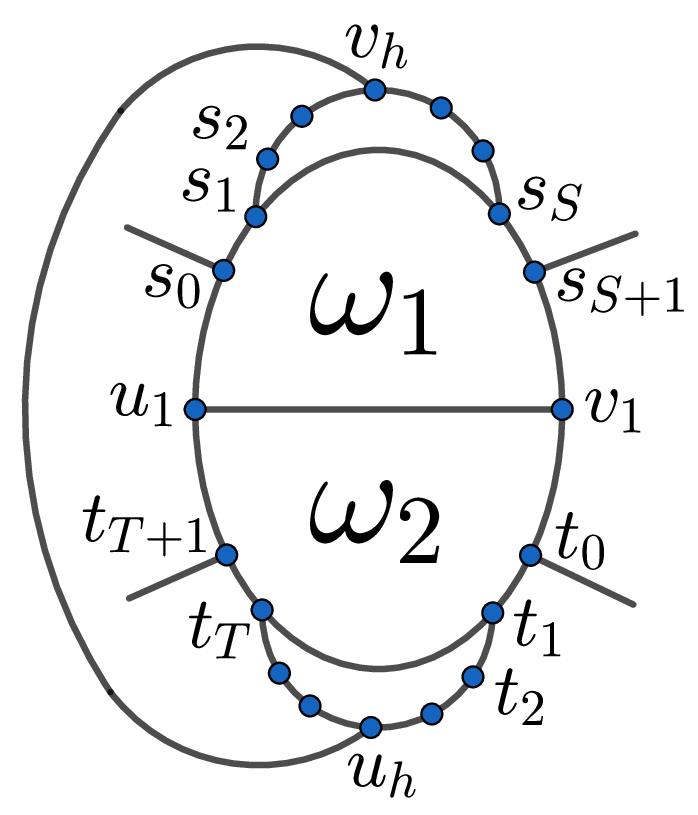}
	\caption{$G$.}
	\label{f:1hg}
\end{subfigure}
\caption{If $|\co|=4$ then $\od(G^*)\simeq K_4$.}
\label{f:1hfg}
\end{figure}

\end{proof}

\section{Polyhedra of maximal degree $4$: proof of Theorems \ref{thm:2} and \ref{thm:0}}
\label{sec:4}

\subsection{Preparatory results}
We begin by formally stating and proving a result mentioned in the introduction.
\begin{lemma}
	\label{le:5}
Let $G$ be a graph such that $\Delta(G)\geq 5$. Then $\cg(G)$ is non-planar.
\end{lemma}
\begin{proof}
Let $x$ be a vertex of degree $d\geq 5$. The neighbours of $x$ are pairwise adjacent in $\cg(G)$, hence $\cg(G)$ contains a subgraph isomorphic to $K_d$, thus $\cg(G)$ is not planar.
\end{proof}

The following result is about the effect on $\cg(G)$ of contracting edges in $G$.
\begin{lemma}
	\label{le:cont}
Let $U\geq 3$ be an integer, and $G$ be a graph containing the path 
\[x,u_1,u_2,\dots,u_U,y=u_{U+1}\]
where
\[\deg_G(u_1)=\deg_G(u_2)=\dots=\deg_G(u_{U})=2.\]
If $i$ is an odd integer in the range $3\leq i\leq U$, and
\[G':=G-u_{2}-u_{3}-\dots-u_i+u_{1}u_{i+1},\]
then $\cg(G')$ is a contraction of $\cg(G)$, so that in particular $\cg(G)$ and $\cg(G')$ are homeomorphic.
\end{lemma}
\begin{proof}
Note that by definition of common neighbourhood graph, one has
\[\deg_{\cg(G)}(u_2)=\deg_{\cg(G)}(u_3)=\dots=\deg_{\cg(G)}(u_{U-1})=2.\]
Also by the same definition, $\cg(G')$ is obtained from $\cg(G)$ by contracting the path
\[x,u_2,u_4,\dots,u_{i+1}\]
to the edge $xu_{i+1}$, and contracting 
the path
\[u_1,u_3,\dots,u_{i}\]
to a single vertex, which is obtained by identifying $u_{1}$ and $u_i$. 
\end{proof}

\begin{lemma}
	\label{le:3cut}
	Let $G$ be a polyhedron containing a vertex $x$ of degree $4$, and such that $\cg(G)$ is planar. Then $\{a,x,c\}$ is a $3$-cut in $G$, where $a,c$ are neighbours of $x$ that are not consecutive in the cyclic order.
\end{lemma}
\begin{proof}
	Let $x\in V(G)$ be of degree $4$, and $y$ the vertex of maximal distance from $x$ in $G$. After checking that the two polyhedra on $5$ vertices have a non-planar common neighbourhood graph, we may assume that $G$ has at least $6$ vertices, thus $\dist_G(x,y)\geq 2$. By $3$-connectivity, there exist three internally disjoint $xy$-paths in $G$. We write in cyclic order
	\[N(x)=\{a_1,b_1,c_1,d_1\},\]
	and w.l.o.g.
	\[
	P_a: x,a_1,a_2,\dots,a_A=y,\qquad P_b: x,b_1,b_2,\dots,b_B=y,\qquad P_c: x,c_1,c_2,\dots,c_C=y,
	\]
	$A,B,C\geq 2$ for the three internally disjoint $xy$-paths, that do not contain $d_1$.
	
	Since $G-x-y$ is connected, and since the cyclic order of vertices around $x$ is $a_1,b_1,c_1,d_1$, then we may find in $G$ a path
	\[P_d: d_1,d_2,\dots,d_D, \qquad D\geq 2,\]
	where either $d_D=a_i$ for some $1\leq i\leq A-1$, or $d_D=c_i$ for some $1\leq i\leq C-1$, and $d_1,d_2,\dots,d_{D-1}$ do not belong to $P_a,P_b,P_c$, as in Figure \ref{f:xyz0} (cf. Figure \ref{f:np}, illustrating what might happen if one drops the assumption that $G$ is planar).
	\begin{figure}[ht]
	\centering
	\includegraphics[width=5cm]{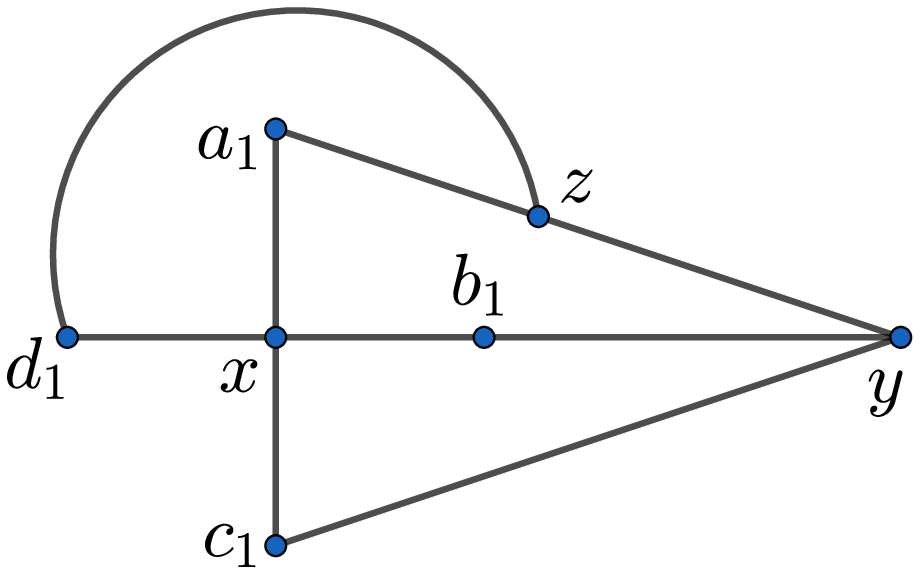}
	\caption{Situation for Lemma \ref{le:3cut}.}
	\label{f:xyz0}
	\end{figure}
	
	We claim that $i=1$. If we can prove it, then every $d_1b_1$-path in $G$ must contain one of ${a_1,x,c_1}$, completing the proof of the present lemma.
	
	Suppose by contradiction that $i\geq 2$. W.l.o.g., we may take
	\[z:=d_D=a_i.\]
	We define a subgraph $G'$ of $G$ consisting of the internally disjoint paths
	\[P_b\text{ of length }B,\quad P_c\text{ of length }C\]
	from $x$ to $y$,
	\[P_d\text{ of length }D,\quad P_a': x,a_1,a_2,\dots,a_{i-1},z\text{ of length }i\]
	from $x$ to $z$, and
	\[P_a'': z,a_{i+1},\dots,y\text{ of length }A-i\]
	from $z$ to $y$. Seeing as $B,C,D,i\geq 2$, we may contract each of
	\[P_b,\ P_c,\ P_d,\ P_a'\]
	to a path of length $2$ or $3$ respectively if the original path is of even or odd length, and $P_a''$ to a path of length $1$, $2$, or $3$, respectively if the length $A-i$ of the original path is $1$, a positive even, a positive odd integer, obtaining a graph $G''$. By Lemma \ref{le:cont}, $\cg(G'')$ is a contraction of $\cg(G')$, thus $\cg(G'')$ is a minor of $\cg(G)$. 
	
	Due to the symmetries one may interchange $y,z$ or $P_b,P_c$ or $P_d,P_a'$, so that there are six total scenarios for the lengths of the corresponding contracted paths in $G''$, as in Figure \ref{f:xyz}. These may be combined with the three options for the length of the contracted path corresponding to $P_a''$, yielding eighteen total possibilities for the graph $G''$. We check that for each of these, $\cg(G'')$ is non-planar. Therefore, $\cg(G)$ is non-planar, contradiction. We have shown that indeed $i=1$.
		\begin{figure}[ht]
			\centering
			\includegraphics[width=0.275\textwidth]{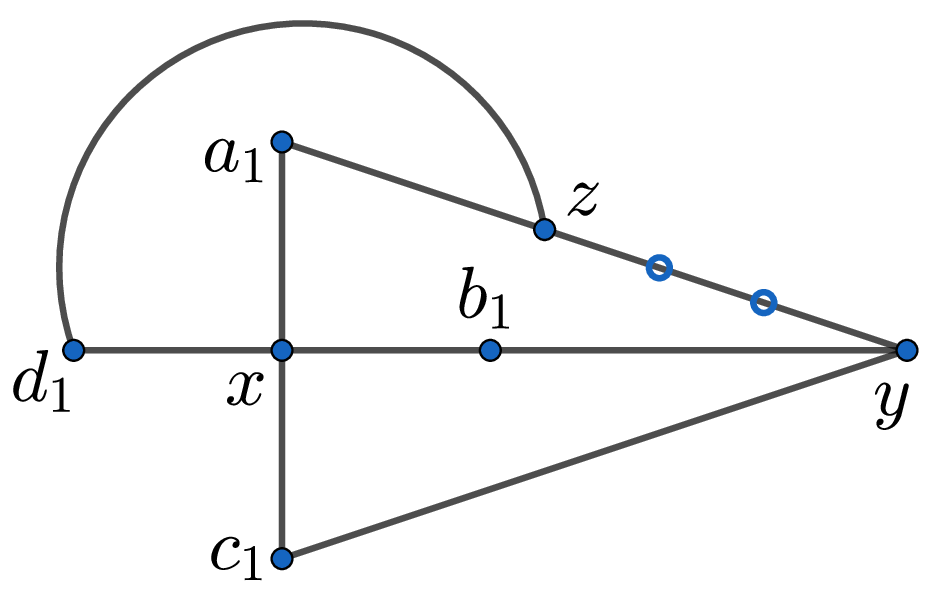}
			\hspace{0.75cm}
			\includegraphics[width=0.275\textwidth]{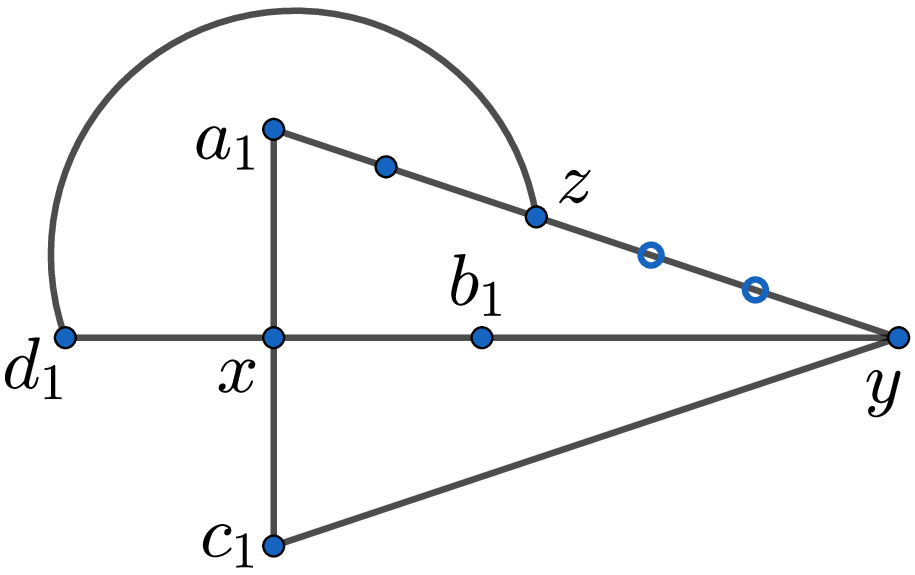}
			\hspace{0.75cm}
			\includegraphics[width=0.275\textwidth]{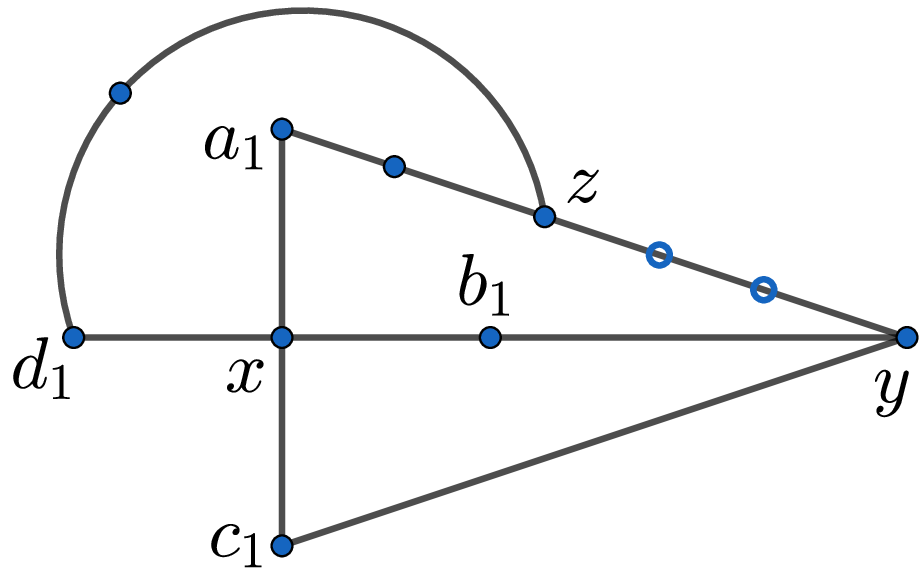}
			\hspace{0.75cm}
			\includegraphics[width=0.275\textwidth]{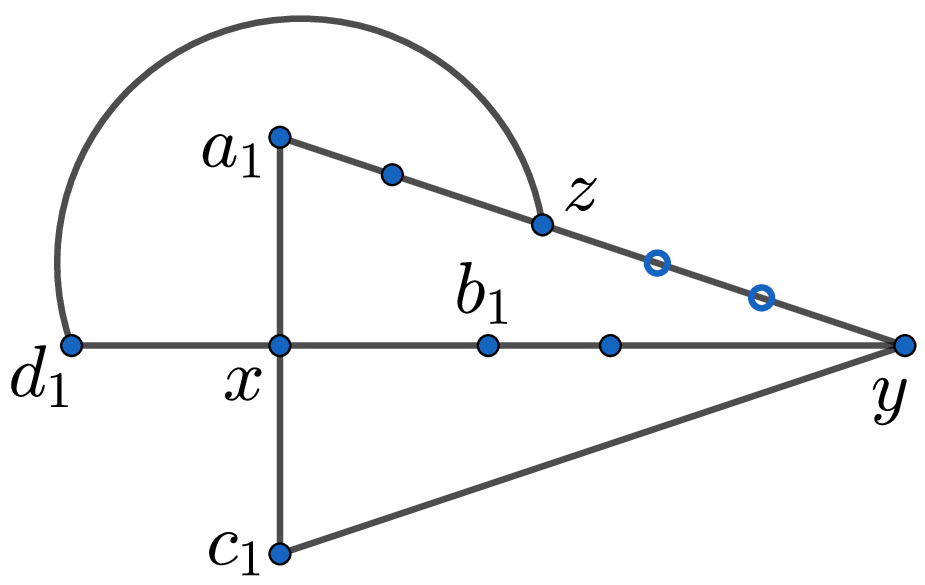}
			\hspace{0.75cm}
			\includegraphics[width=0.275\textwidth]{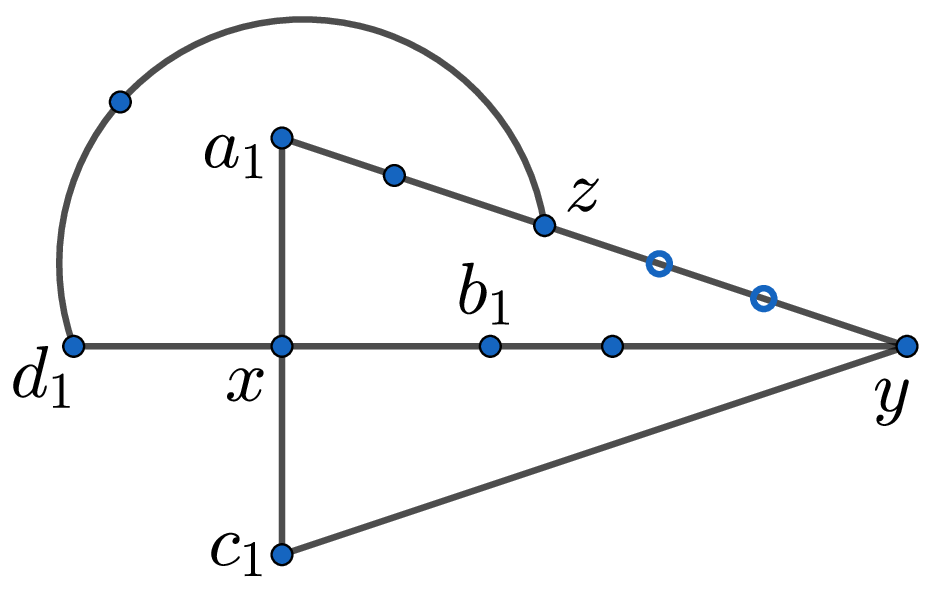}
			\hspace{0.75cm}
			\includegraphics[width=0.275\textwidth]{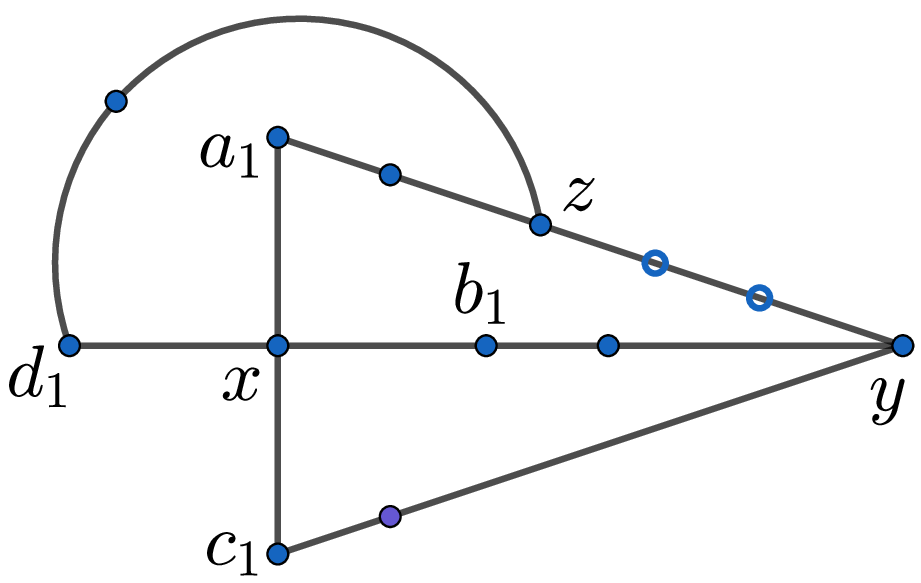}
			\caption{The eighteen possibilities for the minor $G''$ of $G$. In each of the six pictures, there may be either $0$, or $1$, or $2$ of the empty vertices between $z$ and $y$, yielding three possibilities per picture.}
			\label{f:xyz}
		\end{figure}
	
\end{proof}

\begin{lemma}
	\label{le:main}
Let $G$ be a polyhedron such that $\Delta(G)\geq 4$ and $\cg(G)$ is planar. Then we may list the vertices of degree $4$ of $G$ as
\[x_1,x_2,\dots,x_n, \qquad n\geq 1\]
such that for every $1\leq i\leq n$ there is a $3$-cut $\{a_i,x_i,c_i\}$ of $G$, with $a_i,c_i$ neighbours of $x_i$ that are not consecutive in the cyclic order around $x_i$, and that are of degree $3$ in $G$, and moreover the vertices
\[a_1,a_2,\dots,a_n,c_n,c_{n-1},\dots,c_1\]
appear in this order along the external face of $G$, as in Figure \ref{f:axc1}.
\begin{figure}[h!]
	\centering
	\includegraphics[width=5.cm]{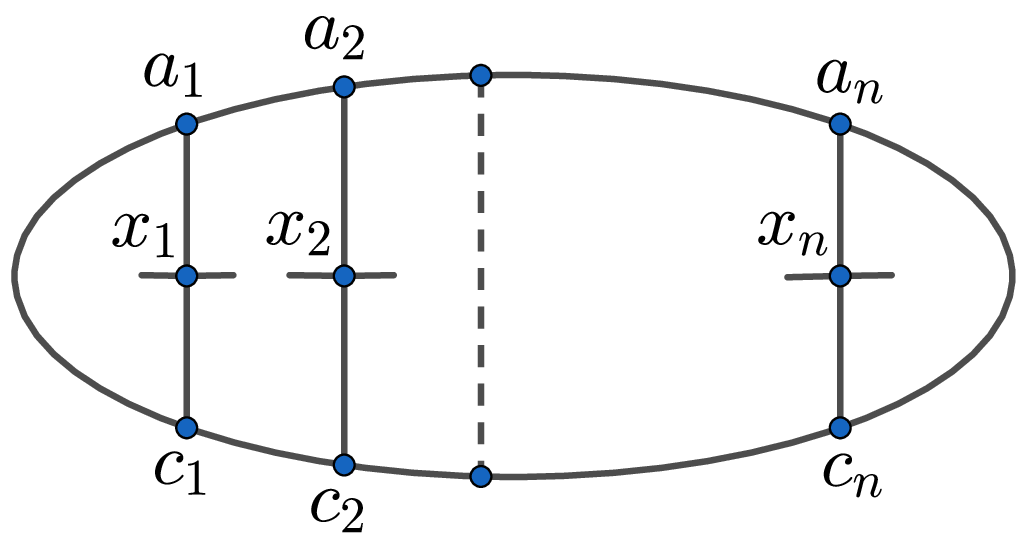}
	\caption{Situation in Lemma \ref{le:main}.}
	\label{f:axc1}
\end{figure}

\end{lemma}
\begin{proof}
Due to Lemma \ref{le:5}, we may assume that $\Delta(G)=4$. Let $x\in V(G)$ be a vertex of degree $4$. By Lemma \ref{le:3cut},
\[\{a,x,c\}\]
is a $3$-cut of $G$, with $a,c$ neighbours of $x$ that are non-consecutive in the cyclic order. It follows that $a,c$ belong to the external face of $G$.

We claim that
\[\deg_{G}(a)=3.\]
By contradiction, we write $N_G(a)=\{w,x,y,z\}$ in this cyclic order, with $y,a,z$ consecutive on the external face of $G$, as in Figure \ref{f:axc}. On the other hand by Lemma \ref{le:3cut}, either $\{w,a,y\}$ or $\{x,a,z\}$ is a $3$-cut in $G$. Hence as above, either $w$ and $y$ or $x$ and $z$ lie on the external face of $G$, contradiction. Therefore, indeed $\deg_{G}(a)=3$. Similarly, $\deg_{G}(c)=3$.
\begin{figure}[ht]
\centering
\includegraphics[width=5.cm]{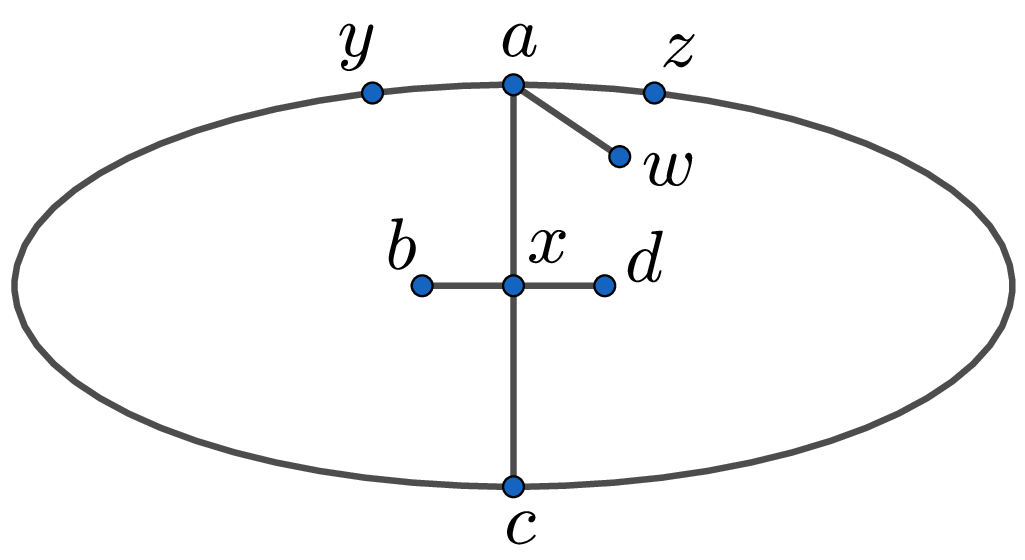}
\caption{$\deg(a)=4$ is impossible.}
\label{f:axc}
\end{figure}

Therefore, every vertex of degree $4$ in $G$ determines a $3$-cut consisting of the vertex itself and two of its neighbours, that are not consecutive in the cyclic order, and that are of degree $3$ in $G$. Every such $3$-cut splits $G$ into $2$ components. We may then list the vertices of degree $4$ of $G$ as
\[x_1,x_2,\dots,x_n, \qquad n\geq 1\]
such that for every $1\leq i\leq n$ there is a $3$-cut $\{a_i,x_i,c_i\}$ of $G$, with $a_i,c_i$ neighbours of $x_i$ that are not consecutive in the cyclic order, and of degree $3$ in $G$, and moreover the vertices
\[a_1,a_2,\dots,a_n,c_n,c_{n-1},\dots,c_1\]
appear in this order along the external face of $G$, as in Figure \ref{f:axc1}.
\end{proof}

\begin{defin}
	\label{def:0n}
Let $G$ be a polyhedron such that $\Delta(G)\geq 4$ and $\cg(G)$ is planar. With the notation of Lemma \ref{le:main}, for $n=1$ we call $G_0,G_{1}$ the connected components of
\[G-a_1-x_1-c_1,\]
and for $n\geq 2$ we call $G_0$ the connected component of
\[G-a_1-x_1-c_1\]
not containing $x_2$ and $G_{n}$ the connected component of
\[G-a_n-x_n-c_n\]
not containing $x_{n-1}$. We then define the plane, $2$-connected graphs
\begin{equation}
	\label{eq:g0}
\overline{G_0}:=\langle\{a_1,x_
1,c_1\}\cup V(G_0)\rangle
\end{equation}
and
\begin{equation}
	\label{eq:gn}
\overline{G_{n}}:=\langle\{a_n,x_n,c_n\}\cup V(G_{n})\rangle.
\end{equation}
\end{defin}
Referring to Figure \ref{f:axc1}, we depict $\overline{G_0}$ and $\overline{G_{n}}$ as in Figure \ref{f:g0n}. Note that $\overline{G_0}$ and $\overline{G_{n}}$ are $2$-connected since by construction, all of their regions are delimited by cycles. The internal ones correspond to faces of $G$.
\begin{figure}[ht]
	\centering
	\begin{subfigure}{0.48\textwidth}
		\centering
		\includegraphics[width=1.25cm]{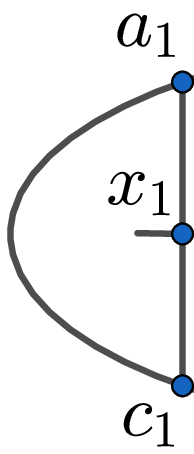}
		\caption{$\overline{G_0}$.}
		\label{f:g0}
	\end{subfigure}
	\begin{subfigure}{0.48\textwidth}
		\centering
		\includegraphics[width=1.25cm]{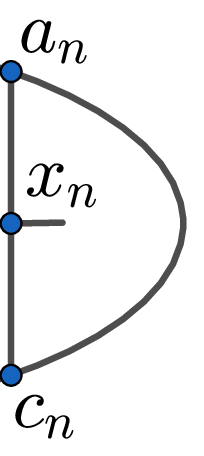}
		\caption{$\overline{G_n}$.}
		\label{f:gn}
	\end{subfigure}
	\caption{Definition \ref{def:0n}.}
	\label{f:g0n}
\end{figure}

\begin{prop}
	\label{p:oddface}
Let $G$ be a polyhedron such that $\Delta(G)\geq 4$ and $\cg(G)$ is planar. Then there is at least one odd face of $G$ lying in each of $\overline{G_0}$ \eqref{eq:g0} and $\overline{G_{n}}$ \eqref{eq:gn}.
\end{prop}
\begin{proof}
By contradiction, assume that all the internal regions of $\overline{G_0}$ correspond to even faces of $G$. By the handshaking lemma, the external region of $\overline{G_0}$ is thus also of even length, so that $\overline{G_0}$ is bipartite.

Call $x:=x_1$, $a:=a_1$, $c:=c_1$, and $b$ the neighbour $\neq a,c$ of $x$ in $\overline{G_0}$. By Lemma \ref{le:main}, in $\overline{G_0}$ every vertex has degree $3$ save for $a$ and $c$ that have degree $2$. We write $N_{\overline{G_0}}(a)=\{x,a'\}$ and $N_{\overline{G_0}}(c)=\{x,c'\}$ (Figure \ref{f:g01}). Thus $a',a,x,c,c'$ are consecutive in this order along the external region of $\overline{G_0}$.
\begin{figure}[ht]
	\centering
	\includegraphics[width=2.cm]{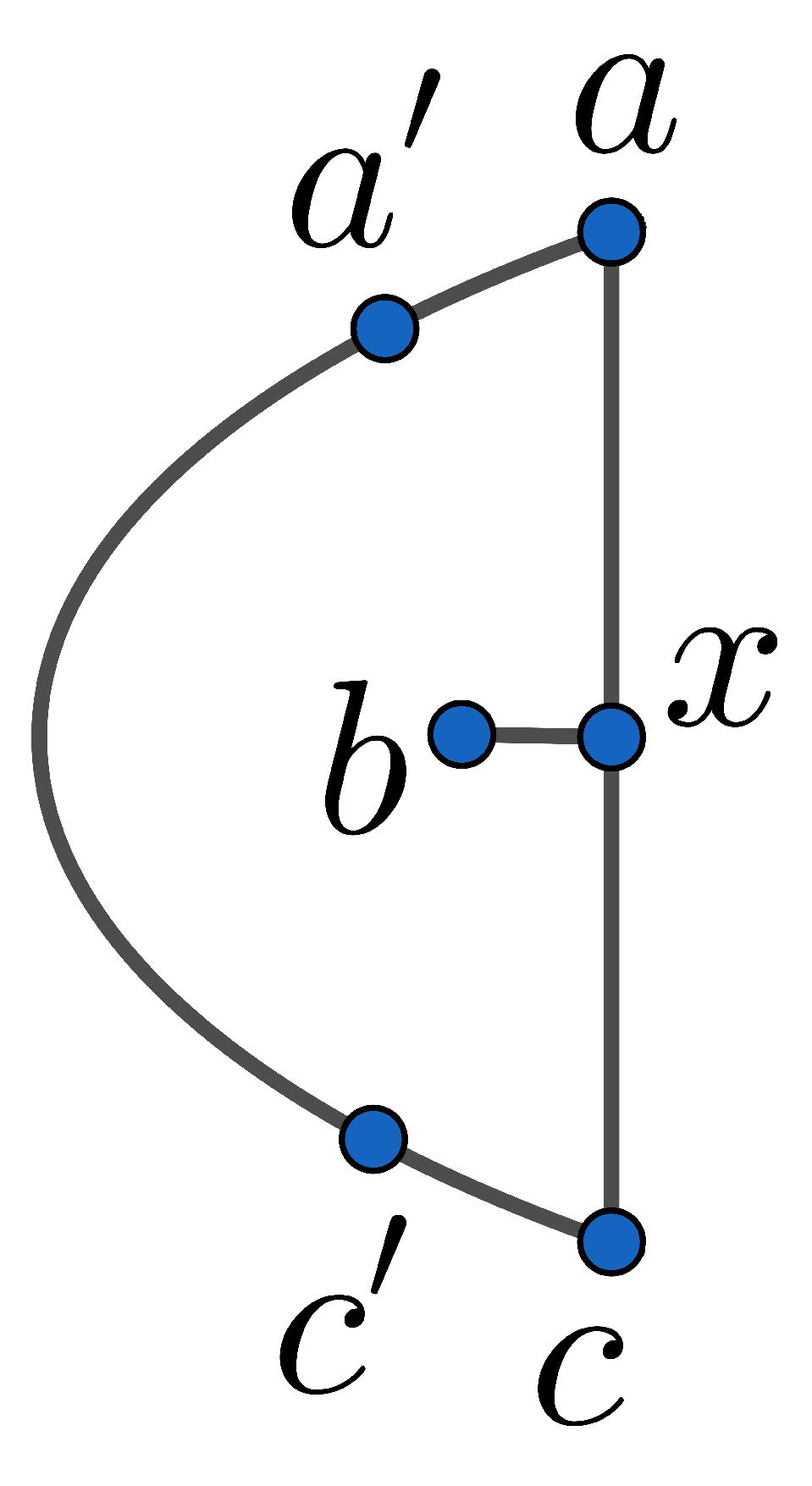}
	\caption{Proposition \ref{p:oddface}, $\overline{G_0}$.}
	\label{f:g01}
\end{figure}

Since $\overline{G_0}$ is bipartite, then $x$ is not adjacent to $a',c'$. Therefore, the contraction
\[G_\#:=\overline{G_0}-a-c+xa'+xc'\]
is a plane, cubic, $2$-connected graph. By construction, every region of $G_\#$ is of even length save for the one containing $b,x,a'$ and the one containing $b,x,c'$, and these regions are adjacent. If we show that $G_\#$ is in fact $3$-connected, then it is a cubic polyhedron with exactly two odd faces that are adjacent, contradicting Lemma \ref{le:cuodd}.

By contradiction, assume that $G_\#$ is not $3$-connected. Since it is plane and $2$-connected, there exist regions $\alpha,\beta$ of $G_\#$ that have two common non-adjacent vertices. By construction, the same is true for $G$, contradiction.

Clearly the reasoning for $\overline{G_{n}}$ is the same as for $\overline{G_0}$.
\end{proof}

\subsection{Concluding the proofs of Theorems \ref{thm:2} and \ref{thm:0}}

The following is an immediate consequence of Proposition \ref{p:oddface}.
\begin{cor}
	\label{cor:bip4}
	Let $G$ be a bipartite polyhedron with $\Delta(G)\geq 4$. Then $\cg(G)$ is not planar.
\end{cor}

Proposition \ref{p:oddface} also has the following consequence.

\begin{cor}
	\label{cor:K24}
	Let $G$ be a polyhedron satisfying $\od(G^*)\simeq K_2$. Then $\cg(G)$ is not planar.
\end{cor}
\begin{proof}
By contradiction, let $\cg(G)$ be planar. Recall the notation of Lemma \ref{le:main} and Definition \ref{def:0n}. By Proposition \ref{p:oddface}, there is at least one odd face of $G$ lying in each of $\overline{G_{0}}$ and $\overline{G_{n}}$. Since $\od(G^*)\simeq K_2$, in fact we have $n=1$ i.e., there is exactly one vertex $x=x_1$ of degree $4$ in $G$. As usual, call $a,b,c,d$ its neighbours in cyclic order, where $\{a,x,c\}$ is a $3$-cut in $G$ due to Lemma \ref{le:main}. Letting $N_G(a)=\{x,a_1,a_2\}$, and calling $\omega_1,\omega_2$ the only two odd faces of $G$, then w.l.o.g., $\omega_1$ contains $a',a,x,b$ consecutively in this order, and $\omega_2$ contains $a'',a,x,d$ consecutively in this order (possibly $a'=b$ and possibly $a''=d$).

We define the graph
\[G':=G-a+a'a''.\]
Clearly in $G'$ every region is bounded by a cycle, hence $G'$ is a plane, $2$-connected, cubic graph. Its only odd regions are the one containing $x,b,d,a',a''$, and the external one. Note that these regions share the edge $a'a''$. If we can show that $G'$ is $3$-connected, we will obtain a contradiction via Lemma \ref{le:cuodd}. Suppose that $G'$ is not $3$-connected. Then there exist two regions of $G'$ that have two common non-adjacent vertices. By construction, one of these regions is the one containing $x,b,d,a',a''$. However this is impossible since $\{a,x,c\}$ is a $3$-cut in $G$.
\end{proof}

The following result constitutes a straightforward way to check that for certain graphs, the congraph is not planar.
\begin{lemma}
	\label{le:loop}
Let $G$ be a graph containing the two cycles
\[a_1=b_1,a_2,\dots,a_m,\qquad\text{and}\qquad a_1=b_1,b_2,\dots,b_n, \qquad m,n \text{ odd}\]
where $a_i\neq b_j$ for every $2\leq i\leq m$ and $2\leq j\leq n$. Then $\cg(G)$ is not planar.
\end{lemma}
\begin{proof}
In $\cg(G)$, the vertices $a_2,a_m,b_2,b_n$ are pairwise adjacent. Moreover, $\cg(G)$ contains the internally disjoint paths
\[x,a_3,a_5,\dots,a_m,\qquad x,a_{m-1},a_{m-3},\dots,a_2,\qquad x,b_3,b_5,\dots,b_n,\qquad x,b_{n-1},b_{n-3},\dots,b_2,\]
where $x:=a_1=b_1$, so that $\cg(G)$ contains a subdivision of $K_5$, hence $\cg(G)$ is not planar.
\end{proof}

The following is the last ingredient needed for the proofs of Theorems \ref{thm:2} and \ref{thm:0}.
\begin{prop}
	\label{p:na}
Let $G$ be a polyhedron such that $\Delta(G)\geq 4$, containing two non-adjacent odd faces. Then $\cg(G)$ is not planar.
\end{prop}
\begin{proof}
By Lemma \ref{le:5} we may take $\Delta(G)=4$. We will use the notation of Lemma \ref{le:main} and Definition \ref{def:0n}. We write $x=x_1\in V(G)$ and $N_G(x)=\{a=a_1,b,c=c_1,d\}$ in this cyclic order around $x$. Call
\begin{equation}
	\label{eq:xi}
\xi_{ab},\ \xi_{bc},\ \xi_{cd},\ \xi_{da}
\end{equation}
the faces of $G$ containing respectively $a,x,b$; $b,x,c$; $c,x,d$; and $d,x,a$, as in Figure \ref{f:omega}.
\begin{figure}[ht]
	\centering
	\includegraphics[width=9.cm]{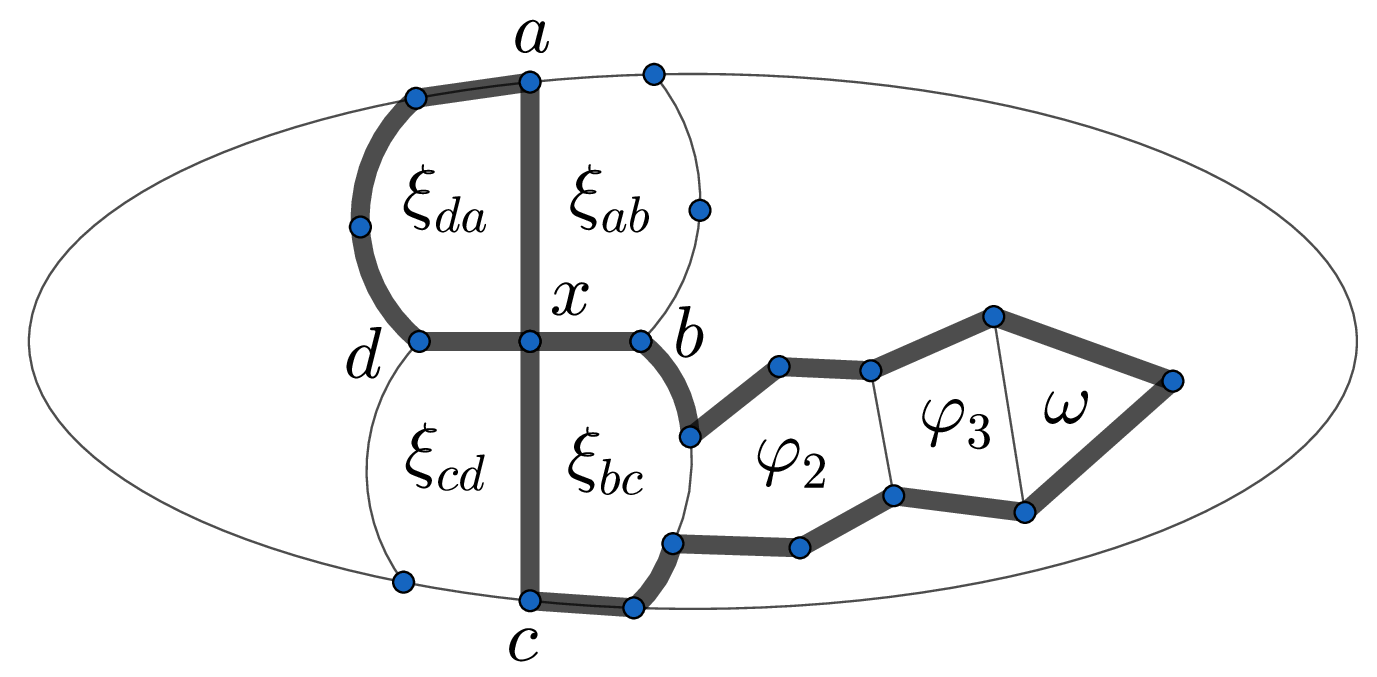}
	\caption{Proposition \ref{p:na}, with $j=m=4$, and $m'=1$. The thick lines delimit the cycles $C_0$ of length $5$ and $C_n$ of length $13$.}
	\label{f:omega}
\end{figure}

Suppose by contradiction that $\cg(G)$ is planar. Then by Proposition \ref{p:oddface}, the graph $\overline{G_0}$ contains at least one internal region $\omega_0$ corresponding to an odd face of $G$, and $\overline{G_{n}}$ contains at least one internal region $\omega_n$ corresponding to an odd face of $G$. By assumption, $\co\neq\{\xi_{da},\xi_{ab}\}$ and $\co\neq\{\xi_{cd},\xi_{bc}\}$. Further, $\xi_{da}$ and $\xi_{bc}$ cannot be both odd, and likewise $\xi_{ab}$ and $\xi_{cd}$ cannot be both odd, else by Lemma \ref{le:loop} $\cg(G)$ would be non-planar.

Thereby, we may assume w.l.o.g.\ that $\overline{G_{n}}$ has an internal odd region $\omega_n$ distinct from $\xi_{ab},\xi_{bc}$ and from the external region. We take $\{\eta,\eta'\}=\{\xi_{ab},\xi_{bc}\}$ with $\eta,\eta'$ to be determined later, and find in $G^*$ a simple path between the vertex corresponding to $\eta$ and the one corresponding to $\omega_n$, that does not include any of the vertices corresponding to $\eta'$, $\xi_{cd}$, $\xi_{da}$, nor the one corresponding to the external face of $G$. This yields in $G$ a list of faces
\[\eta=\varphi_1,\varphi_2,\dots,\omega=\varphi_{m},\qquad m\geq 1,\]
with $\varphi_i$ adjacent to $\varphi_{i+1}$ for every $1\leq i\leq m-1$. Calling $\varphi_j$ the face of odd length where $j\geq 1$ is minimal, we consider in $G$ the cycle $C_{n}$ such that $\varphi_1,\varphi_2,\dots,\varphi_j$ are the only faces internal to $C_n$. Since $\varphi_j$ is odd and $\varphi_1,\varphi_2,\dots,\varphi_{j-1}$ are even, then the length of $C_n$ is odd. All the vertices of $C_{n}$ belong to
\[(V\setminus V(\overline{G_{0}}))\cup\{x,z\},\qquad z=\begin{cases}
a & \eta=\xi_{ab} \\ c & \eta=\xi_{bc}.
\end{cases}\]
If $\xi_{cd}$ is odd, we take $\eta=\xi_{ab}$, so that the only common vertex of $C_n$ and the cycle delimiting $\xi_{ab}$ is $x$. By Lemma \ref{le:loop}, $\cg(G)$ is not planar.

If $\xi_{cd}$ is even, we take $\eta=\xi_{bc}$. By Proposition \ref{p:oddface}, there is an internal region $\omega_0$ in $\overline{G_{0}}$ corresponding to an odd face of $G$. We proceed as above, building a cycle $C_0$ of odd length around the faces \[\xi_{da}=\psi_1,\psi_2,\dots,\psi_{m'}, \qquad m'\geq 1,\]
with $\psi_{m'}$ odd, $\psi_1,\psi_2,\dots,\psi_{m'-1}$ even, $\psi_i$ adjacent to $\psi_{i+1}$ for every $1\leq i\leq m'-1$, and all the vertices of $C_0$ belonging to
\[V(\overline{G_0})\setminus\{c\}\]
($m'=1$ if and only if $\xi_{da}$ is itself odd). By construction, the only common vertex of the cycles $C_0$ and $C_{n}$ is $x$. By Lemma \ref{le:loop}, $\cg(G)$ is not planar.
\end{proof}

We are ready to conclude the proofs of Theorems \ref{thm:2} and \ref{thm:0}.

\begin{proof}[Proof of Theorem \ref{thm:2}]
By Lemma \ref{le:5} we may take $\Delta(G)=4$. If $G$ is bipartite, then $\cg(G)$ is not planar by Corollary \ref{cor:bip4}. If $\od(G^*)\simeq K_2$, then $\cg(G)$ is not planar by Corollary \ref{cor:K24}. If $G$ contains two non-adjacent odd faces, then $\cg(G)$ is not planar by Proposition \ref{p:na}.

The only remaining case is when the odd faces of $G$ are pairwise adjacent, and there are at least four of them. On the other hand by Kuratowski's Theorem for the polyhedron $G^*$, there may be no more than four pairwise adjacent faces in $G$. That is to say, the only remaining case is $\od(G^*)\simeq K_4$. With the notation of Lemma \ref{le:main} and Definition \ref{def:0n}, at least one odd face of $G$ lies in $\overline{G_0}$, and at least one in $\overline{G_{n}}$, and these two faces are adjacent, implying $n=1$ i.e., there is exactly one vertex $x=x_1$ of degree $4$ in $G$. With the notation of Proposition \ref{p:na}, w.l.o.g.\ $\xi_{da},\xi_{ab}$ are odd faces. Then by Lemma \ref{le:loop}, $\xi_{bc},\xi_{cd}$ are even, contradicting $\od(G^*)\simeq K_4$. This completes the proof of Theorem \ref{thm:2}.
\end{proof}

\begin{proof}[Proof of Theorem \ref{thm:0}]
The result follows from Theorem \ref{thm:2}, Proposition \ref{p:p2cb}, and Theorem \ref{thm:1}.
\end{proof}

We end this section with a result of independent interest, on the connectivity of the congraph of any $2$-connected, non-bipartite graph (without assuming planarity).
\begin{prop}
	\label{p:2conn}
	Let $G$ be a $2$-connected, non-bipartite graph. Then $\cg(G)$ is $2$-connected.
\end{prop}
\begin{proof}
	To prove $2$-connectivity of $\cg(G)$, we will show that in $\cg(G)$ every pair of vertices lies on a cycle. Let $u,v\in V(G)$. If there exists in $G$ a cycle of odd length
	\[u=a_1,a_2,\dots,v=a_i,a_{i+1},\dots,a_n, \qquad \text{odd }n\geq 3,\]
	then in $\cg(G)$ we have the cycle \[a_1,a_3,\dots,a_{n},a_2,a_4,\dots,a_{n-1}\]
	containing $u,v$.
	
	Otherwise, as $G$ is $2$-connected and not bipartite, we start by finding an odd cycle $C$ in $G$. As $G$ is $2$-connected, there exist two internally disjoint paths connecting $u$ to two distinct vertices of $C$. Hence in fact $G$ has an odd cycle
	\[C': u=a_1,a_2,\dots,a_n,\qquad \text{odd }n\geq 3.\]
	By $2$-connectivity of $G$, there exist two internally disjoint paths connecting $v$ to two distinct vertices $a_i,a_j$ of $C'$, $1\leq i<j\leq n$. These two paths and the sub-path \[a_j,a_{j+1},\dots,a_n,a_1,a_2,\dots,a_i\]
	of $C'$ determine a cycle containing $u,v$,
	\[C'': c_1=a_j,c_2=a_{j+1},\dots,c_{n-j+1}=a_n,c_{n-j+2}=a_1=u,\dots,c_{n-j+1+i}=a_i,\dots,v,\dots,c_m,\]
	$m>n-j+1+i$. The reader may refer to Figure \ref{f:ac}. If the length $m$ of $C''$ is odd, we have found an odd cycle of $G$ containing $u,v$, and we may conclude as above. If $m$ is even, and the distance between $u,v$ on $C''$ is also even, then considering alternating vertices along $C''$ we obtain a cycle in $\cg(G)$ containing $u,v$.
	\begin{figure}[ht]
		\centering
		\includegraphics[width=5.cm]{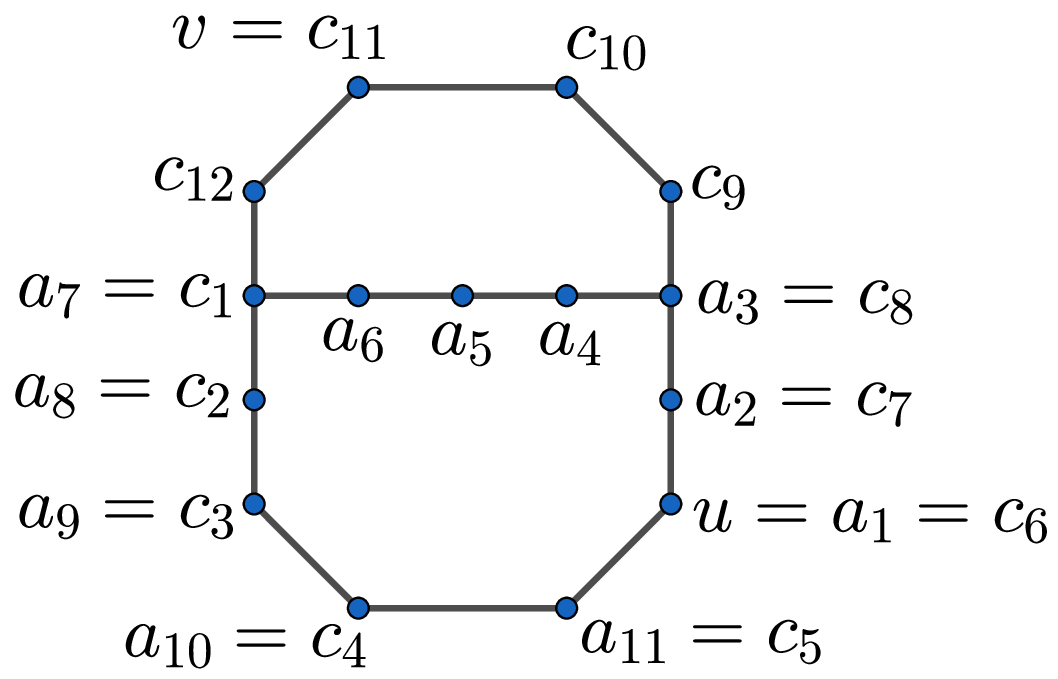}
		\caption{Illustration of Proposition \ref{p:2conn}, with $n=11$, $i=3$, $j=7$, and $m=12$. Here we have
		$C''': u,a_2,a_3,a_4,a_5,a_6,a_7,c_{12},v,c_{10},c_9,a_3,a_4,a_5,a_6,a_7,a_8,a_9,a_{10},a_{11}$, yielding the cycle $u,a_3,a_5,a_7,v,c_9,a_4,a_6,a_8,a_{10}$ in $\cg(G)$.}
		\label{f:ac}
	\end{figure}
	
	Lastly, if $m$ is even and the distance between $u,v$ on $C''$ is odd, we write $u=a_1=c_{n-j+2}$ and $v=c_{n-j+2h+1}$, with $n-j+1+i<n-j+2h+1\leq m$. We 
	consider the following circuit in $G$,
	\[C''': u=a_{1},a_{2},\dots,a_{j}=c_1,c_m,c_{m-1},\dots,v=c_{n-j+2h+1},\dots,c_{n-j+1+i}=a_i,a_{i+1},\dots,a_n.\]
	It has length $m+2(j-i)$, which is even, and the distance between $u,v$ along $C'''$ equals
	\[(j-1)+1+(m-(n-j+2h+1))=m-n+2j-2h-1\]
	which is also even. Then in $\cg(G)$ we can find a cycle containing $u,v$ by taking every second vertex along $C'''$.
\end{proof}

\section{Extremal results for the congraph of a polyhedron}
\label{sec:ext}

For every polyhedron $G$, we establish the maximal and minimal edges that $\cg(G)$ may have in the planar case, and find all extremal graphs. Following Theorem \ref{thm:0}, we consider separately the case of $G$ bipartite.

\begin{prop}
	\label{p:bd}
	Let $G$ be a bipartite polyhedron on $p$ vertices, other than the cube. Then
	\begin{equation}
		\label{eq:ecgb}
		2p\leq |E(\cg(G))|\leq 3p-12.
	\end{equation}
	The lower bound is attained if and only if $G$ is the $p/2$-gonal prism with $p$ a multiple of $4$. The upper bound is attained if and only if $G$ contains exactly six quadrangular faces, and the rest are hexagonal.
\end{prop}
\begin{proof}%[Proof of Proposition \ref{p:bd}]
	Let $G$ be a bipartite, polyhedral $(p,q)$-graph on $f$ faces, other than the cube. By Theorem \ref{thm:0}, $G$ is cubic. We will denote by $f_i$ the quantity of faces of length $i$ for every $i\geq 4$. Then using the fact the $G^*$ is maximal planar and the handshaking lemma for $G^*$ we may write
	\begin{equation}
		\label{eq:pqf}
		3p=2q=6f-12
	\end{equation}
	thus
	\[6\sum_{i\geq 2}f_{2i}-12=6f-12=2q=\sum_{i\geq 2}2if_{2i},\]
	whence
	\begin{equation}
		\label{eq:f4}
		f_4=6+\sum_{i\geq 4}(i-3)f_{2i}.
	\end{equation}
	
	Next, since $G$ is a cubic, bipartite polyhedron, then by Lemma \ref{le:deg} the degree of each $v\in V=V(G)=V(\cg(G))$ in the graph $\cg(G)$ equals six minus the number of quadrangular faces of $G$ containing $v$. Therefore,
	\[\sum_{v\in V}\deg_{\cg(G)}(v)=6p-4f_4,\]
	thus by the handshaking lemma for $\cg(G)$,
	\begin{equation}
		\label{eq:ecg}
		|E(\cg(G))|=3p-2f_4.
	\end{equation}
	From \eqref{eq:f4} we get $f_4\geq 6$, hence by \eqref{eq:ecg} we obtain the upper bound in \eqref{eq:ecgb}. This bound is attained if and only if $f_4=6$, that is to say, if and only if the faces of $G$ comprise $6$ quadrangles and $f-6$ hexagons.
	
	It remains to prove the lower bound in \eqref{eq:ecgb}. We assume by contradiction that $f_4\geq f-1$. Then the degree sequence of the maximal planar $(f,3f-6)$-graph $G^*$ is
	\[x,4,4,\dots,4\]
	hence via the handshaking lemma $x+4(f-1)=6f-12$ i.e., $x=2f-8$. On the other hand, clearly any vertex of $G^*$ may have at most $f-1$ neighbours i.e., $x\leq f-1$, hence $f\leq 7$. The case $f=6$ corresponds to $G$ being the cube, excluded by hypothesis, while $f=7$ is impossible. It follows that $f_4\leq f-2$. Now by \eqref{eq:pqf}, $f_4\leq p/2$ and substituting into \eqref{eq:ecg} yields $|E(\cg(G))|\geq 2p$ as desired.
	
	Let us investigate which polyhedra $G$ satisfy $|E(\cg(G))|=2p$. As we have seen, equivalently we need to find $G$ such that $f_4=f-2$. Writing the degree sequence of $G^*$,
	\[x,y,4,4,\dots,4\]
	and applying \eqref{eq:f4} we see that $x+y=2f_4=2f-4$. Since $x,y\leq f-1$, the only options are $x=f-1$ and $y=f-3$, or $x=y=f-2$. If the triangulation $G^*$ has a dominating vertex $u$ of degree $f-1$, then $G^*-u$ is a maximal outerplanar graph, and as such it contains at least two vertices of degree $2$ \cite[Lemma 3.3]{maffucci2025regularity}, contradiction. Similarly, the only planar graph of degree sequence $f-2,f-2,4,4,\dots,4$ is the $f-2$-gonal bipyramid, thus $G$ is the $p/2$-gonal prism, where $p/2$ is even.
\end{proof}
If $G$ is the $p/2$-gonal prism with $p$ a multiple of $4$, then $\cg(G)$ is the disjoint union of two $p/4$-gonal antiprisms. We also note that $G$ attains the upper bound in \eqref{eq:ecgb} if and only if $\cg(G)$ is the disjoint union of two maximal planar graphs.

\begin{prop}
	\label{p:bd2}
	Let $G$ be a polyhedron on $p$ vertices, such that $\cg(G)$ is a polyhedron. Then
	\begin{equation}
		\label{eq:ecgo}
		2p\leq |E(\cg(G))|\leq 3p-6.
	\end{equation}
	The lower bound is attained if and only if $G$ is the $p/2$-gonal prism with $p$ odd. Apart from the tetrahedron, the upper bound is attained if and only if $G$ contains exactly $3$ quadrangular faces. In this scenario, either $\od(G)\simeq\overline{K_2}$ and the two odd faces of $G$ are both triangular, or
	$\od(G)\simeq K_4$ and $G$ contains one triangular and three pentagonal faces.
\end{prop}
\begin{proof}[Proof of Proposition \ref{p:bd2}]
	By Theorem \ref{thm:0}, $G$ is cubic.	We have \eqref{eq:pqf}, thus
	\[6\sum_{i\geq 3}f_{i}-12=6f-12=2q=\sum_{i\geq 3}if_{i},\]
	whence
	\begin{equation}
		\label{eq:f3f4f5}
		3f_3+2f_4+f_5=12+\sum_{i\geq 7}(i-6)f_{i}.
	\end{equation}
	Moreover, recalling that in a cubic polyhedron other than the tetrahedron two triangular faces may never be adjacent, we deduce via Lemma \ref{le:deg} that the degree of each vertex in the graph $\cg(G)$ equals six minus the number of quadrangular faces of $G$ containing the vertex, thus as in \eqref{eq:ecg},
	\begin{equation}
		\label{eq:ecg2}
	|E(\cg(G))|=3p-2f_4.
	\end{equation}
	Then $\cg(G)$ is a maximal planar graph if and only if $f_4=3$. By \eqref{eq:f3f4f5} one has $3f_3+2f_4+f_5\geq 12$, which combined with $f_4=3$ yields
	\begin{equation}
		\label{eq:f3f5}
		3f_3+f_5\geq 6.
	\end{equation}
	On the other hand by Theorem \ref{thm:1}, $|\co(G)|\in\{2,4\}$. If $|\co(G)|=2$, then \eqref{eq:f3f5} implies $f_3=2$. If $|\co(G)|=4$, then by Theorem \ref{thm:1} we have $\od(G)\simeq K_4$ i.e., the odd faces of $G$ are pairwise adjacent. Since a cubic polyhedron $G\not\simeq K_4$ cannot contain a pair of adjacent triangular faces, then $f_3\leq 1$, which combined with \eqref{eq:f3f5} yields $f_5\geq 3$, so that ultimately in this case we have $f_3=1$ and $f_5=3$, and these are the only odd faces of $G$.
	
	To prove the lower bound in \eqref{eq:ecgo}, we wish to maximise $f_4$ in \eqref{eq:ecg2}. We note that $G$ contains at least two odd faces, thus the number of quadrangular faces will be at most $f-2$. As in Proposition \ref{p:bd}, this occurs if and only if $G^*$ is the $f-2$-gonal bipyramid. Since $G$ contains odd faces, $G^*$ has vertices of odd degree, thus $f-2$ is odd. We recall that $f-2=p/2$, so that $G$ is the $p/2$-gonal prism with $p/2$ odd.
\end{proof}

If $G$ is the $p/2$-gonal prism with $p/2$ odd, then $\cg(G)$ is the $p/2$-gonal antiprism.

\section{Constructions and extremal theorems for facecongraphs}
\label{sec:faceconpf}

In this section, $f_i$ is the number of $i$-gonal faces of a polyhedron $G$, and $q_{i,j}$ is the total number of common edges of $i$-gonal and $j$-gonal faces, $3\leq i\leq j$.

\subsection{Number of edges of the facecongraph}

We start with a result about the degree of a vertex in $\fcg(G)$.
\begin{lemma}
	\label{le:deg}
	Let $G$ be a polyhedron, and $v\in V(G)$. Then
	\[\deg_G(v)\leq\deg_{\fcg(G)}(v)\leq 2\deg_G(v).\]
	The lower bound is attained if and only if $v$ lies only on quadrangular faces. The upper bound is attained if and only if $v$ lies on no quadrangular faces and on no pair of adjacent triangular faces.
\end{lemma}
\begin{proof}
	Call $d=\deg_G(v)$,
	\[u_1,u_2,\dots,u_d\]
	the neighbours of $v$ in $G$ listed in cyclic order, and
	\begin{equation}
		\label{eq:al}
		\alpha_1,\alpha_2,\dots,\alpha_d
	\end{equation}
	the faces of $G$ containing $v$, where $\alpha_i$ contains $u_i$ and $u_{i+1}$, taking the indices modulo $d$. Further, we may suppose that the triangular faces containing $v$ in $G$ are exactly
	\begin{equation}
		\label{eq:trfa}
		\alpha_{i_1},\alpha_{i_1+1},\dots,\alpha_{i_1+j_1},\alpha_{i_2},\alpha_{i_2+1},\dots,\alpha_{i_2+j_2},\dots,\alpha_{i_n},\alpha_{i_n+1},\dots,\alpha_{i_n+j_n}
	\end{equation}
	with $n\geq 0$, $j_h\geq 0$ for every $1\leq h\leq n$, and \eqref{eq:trfa} are all distinct. Call $f_3(v)$, $f_4(v)$, and $f_{\geq 5}(v)$ the number of faces containing $v$ of lengths respectively $3$, $4$, and at least $5$. We record that
	\begin{equation}
		\label{eq:f3v}
		f_3(v)=\sum_{h=1}^{n}(1+j_h)=n+\sum_{h=1}^{n}j_h.
	\end{equation}
	
	In $\fcg(G)$, we count the neighbours of $v$ as follows. Letting $\alpha$ be an element of \eqref{eq:al}, we count one neighbour of $v$ in $\fcg(G)$ if $\alpha$ is of length $4$ (the vertex at distance $2$ from $v$ on $\alpha$); we count two neighbours of $v$ in $\fcg(G)$ if $\alpha$ is of length at least $5$ (the two vertices at distance $2$ from $v$ on $\alpha$); we count two neighbours of $v$ in $\fcg(G)$ if $\alpha$ is of length $3$ (the two vertices on $\alpha$ distinct from $v$). In this fashion, we have counted twice the neighbours $u$ of $v$ in $G$ such that $uv$ lies on two triangular faces $\alpha_i$ and $\alpha_{i+1}$. Therefore,
	\[\deg_{\fcg(G)}(v)=f_4(v)+2f_{\geq 5}(v)+2f_3(v)-\sum_{h=1}^{n}j_h=f_3(v)+f_4(v)+2f_{\geq 5}(v)+n\]
	where we have used \eqref{eq:f3v}. We deduce the lower bound
	\[\deg_{\fcg(G)}(v)\geq f_3(v)+f_4(v)+f_{\geq 5}(v)=d\]
	and the upper bound
	\[\deg_{\fcg(G)}(v)\leq 2f_3(v)+2f_4(v)+2f_{\geq 5}(v)=2d,\]
	as claimed. The lower bound is attained if and only if $f_{\geq 5}(v)+n=0$ i.e., $v$ lies only on quadrangular faces in $G$. The upper bound is attained if and only if $\sum_{h=1}^{n}j_h+f_4(v)=0$ i.e., $v$ lies on no quadrangular faces in $G$, and \eqref{eq:al} contains no consecutive pairs of triangular faces.% In particular, the bounds are sharp.
\end{proof}

Summing the bounds of Lemma \ref{le:deg} over all vertices and applying the handshaking lemma, we deduce the following.
\begin{cor}
	\label{cor:bds}
Let $G$ be a polyhedron. Then
\[|E(G)|\leq|E(\fcg(G))|\leq 2|E(G)|.\]
The lower bound is attained if and only if $G$ is a quadrangulation. The upper bound is attained if and only if $G$ contains no quadrangular faces and no pair of adjacent triangular faces.
\end{cor}

We now write a closed formula for $|E(\fcg(G))|$. 
\begin{lemma}
	\label{le:efcg}
	Let $G$ be a polyhedron. Then
	\begin{equation}
		\label{eq:pcE}
		|E(\fcg(G))|=2|E(G)|-2f_4-q_{3,3}.
	\end{equation}
\end{lemma}
\begin{proof}
	Within each face of $G$ of length $i\geq 5$, there are $i$ pairs of vertices with exactly one vertex between them, hence we count $i$ edges in $\fcg(G)$. For each quadrangular face of $G$, we count $2$ edges in $\fcg(G)$. We count an edge of $G$ for each edge of every triangular face of $G$, and deduct one for each pair of adjacent triangular faces, since we had counted these twice. Therefore, via the handshaking lemma for $G^*$,
	\[|E(\fcg(G))|=\sum_{i\geq 5}if_i+2f_4+3f_3-q_{3,3}=2|E(G)|-2f_4-q_{3,3}.\]
\end{proof}
Note that Lemma \ref{le:efcg} is consistent with the bounds in Corollary \ref{cor:bds}.

For cubic polyhedra $G$, from Lemma \ref{le:efcg} we recover \eqref{eq:ecg} and \eqref{eq:ecg2}, consistent with the fact that for cubic polyhedra, $\cg(G)$ and $\fcg(G)$ coincide.

\subsection{Proof of Theorem \ref{thm:maxpl} and Propositions \ref{p:min} and \ref{p:3456}}

We are ready to give an explicit lower bound for the edges of the facecongraph depending on the number of vertices.

\begin{proof}[Proof of Proposition \ref{p:min}]
Let $G$ be a $(p,q)$-polyhedron with $f$ faces. By Corollary \ref{cor:bds}, we have the lower bound $|E(\fcg(G))|\geq q$, with equality if and only if $G$ is a quadrangulation. For quadrangulations we have $q=2p-4$, hence the result.
\end{proof}

As an illustration, let $G$ be a $3$-connected quadrangulation of order $p$. Then $\fcg(G)=H\ \dot\cup \ H^*$, where $H$ is a planar, $2$-connected, $3$-edge-connected graph of order $p_H$ satisfying \eqref{eq:rad} \[G\simeq\calr(H)\simeq\calr(H^*).\]By Proposition \ref{p:min}, the total number of edges in $H$ and $H^*$ is $2p-4$, and since they are the dual of one another, each has $p-2$ edges. Further, since $H,H^*$ are planar, calling $p_H$ the order of $H$ we obtain respectively $p-2\leq 3p_H-6$ and $p-2\leq 3(p-p_H)-6$, whence
\[\frac{p+4}{3}\leq p_H\leq\frac{2p-4}{3}.\]
The lower/upper bound is attained if and only if $H$ is a triangulation and $H^*$ is cubic/vice versa.

We now turn to polyhedra $G$ such that $\fcg(G)$ is planar and has the maximal number of edges. As mentioned in the introduction, the maximal value of $|E(\fcg(G))|$ is attained by all the maximal planar graphs $G$ however, there are polyhedra $G$ other than the maximal planar graphs such that $\fcg(G)$ is maximal planar. A few examples were given in Figure \ref{f:mpl}. Let us prove that if $G$ is a polyhedron such that $\fcg(G)$ is maximal planar, then every face of $G$ has length at most $6$.

\begin{proof}[Proof of Proposition \ref{p:3456}]
Since $\fcg(G)$ is maximal planar, if $v,w\in V(G)$ satisfy $vw\in E(\fcg(G))$, then there exist $x\neq y\in V(G)$ such that
\[vx,\ xw,\ vy,\ yw\in E(\fcg(G)).\]

To prove the present proposition, we will argue by contradiction. Let
\[\varphi=[u_1,u_2,\dots,u_n], \qquad n\geq 7\]
be a face of $G$, taking the indices modulo $n$. For $1\leq i\leq n$, since $u_i,u_{i+2}$ lie on the same face of $G$ with exactly one vertex between them, then $u_iu_{i+2}\in E(\fcg(G))$. As argued above, there exist $x_i\neq y_i\in V(G)$ such that
\[u_ix_i,\ x_iu_{i+2},\ u_iy_i,\ y_iu_{i+2}\in E(\fcg(G)),\]
as in Figure \ref{f:7a}. Hence by definition of facecongraph, there exist faces
\[\alpha_i,\beta_i,\gamma_{i},\delta_i\]
such that $u_i,x_i$ lie on $\alpha_i$; $x_i,u_{i+2}$ lie on $\beta_{i}$; $u_i,y_i$ lie on $\gamma_i$; $y_i,u_{i+2}$ lie on $\delta_{i}$ with exactly one vertex between them in each case.
\begin{figure}[ht]
	\centering
	\begin{subfigure}{0.48\textwidth}
		\centering		\includegraphics[width=6.cm]{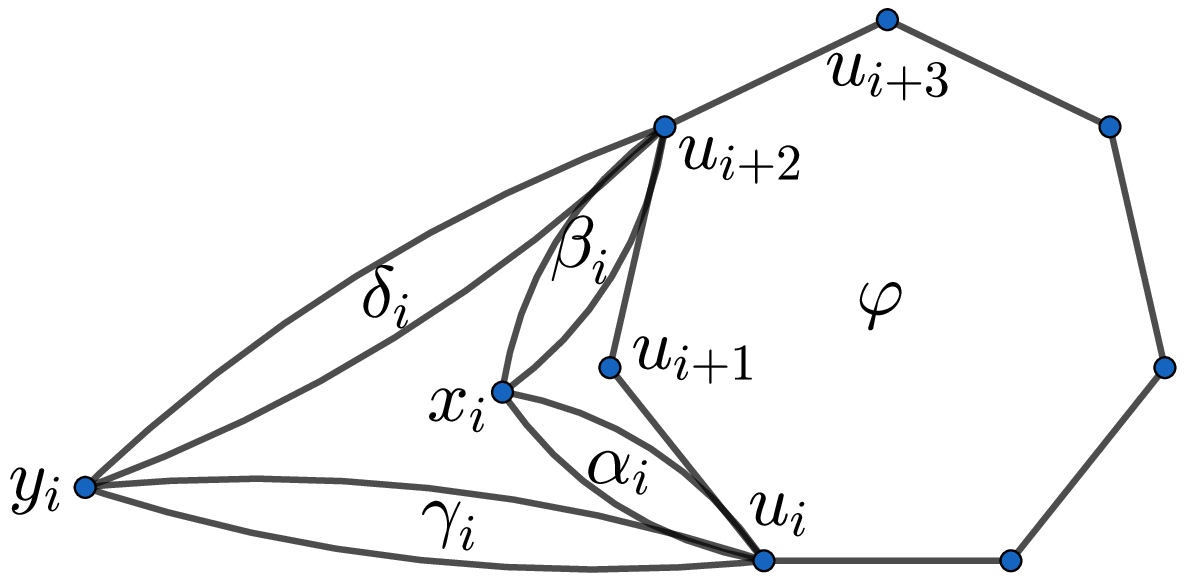}
		\caption{$G$ has a face $\varphi$ of length at least $7$.}
		\label{f:7a}
	\end{subfigure}
	\begin{subfigure}{0.48\textwidth}
		\centering
		\includegraphics[width=4.cm]{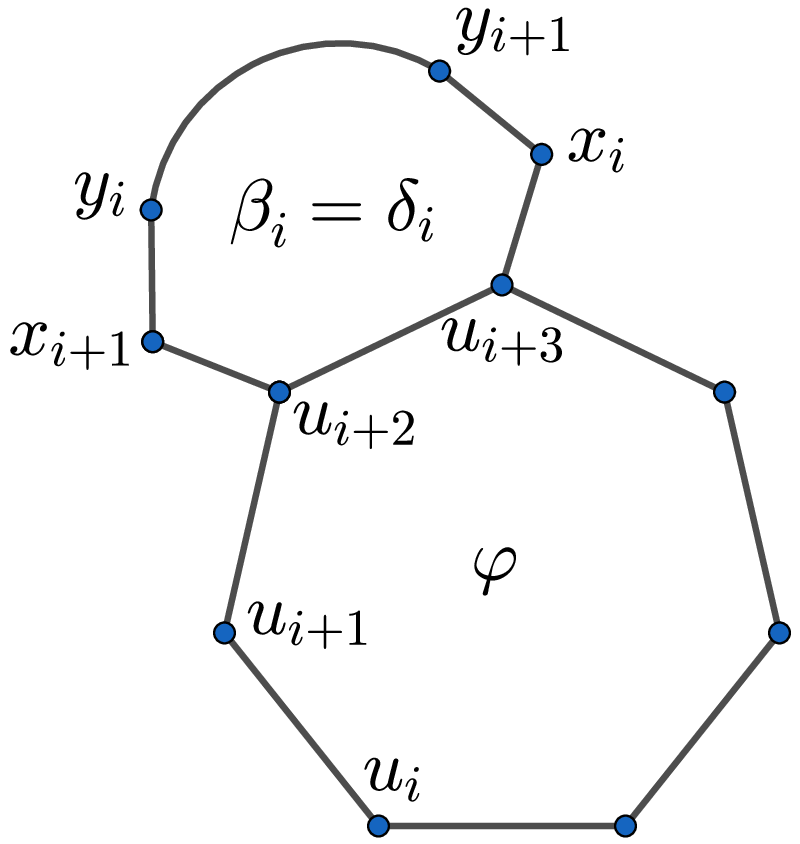}
		\caption{$\beta_i=\delta_i$.}
		\label{f:7b}
	\end{subfigure}
	\caption{Proposition \ref{p:3456}.}
	\label{f:7}
\end{figure}

We claim that $\alpha_i,\beta_i,\gamma_i,\delta_i\neq\varphi$. Indeed, if $\alpha_i=\varphi$, say, then $u_i,x_i$ lie on $\varphi$ with exactly one vertex between them, and as $x_i\neq u_{i+2}$, the only possibility is $x_i=u_{i-2}$. Thus $\beta_i$ contains $u_{i-2}$ and $u_{i+2}$. Therefore, either $\beta_i=\varphi$ so that $u_{i-2},u_{i+2}$ lie on $\varphi$ with exactly one vertex between them, impossible as $n\neq 6$, or $\beta_i\neq\varphi$ and $\beta_i,\varphi$ have two common vertices $u_{i-2},u_{i+2}$, thus as $G$ is a polyhedron, $u_{i-2}u_{i+2}\in E(G)$, impossible as $n\neq 5$. We deduce that indeed $\alpha_i,\beta_i,\gamma_i,\delta_i\neq\varphi$ for every $1\leq i\leq n$.

For $1\leq i\leq n$, by planarity of $G$ we have
\[\beta_{i+1},\delta_{i+1}\in\{\alpha_i,\beta_i,\gamma_i,\delta_i\}.\]
Moreover, $\beta_{i+1},\delta_{i+1}\not\in\{\alpha_i,\gamma_i\}$ e.g., if $\alpha_i=\beta_{i+1}$, then this face contains $u_i,u_{i+3}$, impossible. It follows that 
\[\beta_{i+1},\delta_{i+1}\in\{\beta_i,\delta_i\}.\]

Next, for $1\leq i\leq n$, we claim that $\beta_i\neq\delta_i$. Indeed, if $\beta_i=\delta_i$, then $\beta_i=\delta_i=\beta_{i+1}=\delta_{i+1}$, and this face, distinct from $\varphi$, contains $u_{i+2}$ and $u_{i+3}$, so that $\varphi$ and $\beta_i$ are adjacent. Moreover, $\beta_i$ contains $u_{i+2},u_{i+3},x_i,y_i,x_{i+1},y_{i+1}$ where each of the pairs $u_{i+2},x_i$; $u_{i+2},y_i$; $u_{i+3},x_{i+1}$; $u_{i+3},y_{i+1}$ have exactly one vertex between them on $\beta_i$. Recalling that $x_i\neq y_i$ for every $1\leq i\leq n$, we deduce that
\[y_i,x_{i+1},u_{i+2},u_{i+3},x_i,y_{i+1}\]
appear consecutively along $\beta_i$ in this order, as in Figure \ref{f:7b}. Recalling that $u_i,x_i$ lie on $\alpha_i$ and $u_i,y_i$ lie on $\gamma_i$ for every $1\leq i\leq n$, then by planarity of $G$ we deduce that 
$\alpha_i=\gamma_i=\alpha_{i+1}=\gamma_{i+1}$. Then $\alpha_i,\beta_i$ have four common vertices $x_i,y_i,x_{i+1},y_{i+1}$ hence in fact $\alpha_i=\beta_i$, but then $\alpha_i$ contains $u_i,u_{i+1},u_{i+2},u_{i+3}$ and is distinct from $\varphi$, impossible. We have shown that $\beta_i\neq\delta_i$ (and likewise $\alpha_i\neq\gamma_i$) for every $1\leq i\leq n$.

It follows that for every $1\leq i\leq n$ we have w.l.o.g., $\beta_i=\beta_{i+1}$ and $\delta_i=\delta_{i+1}$. Then $\beta_1=\beta_2=\dots=\beta_n$ is a face distinct from $\varphi$ containing $u_1,u_2,\dots,u_n$, contradiction.
\end{proof}
On the other hand, there exist polyhedra $G$ containing a face of arbitrarily long length and such that $\fcg(G)$ is planar, but not maximal planar. For instance, every $n$-gonal prism, $n\geq 7$, where by $3$-regularity, $\fcg(G)=\cg(G)$ and as we have seen in Section \ref{sec:ext}, $\fcg(G)$ is either an antiprism or the disjoint union of two antiprisms.

It is probably difficult to fully characterise the polyhedra $G$ such that $\fcg(G)$ is maximal planar. The final goal of this paper is to prove the remarkable characterisation and construction of the class with face length at most $4$, stated in Theorem \ref{thm:maxpl}. We begin with a couple of preparatory results.

\begin{lemma}
	\label{le:qij}
	Let $G$ be a $(p,q)$-polyhedron with $f$ faces. If $\fcg(G)$ is maximal planar, then
	\begin{equation}
		\label{eq:qij}
	3f=q+q_{3,3}+2f_4.
	\end{equation}
\end{lemma}
\begin{proof}
By Lemma \ref{le:efcg},
\[2q=|E(\fcg(G))|+q_{3,3}+2f_4.\]
Since $\fcg(G)$ is maximal planar, we deduce
\[2q=3p-6+q_{3,3}+2f_4.\]
Using Euler's formula, we obtain \eqref{eq:qij}.
\end{proof}

\begin{lemma}
	\label{le:2sq}
Let $G$ be a polyhedron of maximal face length $4$, such that $\fcg(G)$ is maximal planar. Then each quadrangular face of $G$ is adjacent to exactly two quadrangular and two triangular faces.
\end{lemma}
\begin{proof}
Let $G$ be a $(p,q)$-polyhedral graph with $f$ faces, $f_i$ be the number of $i$-gonal faces, and $q_{i,j}$ the total number of common edges of $i$-gonal and $j$-gonal faces, $3\leq i\leq j\leq 4$. By Lemma \ref{le:qij} one has
$3(f_3+f_4)=q+q_{3,3}+2f_4$,
whence
\begin{equation}
	\label{eq:comb1}
3f_3+f_4=q+q_{3,3}.
\end{equation}
On the other hand, the double-counting argument at the beginning of the proof of Proposition \ref{p:min} yields
\begin{equation}
	\label{eq:comb2}
3f_3=2q_{3,3}+q_{3,4}.
\end{equation}
Combining \eqref{eq:comb1}, \eqref{eq:comb2}, and $q=q_{3,3}+q_{3,4}+q_{4,4}$, we obtain
\begin{equation}
	\label{eq:q44}
f_4=q_{4,4}.
\end{equation}
Double-counting the edges of $G$ lying on quadrangular faces we obtain $4f_4=2q_{4,4}+q_{3,4}$, thus ultimately
\begin{equation}
	\label{eq:q34}
2f_4=q_{3,4}.
\end{equation}

Next, we employ the notation $q_{3,4}(\alpha)$ for the number of edges that a quadrangular face $\alpha$ shares with triangular faces in $G$. Then
\[2f_4=q_{3,4}=\sum_{\alpha\in F_4}q_{3,4}(\alpha),\]
where $F_4$ is the set of quadrangular faces of $G$. If we can show that in $G$ each quadrangular face is adjacent to at least two triangular faces, then
\[2f_4=\sum_{\alpha\in F_4}q_{3,4}(\alpha)\geq \sum_{\alpha\in F_4}2=2f_4,\]
implying that in fact in $G$ each quadrangular face is adjacent to exactly two quadrangular and two triangular faces, as claimed.

By contradiction, in the face $[a,b,c,d]$, let each of $ab$, $bc$, $cd$ lie on two quadrangular faces, as in Figure \ref{f:2sq1}. Since $ac\in E(\fcg(G))$, and $\fcg(G)$ is maximal planar, then there exist vertices $w,y$ such that
\begin{equation}
	\label{eq:wy}
wa,\ wc,\ ya,\ yc\in\ E(\fcg(G)).
\end{equation}
Likewise, there exist vertices $x,z$ such that
\begin{equation}
	\label{eq:xz}
xb,\ xd,\ zb,\ zd\in\ E(\fcg(G)).
\end{equation}
By definition of facecongraph, in $G$ each of \eqref{eq:wy} and \eqref{eq:xz} is either an edge of a triangular face or a diagonal of a quadrangular face. In particular,
\[\{a,b,c,d\}\cap\{w,x,y,z\}=\emptyset.\]
In $G$, any edge corresponding to \eqref{eq:wy}, \eqref{eq:xz}, and any quadrangular face with a diagonal corresponding to \eqref{eq:wy}, \eqref{eq:xz} must of course be drawn outside of the face $[a,b,c,d]$. By planarity of $G$, this is only possible if (w.l.o.g., up to permuting $x,z$) all of
\[[a,b,w,z],\ [b,c,x,w],\ [c,d,y,x],\ [d,a,z,y]\]
are faces in $G$ (Figure \ref{f:2sq2}). Hence $[a,b,c,d]$ is adjacent to these four quadrangular faces, so that
\[N_{\fcg(G)}(a)=\{c,w,y\}\qquad\text{and}\qquad N_{\fcg(G)}(c)=\{a,w,y\}.\]
This contradicts the $3$-connectivity of $\fcg(G)$.
\begin{figure}[ht]
	\centering
	\begin{subfigure}{0.45\textwidth}
		\centering		\includegraphics[width=3.75cm]{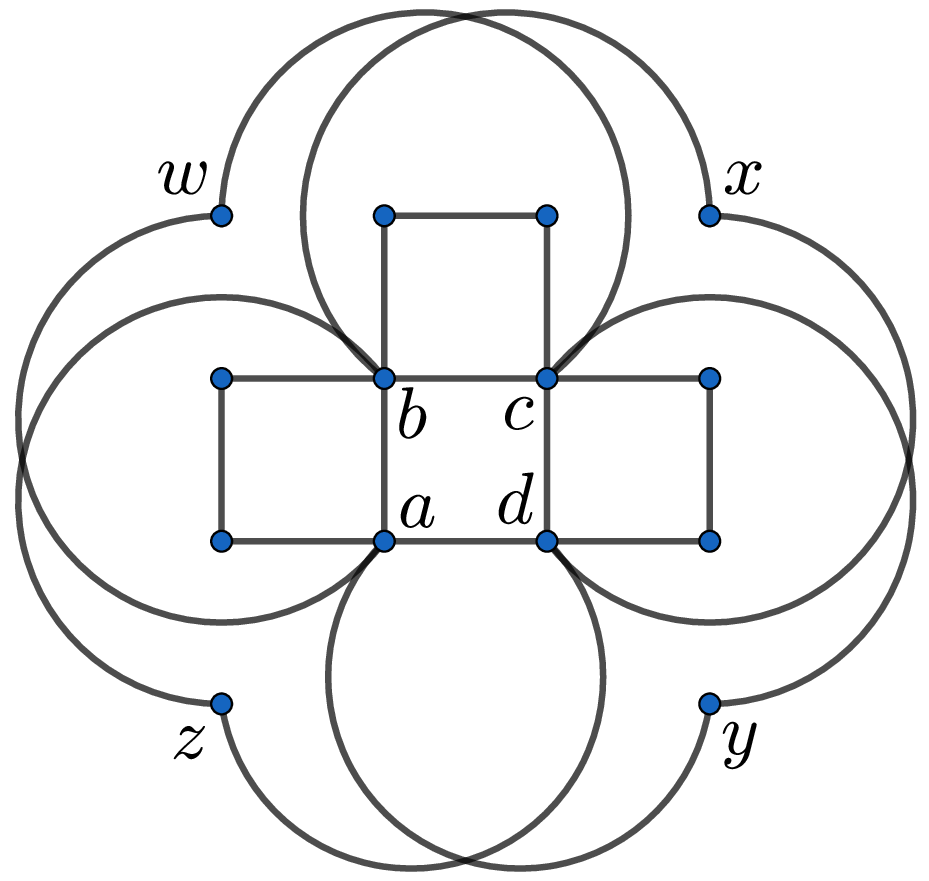}
		\caption{$[a,b,c,d]$ is adjacent to three quadrangular faces.}
		\label{f:2sq1}
	\end{subfigure}
	\hspace{0.5cm}
	\begin{subfigure}{0.45\textwidth}
		\centering
		\includegraphics[width=3.25cm]{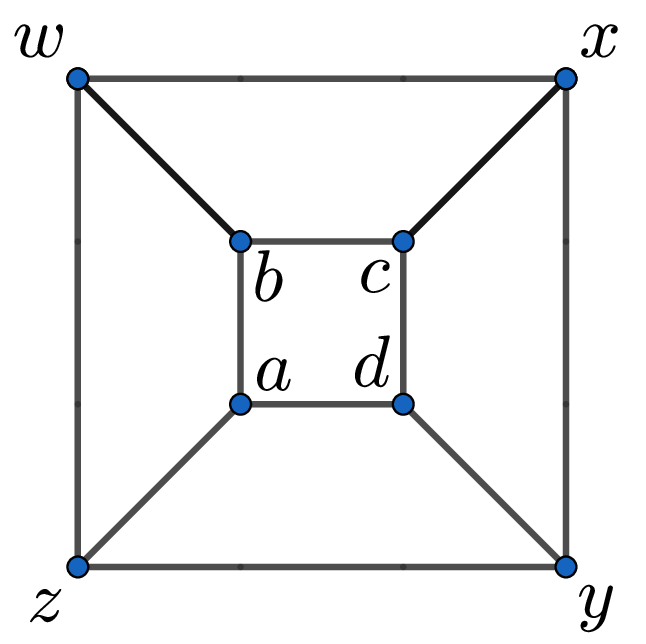}
		\caption{Hence $[a,b,w,z]$, $[b,c,x,w]$, $[c,d,y,x]$, $[d,a,z,y]$ are faces in $G$.}
		\label{f:2sq2}
	\end{subfigure}
	\caption{Lemma \ref{le:2sq}.}
	\label{f:2sq}
\end{figure}

\end{proof}

The following result is at the heart of the proof of Theorem \ref{thm:maxpl}.
\begin{prop}
	\label{p:3sq}
	Let $G$ be a polyhedron of maximal face length $4$, such that $\fcg(G)$ is maximal planar. Then each connected component of the subgraph of $G$ generated by the vertices lying on at least one quadrangular face is isomorphic to one of the two graphs in Figure \ref{f:tsab}. Moreover, in Figure \ref{f:tsb}, $[w_2,w_4,w_6]$ is a face of $G$.
	\begin{figure}[h!]
		\centering
		\begin{subfigure}{0.48\textwidth}
			\centering		\includegraphics[width=3.5cm]{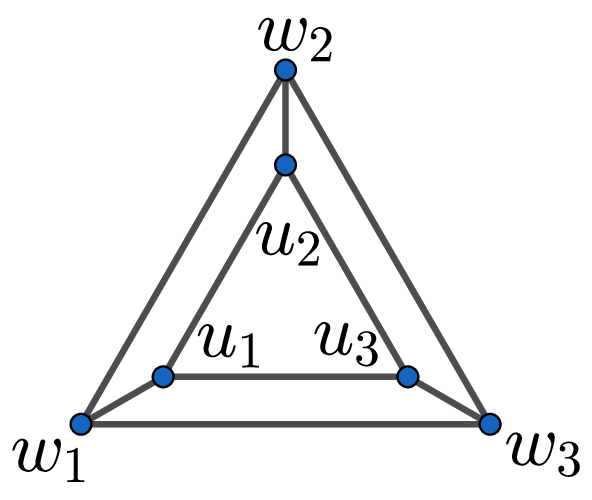}
			\caption{First scenario.}
			\label{f:tsa}
		\end{subfigure}
		\begin{subfigure}{0.48\textwidth}
			\centering
			\includegraphics[width=3.25cm]{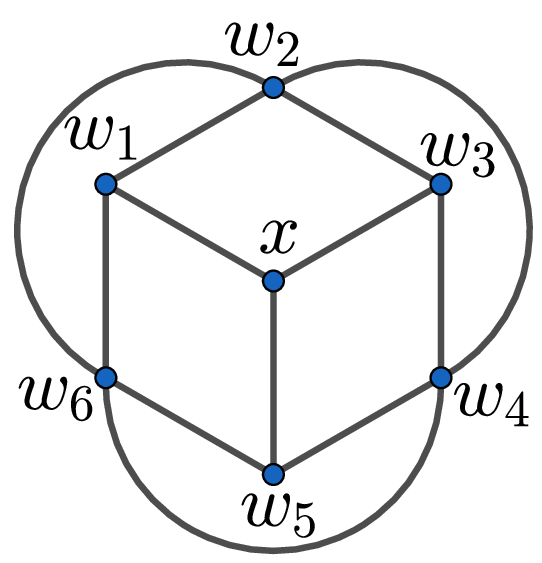}
			\caption{Second scenario.}
			\label{f:tsb}
		\end{subfigure}
		\caption{Proposition \ref{p:3sq}.}
		\label{f:tsab}
	\end{figure}
	
\end{prop}
\begin{proof}
Let $G$ be a polyhedron of maximal face length $4$, such that $\fcg(G)$ is maximal planar. By Lemma \ref{le:2sq}, in $G$ every quadrangular face is adjacent to exactly two quadrangular faces. Thereby, the quadrangular faces of $G$ correspond in $G^*$ to vertices that induce in $G^*$ a subgraph $H$ that is $2$-regular i.e., a disjoint union of cycles. For a fixed connected component of $H$, we denote the corresponding quadrangular faces of $G$ by
\begin{equation}
	\label{eq:fac}
	\alpha_1,\alpha_2,\dots,\alpha_A, \qquad A\geq 3,
\end{equation}
where for $1\leq i\leq A$, the face $\alpha_i$ is adjacent to $\alpha_{i+1}$ and $\alpha_{i-1}$, with the indices taken modulo $A$. The edges of the faces \eqref{eq:fac} that are external to all of \eqref{eq:fac} form a cycle
\begin{equation}
	\label{eq:w}
	C_w: w_1,w_2,\dots,w_W\qquad W\geq 3
\end{equation}
with indices taken modulo $W$. If there are any edges of \eqref{eq:fac} that are internal to all of \eqref{eq:fac}, then they form a cycle that we will call
\begin{equation}
	\label{eq:u}
	C_u: u_1,u_2,\dots,u_U, \qquad 3\leq U\leq W
\end{equation}
with indices taken modulo $U$, otherwise there exists a vertex $x$ lying on all of \eqref{eq:fac}.

For the moment we will suppose that $C_u$ exists. Note that
\[U+W=2A\qquad\text{thus}\qquad 3\leq U\leq A\leq W\leq 2A-3.\]
We record that each of $u_iu_{i+1}$, $1\leq i\leq U$ and each of $w_iw_{i+1}$, $1\leq i\leq W$ is an edge also in $\fcg(G)$, since it lies on one quadrangular and one triangular face in $G$. Thus $C_u,C_w$ are subgraphs of $\fcg(G)$ as well. Call 
\[G_u\qquad\text{and}\qquad G_w\]
respectively the subgraph of $G$ induced by the vertices inside and on $C_u$, and the subgraph of $G$ induced by the vertices outside and on $C_w$ (Figure \ref{f:ts1}). As the edges of $\fcg(G)$ correspond in $G$ either to edges of triangular faces or diagonals of quadrangular faces, then by planarity of $G$, in $\fcg(G)$ no vertex of $G_u$ may be adjacent to a vertex outside of $C_w$, and likewise no vertex of $G_w$ may be adjacent to a vertex inside of $C_u$. It follows that $\fcg(G_u)$ is a planar subgraph of $\fcg(G)$, one of its regions is delimited by $C_u$, and all other regions are triangular. Likewise, $\fcg(G_w)$ is a planar subgraph of $\fcg(G)$, one of its regions is delimited by $C_w$, and all other regions are triangular. Thereby,
\[E(\fcg(G))=E(\fcg(G_u))\ \dot\cup\  E(\fcg(G_w))\ \dot\cup\ E',\]
where $E'$ is a subset of the set of diagonals of the quadrangular faces \eqref{eq:fac} of $G$. In fact, as we know that $\fcg(G)$ is maximal planar, that $C_u$ has length $U$, that $C_w$ has length $W$, and that there are $2A=U+W$ diagonals of the faces \eqref{eq:fac}, then $E'$ is exactly the set of diagonals of \eqref{eq:fac}.
\begin{figure}[ht]
	\centering
	\includegraphics[width=4.cm]{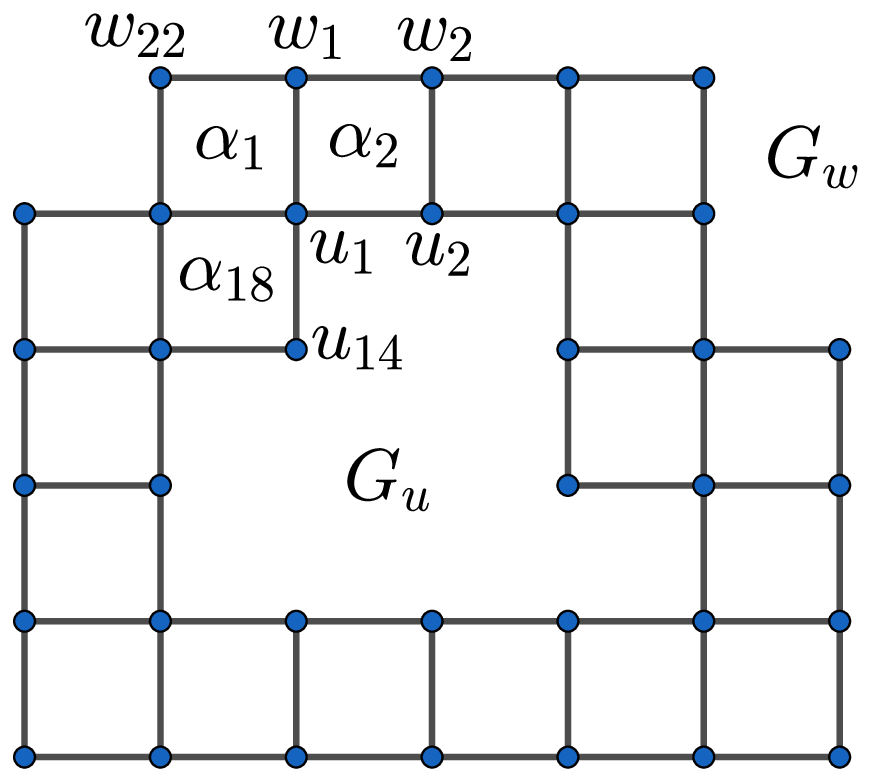}
	\caption{Illustration of the subgraphs $G_u$ and $G_w$ of $G$. Here $A=18$, $W=22$, and $U=14$. Note that $U+W=2A$.}
	\label{f:ts1}
\end{figure}

Each of the quadrangular faces \eqref{eq:fac} has either one, two, or three vertices in $C_u$. Suppose that one of these faces is
\[[u_1,u_2,u_3,w_1].\]
Then on one hand $u_1u_3\in E'$. As $E'\cap E(\fcg(G_u))=\emptyset$, this edge is external to $C_u$, thus by planarity of $\fcg(G)$, the vertex $u_2$ is not adjacent to any vertex of $C_w$. On the other hand $u_2w_1\in E'\subset E(\fcg(G))$, contradiction. Hence none of \eqref{eq:fac} contain three vertices of $C_u$.

As we are assuming that we have the cycle $C_u$ \eqref{eq:u} in $G$, then there exists an element of \eqref{eq:fac} containing two vertices of $C_u$, say
\[\alpha_2=[u_2,u_3,w_3,w_2],\]
as in Figure \ref{f:ts3}. Since none of \eqref{eq:fac} contain three vertices of $C_u$ then w.l.o.g., $w_1,w_2,u_2$ are consecutive vertices on $\alpha_1$, and $w_4,w_3,u_3$ are consecutive on $\alpha_3$. If $W\geq 4$, then $w_1\neq w_4$ and we reach a contradiction as $\fcg(G)$ contains a $K(3,3)$-minor, with partition given by
\[\{\{u_3,w_1,w_3\},\{u_2,w_2,w_4\}\}.\]
To see this, we consider in $\fcg(G)$ the edges
\[u_2u_3,\ w_1w_2,\ w_2w_3,\ w_3w_4,\ u_2w_1,\ u_2w_3,\ u_3w_2,\ u_3w_4\]
and the path
\[w_4,w_{5},\dots,w_W,w_1.\]
It follows that $W=3$ so that in fact $U=A=W=3$ i.e., we are in the scenario of Figure \ref{f:tsa}.
\begin{figure}[ht]
	\centering
	\includegraphics[width=3.cm]{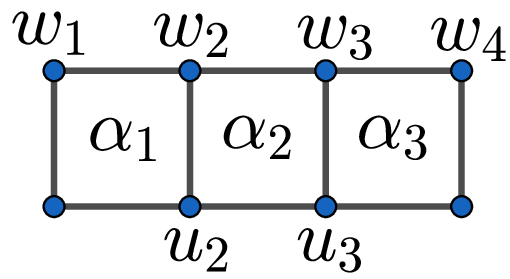}
	\caption{$\alpha_2=[u_2,u_3,w_3,w_2]$.}
	\label{f:ts3}
\end{figure}

It remains to inspect what happens for connected components of $H$ such that there is no cycle $C_u$ i.e., there is a common vertex $x$ to all the faces \eqref{eq:fac}. We may write w.l.o.g.
\[\alpha_i=[x,w_{2i-1},w_{2i},w_{2i+1}],\qquad 1\leq i\leq A=W/2\]
for even $W\geq 6$, as in Figure \ref{f:ts4}. Accordingly, $\fcg(G)$ has a subgraph as depicted in Figure \ref{f:ts5}.
\begin{figure}[ht]
	\centering
	\begin{subfigure}{0.48\textwidth}
		\centering		\includegraphics[width=3.5cm]{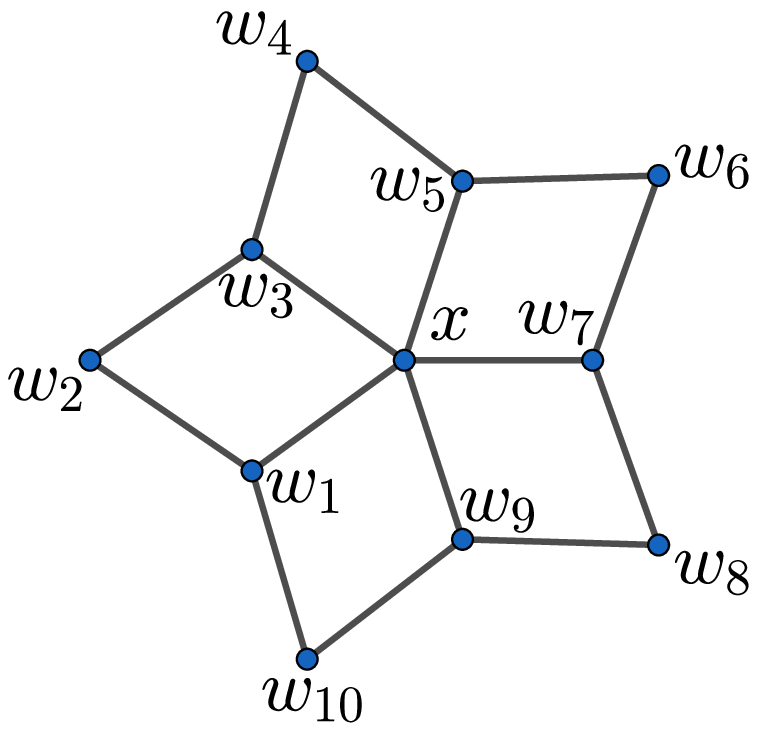}
		\caption{Subgraph of $G$.}
		\label{f:ts4}
	\end{subfigure}
	\begin{subfigure}{0.48\textwidth}
		\centering
		\includegraphics[width=3.5cm]{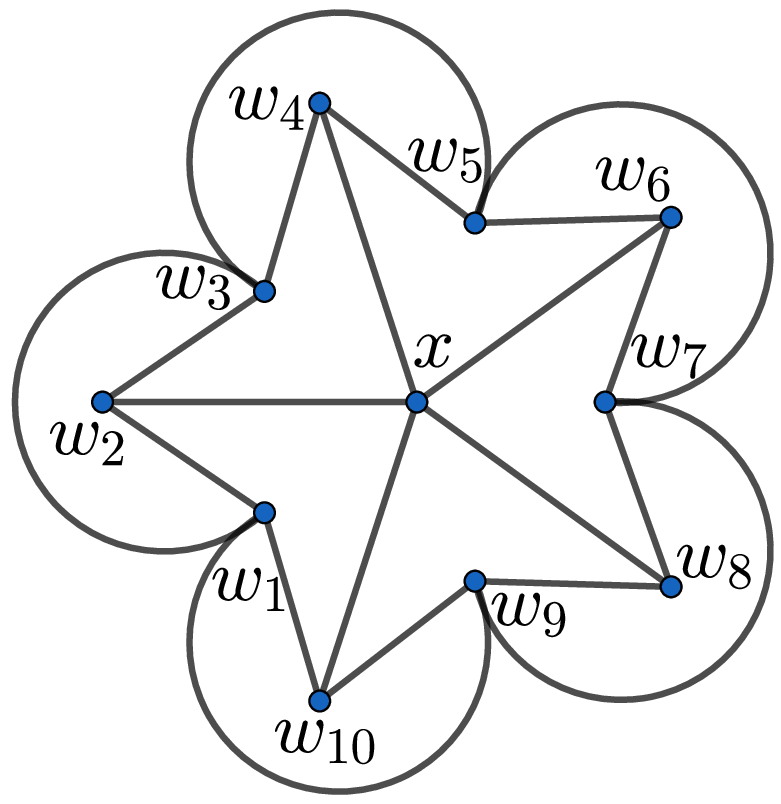}
		\caption{Subgraph of $\fcg(G)$.}
		\label{f:ts5}
	\end{subfigure}
	\caption{Case when $\alpha_1,\alpha_2,\dots,\alpha_A$ have a common vertex $x$, with $A=5$.}
	\label{f:ts45}
\end{figure}

In particular, $x,w_{2i},w_{2i+1},w_{2i+2}$ is a $4$-cycle in $\fcg(G)$ for every $1\leq i\leq A$. Note also that in $G$, the vertex $x$ has degree $A$ and lies on $A$ quadrangular faces, thus via Lemma \ref{le:deg},
\[\deg_{\fcg(G)}(x)=A.\]
Since $\fcg(G)$ is maximal planar, we deduce that $w_{2i}w_{2i+2}\in E(\fcg(G))$ for every $1\leq i\leq A$. Now each edge of $\fcg(G)$ corresponds in $G$ either to an edge of a triangular face, or a diagonal of a quadrangular face.

We claim that $w_{2i}w_{2i+2}\in E(G)$ for every $1\leq i\leq A$. By contradiction, suppose that $w_2w_4\not\in E(G)$ (we have chosen to fix $i=1$ to simplify the exposition). Then
\[[w_2,y,w_4,z]\]
is a quadrangular face of $G$, where by Lemma \ref{le:2sq} one has $y,z\not\in\{w_1,w_2,\dots,w_W\}$. The neighbours of $w_3$ in $\fcg(G)$ include $w_1,w_2,w_4,w_5$, and at least one more vertex $t$, as we have assumed that $w_2w_4\not\in E(G)$. The vertex $t$ lies inside of the quadrilateral $y,w_2,w_3,w_4$ in Figure \ref{f:ts6} (possibly $t=z$). Now since $\fcg(G)$ is maximal planar, the neighbours of each of its vertices $v$ induce a cycle $L_v$ of length equal to the degree of $v$ (that is to say, the neighbourhood graph of each $v$ is a pyramid centred at $v$). Considering $L_{w_3}$, in $\fcg(G)$ either $t$ is adjacent to one of $w_1,w_5$, impossible by planarity of $G$, or $w_1,w_2,w_4,w_5$ are consecutive in some order along $L_{w_3}$. Again by planarity of $G$, we deduce that $w_1w_5\in E(\fcg(G))$. On the other hand, $w_4w_6\in E(\fcg(G))$, contradicting the planarity of $G$. This reasoning for $i=1$ generalises to every $1\leq i\leq A$, hence indeed $w_{2i}w_{2i+2}\in E(\fcg(G))$ for every $1\leq i\leq A$, as in Figure \ref{f:ts78}.
\begin{figure}[ht]
	\centering
	\includegraphics[width=3.5cm]{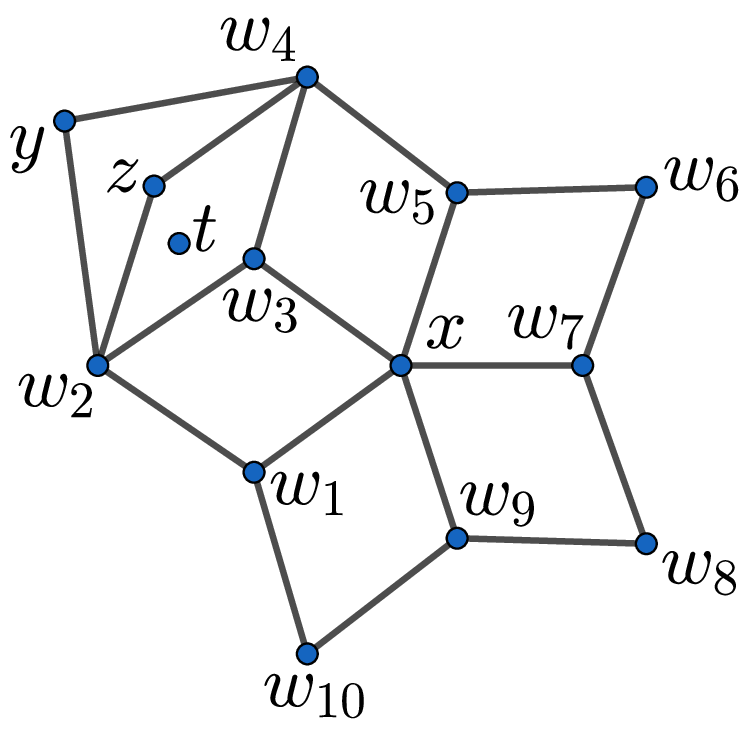}
	\caption{$w_2w_4\not\in E(G)$.}
	\label{f:ts6}
\end{figure}

\begin{figure}[ht]
	\centering
	\begin{subfigure}{0.48\textwidth}
		\centering		\includegraphics[width=3.5cm]{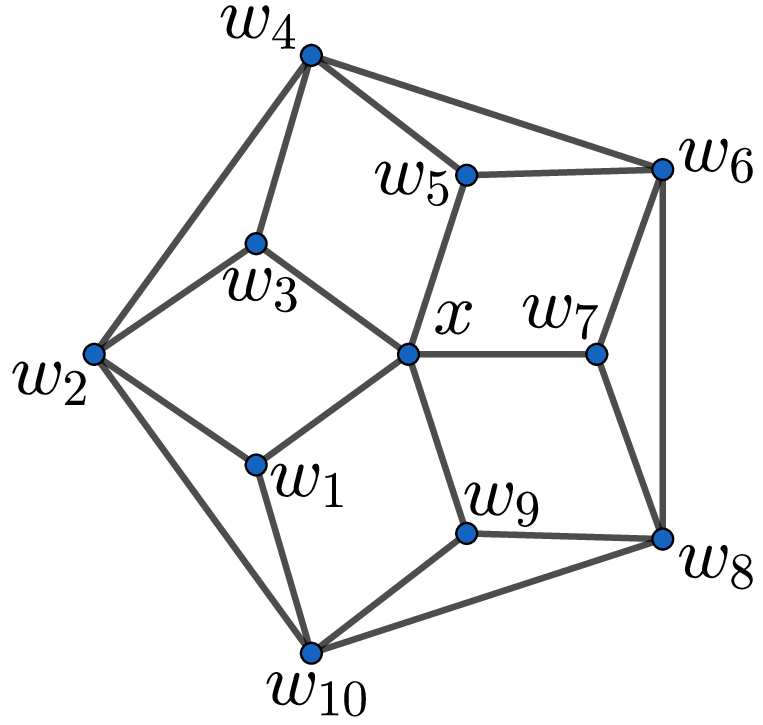}
		\caption{Subgraph of $G$.}
		\label{f:ts7}
	\end{subfigure}
	\begin{subfigure}{0.48\textwidth}
		\centering
		\includegraphics[width=3.5cm]{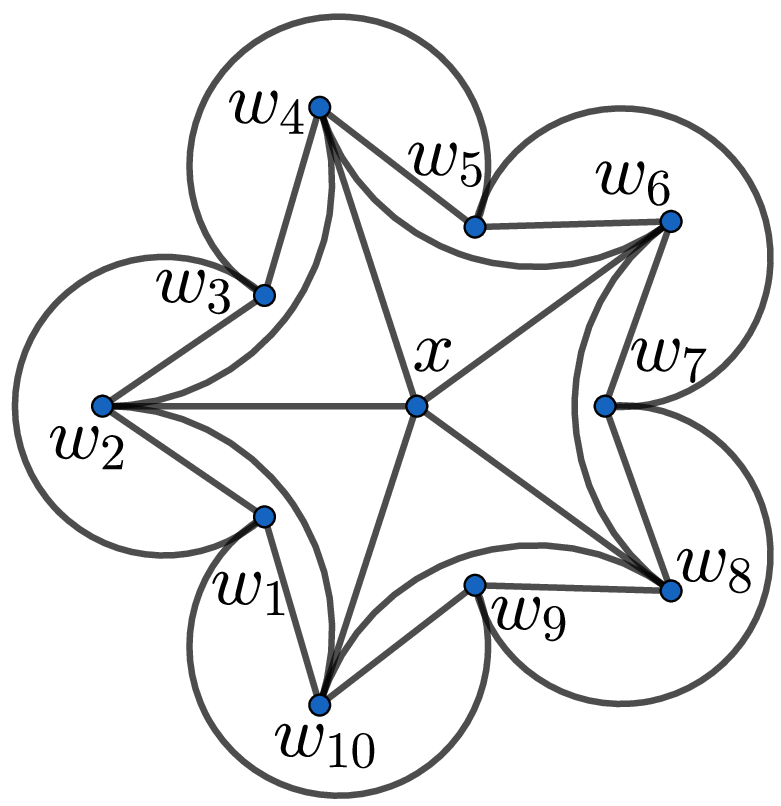}
		\caption{Subgraph of $\fcg(G)$.}
		\label{f:ts8}
	\end{subfigure}
	\caption{If $\alpha_1,\alpha_2,\dots,\alpha_A$ have a common vertex $x$, then $w_{2i}w_{2i+2}\in E(G)$ for every $1\leq i\leq A$.}
	\label{f:ts78}
\end{figure}

Then by planarity of $\fcg(G)$, there is no face of $\fcg(G)$ containing three of $w_1,w_2,\dots,w_W$ with even indices. It follows that in $G$, outside the cycle $w_2,w_4,\dots,w_W$, we cannot insert any vertices or edges without contradicting the planarity of $\fcg(G)$ or the $3$-connectivity of $G$. Hence
\[[w_2,w_4,\dots,w_W]\]
is a face in $G$. By assumption, the maximal face length of $G$ is $4$, thus $W\leq 8$. Moreover, for every $1\leq i\leq A$, in $G$ the edge $w_{2i}w_{2i+2}$ lies on at least one triangular face. Thus if $[w_2,w_4,\dots,w_W]$ is quadrangular, then it is adjacent only to triangular faces, contradicting Lemma \ref{le:2sq}. Therefore $W=6$, yielding the scenario in Figure \ref{f:tsb}.
\end{proof}

An illustration of Proposition \ref{p:3sq} is given by the leftmost graph in Figure \ref{f:mpl}. It contains three quadrangular faces arranged as in Figure \ref{f:tsa} and three arranged as in Figure \ref{f:tsb}.

We may now complete the proof of Theorem \ref{thm:maxpl}.
\begin{proof}[Proof of Theorem \ref{thm:maxpl}]
Given a separating triangle $T$ of $G$, we call $\inte(T)$ the subgraph of $G$ induced by the vertices inside and on $T$. Let $G$ be a polyhedron of maximal face length $4$, such that $\fcg(G)$ is maximal planar. By Proposition \ref{p:3sq}, and by finiteness of $G$, we may find a subgraph of $G$ as in Figures \ref{f:tr1} or \ref{f:tr2}, where $a,b,c$ is a separating triangle in $G$, and $\inte(a,b,c)$ contains exactly three quadrangular faces of $G$. By Proposition \ref{p:3sq}, in Figure \ref{f:tr1} each of the edges $ab,bc,ca$ lies on a triangular face, and in Figure \ref{f:tr2} each of $bc,ca$ lies on a triangular face. If a triangle $T$ contained inside $a,b,c$ is not a face of $G$, then $\inte(T)$ is a maximal planar graph. For each such $T$, we replace $\inte(T)$ with a triangle. This is the reverse transformation of $\ct_3$. Then, we replace $\inte(a,b,c)$ with a triangle. This is the reverse transformation of either $\ct_1$ or $\ct_2$.

We repeat these steps until there are no more quadrangular faces. The resulting graph is maximal planar, obtained from a triangle via $\ct_3$.

Vice versa, let $G$ be a polyhedron such that $\fcg(G)$ is maximal planar. Applying $\ct_3$ to $G$ has the effect of replacing a triangular face of $\fcg(G)$ with a maximal planar graph. Applying $\ct_1$ to $G$ has the effect of replacing a triangular face of $\fcg(G)$ with a copy of the octahedron. Applying $\ct_2$ to $G$ has the effect of replacing a triangular face of $\fcg(G)$ with a copy of the graph depicted in Figure \ref{f:tr3}. For all three transformations, the resulting graph is still maximal planar. For $\ct_2$, colouring the face in red as in Figure \ref{f:tr2} has the effect of ensuring that subsequent applications of $\ct_1,\ct_2,\ct_3$ do not break the planarity of $\fcg(G)$, since the red face will be a face in the final graph, consistent with Proposition \ref{p:3sq}.
\begin{figure}[ht]
	\centering
	\includegraphics[width=2.5cm]{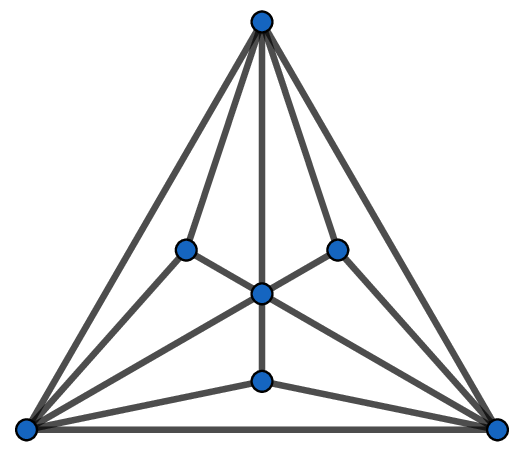}
	\caption{Applying $\ct_2$ to $G$ replaces a triangular face of $\fcg(G)$ with the depicted triangulation.}
	\label{f:tr3}
\end{figure}

To prove the final statement of the present theorem, let $G$ be a $(p,q)$-polyhedral graph of maximal face length $4$, $f_i$ be the number of $i$-gonal faces, and $q_{i,j}$ the total number of common edges of $i$-gonal and $j$-gonal faces, for $3\leq i\leq j\leq 4$. Supposing that $\fcg(G)$ is maximal planar, it follows from Proposition \ref{p:3sq} that $f_4$ is a multiple of $3$. By \eqref{eq:q44} and \eqref{eq:q34}, $q_{4,4}$ and $q_{3,4}$ are multiples of $3$. Now by \eqref{eq:comb2}, $q_{3,3}$ is a multiple of $3$, thus $q$ is a multiple of $3$.
\end{proof}

We close with a result of independent interest on the connectivity of facecongraphs.
\begin{prop}
	\label{p:K24}
	Let $G$ be a polyhedron satisfying $\od(G^*)\simeq K_2$. Then $\fcg(G)$ is not $3$-connected.
\end{prop}
\begin{proof}
	Call $\varphi_1,\varphi_2$ the only two odd faces of $G$, and $yz$ the edge belonging to both. We define
	\begin{equation}
		\label{eq:gp}
		G'=G-yz.
	\end{equation}
	Hence $G'$ is a plane, $2$-connected, bipartite graph, different from $K_{2,\ell}$ for every $\ell\geq 2$. Recalling Remark \ref{rem:bip}, $\fcg(G')$ has exactly two connected components $G_1',G_2'$, both of which are planar and $2$-connected. We also note that $y,z$ belong to the same component of $\fcg(G')$, say $G_1'$. Now $\fcg(G)$ is obtained from $\fcg(G')$ by adding the edges
	\[yz_1,\ yz_2,\ zy_1,\ zy_2\]
	where for $i=1,2$ the vertices $y_i,y,z,z_i$ are consecutive along $\varphi_i$ (Figure \ref{f:yz}), with $y_i=z_i$ if and only if $\varphi_i$ is a triangular face. Note that $y_1,y_2,z_1,z_2\in V(G_2')$. Hence the graph
	\[\fcg(G)-y-z\]
	is disconnected, so that $\fcg(G)$ is not $3$-connected.
	\begin{figure}[ht]
		\centering
			\includegraphics[width=3.cm]{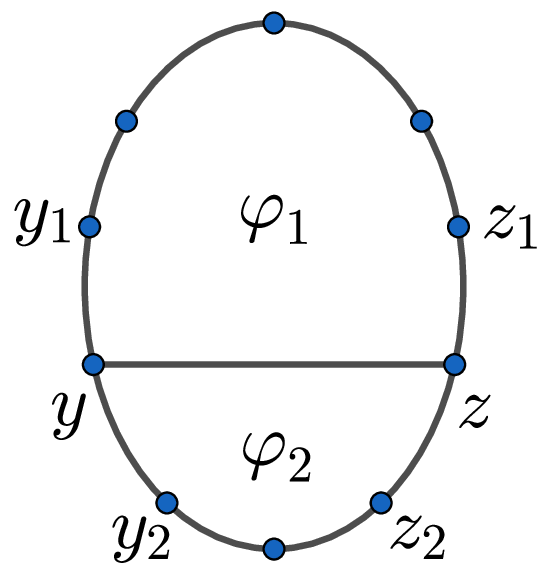}
		\caption{Proposition \ref{p:K24}.}
		\label{f:yz}
	\end{figure}
	
\end{proof}

\paragraph{Closing remarks and future work.} Here we have completely classified the polyhedra $G$ such that $\cg(G)$ is planar (Theorem \ref{thm:0}). Several intriguing questions arise naturally. The first concerns the classification of all graphs $G$ such that $\cg(G)$ is planar/polyhedral. The second concerns the classification of polyhedra and more generally, plane graphs $G$ such that $\fcg(G)$ is planar/polyhedral.
\\
In Theorem \ref{thm:maxpl}, we have characterised and constructed the polyhedra $G$ with $\Delta(G)\leq 4$ such that $\fcg(G)$ is maximal planar. We have also shown in Proposition \ref{p:3456} that if a polyhedron $G$ is such that $\fcg(G)$ is maximal planar, then $\Delta(G)\leq 6$. One would like to classify the remaining polyhedra with facecongraph being maximal planar.
\\
Another interesting question is to classify the plane graphs $G\not\simeq G'$ such that $\fcg(G)\simeq\fcg(G')$.
\\
In this paper, we have introduced the concept of facial common neighbourhood graph. If $G$ is a polyhedron, then $\fcg(G)$ is an intrinsic property of $G$. It would be interesting to employ $\fcg(G)$ to study the fine geometric structure of $G$.

\paragraph{Acknowledgements.}
Riccardo W. Maffucci was partially supported by Programme for Young Researchers `Rita Levi Montalcini' PGR21DPCWZ \textit{Discrete and Probabilistic Methods in Mathematics with Applications}, awarded to Riccardo W. Maffucci.

\clearpage
\addcontentsline{toc}{section}{References}
\bibliographystyle{abbrv}
\bibliography{allgra}

@article{radste,
  title={Vorlesungen {\"u}ber die {T}heorie der {P}olyeder. {S}pringer 1934},
  author={Steinitz, E and Rademacher, H},
  publisher={},
  year={}
}

@article{whit32,
	title={Congruent Graphs and the Connectivity of Graphs},
	author={Whitney, Hassler},
	journal={American Journal of Mathematics},
	volume={54},
	number={1},
	pages={150--168},
	year={1932},
	publisher={JSTOR}
}

@article{tutt61,
	title={A theory of 3-connected graphs},
	author={Tutte, William Thomas},
	journal={Indag. Math},
	volume={23},
	number={441-455},
	pages={8},
	year={1961}
}

@article{dieste,
	title={Graph theory 3rd ed},
	author={Diestel, Reinhard},
	journal={Graduate texts in mathematics},
	volume={173},
	year={2005}
}

@article{brin05,
	title = {Generation of simple quadrangulations of the sphere},
	volume = {305},
	issn = {0012365X},
	url = {https://linkinghub.elsevier.com/retrieve/pii/S0012365X05005170},
	doi = {10.1016/j.disc.2005.10.005},
	language = {en},
	number = {1-3},
	urldate = {2022-10-19},
	journal = {Discrete Mathematics},
	author = {Brinkmann, Gunnar and Greenberg, Sam and Greenhill, Catherine and McKay, Brendan D. and Thomas, Robin and Wollan, Paul},
	month = dec,
	year = {2005},
	pages = {33--54},
}

@article{sonntag2012iterated,
  title={Iterated neighborhood graphs},
  author={Sonntag, Martin and Teichert, Hanns-Martin},
  journal={Discussiones Mathematicae Graph Theory},
  volume={32},
  number={3},
  pages={403--417},
  year={2012},
  publisher={Uniwersytet Zielonog{\'o}rski. Wydzia{\l} Matematyki, Informatyki i Ekonometrii}
}

@article{alwardi2012common,
  title={The common neighborhood graph and its energy},
  author={Alwardi, Anwar and Arsic, Branko and Gutman, Ivan and Soner, Nandappa D},
  journal={Iranian Journal of Mathematical Sciences and Informatics},
  volume={7},
  number={2},
  pages={1--8},
  year={2012},
  publisher={Iranian Journal of Mathematical Sciences and Informatics}
}

@article{schweitzer2013iterated,
  title={Iterated open neighborhood graphs and generalizations},
  author={Schweitzer, Pascal},
  journal={Discrete Applied Mathematics},
  volume={161},
  number={10-11},
  pages={1598--1609},
  year={2013},
  publisher={Elsevier}
}

@article{knor2014wiener,
  title={On {W}iener index of common neighborhood graphs},
  author={Knor, Martin and Luzar, Borut and Skrekovski, Riste and Gutman, Ivan},
  journal={MATCH Commun. Math. Comput. Chem},
  volume={72},
  number={1},
  pages={321--332},
  year={2014}
}

@inproceedings{tian2020exploiting,
  title={Exploiting common neighbor graph for link prediction},
  author={Tian, Hao and Zafarani, Reza},
  booktitle={Proceedings of the 29th ACM International Conference on Information \& Knowledge Management},
  pages={3333--3336},
  year={2020}
}

@article{maffucci2023smallest,
  title={On smallest 3-polytopes of given graph radius},
  author={Maffucci, Riccardo W and Willems, Niels},
  journal={Discrete Mathematics},
  volume={346},
  number={5},
  pages={113322},
  year={2023},
  publisher={Elsevier}
}

@article{maffucci2024classification,
  title={Classification and construction of planar, 3-connected {K}ronecker products},
  author={Maffucci, Riccardo W},
  journal={arXiv:2402.01407},
  year={}
}

@article{gaspoz2024independence,
  title={Independence numbers of polyhedral graphs},
  author={Gaspoz, S{\'e}bastien and Maffucci, Riccardo W},
  journal={Applied Mathematics and Computation},
  volume={462},
  pages={128349},
  year={2024},
  publisher={Elsevier}
}

@article{hollowbread2025generation,
  title={Generation of 3-connected, planar line graphs},
  author={Hollowbread-Smith, Phoebe and Maffucci, Riccardo W},
  journal={Discrete Mathematics},
  volume={348},
  number={2},
  pages={114302},
  year={2025},
  publisher={Elsevier}
}

@article{maffucci2025classification,
  title={Classification of polyhedral graphs by numbers of common neighbours},
  author={Maffucci, Riccardo W},
  journal={arXiv:2508.01349},
  year={}
}

@article{maffucci2025common,
  title={Common neighbours in planar graphs},
  author={Maffucci, Riccardo W},
  journal={arXiv:2511.19251},
  year={}
}

@article{maffucci2025deza,
  title={Deza graphs and regular polyhedra},
  author={Maffucci, Riccardo W},
  journal={The Art of Discrete and Applied Mathematics},
  year={}
}

@article{maffucci2025regularity,
  title={Regularity and separation for {S}ierpi\'nski products of graphs},
  author={Maffucci, Riccardo W},
  journal={arXiv:2506.16864},
  year={}
}
\end{document}